\documentclass[11pt]{article}

\usepackage[latin1]{inputenc}
\usepackage{pdfsync}

\usepackage{amssymb,enumerate,amsmath,amsthm,graphicx,verbatim,color}
\usepackage[all,color]{xy}
\usepackage{pgffor,tikz}
\usepackage[hidelinks]{hyperref}

\usepackage{pict2e}

\makeatletter
\DeclareRobustCommand{\loplus}{\mathbin{\mathpalette\dog@lsemi{+}}}
\DeclareRobustCommand{\lotimes}{\mathbin{\mathpalette\dog@lsemi{\times}}}
\DeclareRobustCommand{\roplus}{\mathbin{\mathpalette\dog@rsemi{+}}}
\DeclareRobustCommand{\rotimes}{\mathbin{\mathpalette\dog@rsemi{\times}}}

\newcommand{\dog@rsemi}[2]{\dog@semi{#1}{#2}{-90,90}}
\newcommand{\dog@lsemi}[2]{\dog@semi{#1}{#2}{270,90}}
\newcommand{\dog@semi}[3]{%
	\begingroup
	\sbox\z@{$\m@th#1#2$}%
	\setlength{\unitlength}{\dimexpr\ht\z@+\dp\z@\relax}%
	\makebox[\wd\z@]{\raisebox{-\dp\z@}{%
			\begin{picture}(1,1)
				\linethickness{\variable@rule{#1}}
				\roundcap
				\put(0.5,0.5){\makebox(0,0){\raisebox{\dp\z@}{$\m@th#1#2$}}}
				\put(0.5,0.5){\arc[#3]{0.5}}
			\end{picture}%
	}}%
	\endgroup
}
\newcommand{\variable@rule}[1]{%
	\fontdimen8  
	\ifx#1\displaystyle\textfont3\else
	\ifx#1\textstyle\textfont3\else
	\ifx#1\scriptstyle\scriptfont3\else
	\scriptscriptfont3\relax
	\fi\fi\fi
}
\makeatother

\newcommand{\mf}[1]{\ensuremath{\mathfrak{#1}}}
\newcommand{\mr}[1]{\ensuremath{\mathrm{#1}}}
\newcommand{\mc}[1]{\ensuremath{\mathcal{#1}}}
\newcommand{\mb}[1]{\ensuremath{\mathbb{#1}}}

\newcommand{\Z}{\mb Z}
\newcommand{\R}{\mb R}
\newcommand{\C}{\mb C}
\newcommand{\Q}{\mb Q}

\newcommand{\p}{\mb P}

\newcommand{\F}{\mc F}

\newtheorem{teo}{Theorem}[section]
\newtheorem{thmm}{Theorem}

\newtheorem{lema}[teo]{Lemma}

\newtheorem{prop}[teo]{Proposition}
\newtheorem{defin}[teo]{Definition}
\newtheorem{cor}[teo]{Corollary}
\newtheorem{obs2}[teo]{Remark}
\newtheorem{q2}[teo]{Question}
\newtheorem{p2}[teo]{Problem}
\newtheorem{ex2}[teo]{Example}
\newenvironment{obs}{\begin{obs2}\rm}{\qed
\end{obs2}}
\newenvironment{ex}{\begin{ex2}\rm}{\qed\end{ex2}}

\newenvironment{dem}{\begin{proof}[Proof]}{\end{proof}}

\definecolor{green}{rgb}{0,0.4,0}
\begin{document}

\title{Birationally integrable vector fields  on complex projective surfaces 
\footnotetext{This work has been partially funded by the Ministry of Science, Innovation and Universities of Spain through the grant PID2021-125625NB-I00  and by the Agency for Management of University and Research Grants of Catalonia through the grant 2021SGR01015.
}
} 
\author{David Mar\'{\i}n and Marcel Nicolau\\*[.1truecm]
{\small \textsl{Departament de Matem{\`a}tiques, Edifici Cc,
Universitat Aut{\`o}noma de Barcelona,}}\\*[-.05truecm]
{\small\textsl{08193 Cerdanyola del Vall\`es (Barcelona), Spain}}
}
\maketitle

\begin{abstract}
A rational vector field on a complex projective smooth surface $S$ is said to be birationally integrable if it generates,  by integration, a one-parameter subgroup of the group $\operatorname{Bir}(S)$ of birational transformations of $S$. We prove that every birationally integrable vector field is regularizable, i.e. birationally conjugated to a holomorphic vector field. Next, we extend this result to any finite-dimensional Lie algebra $\mf g$ of birationally integrable vector fields. This implies that $\mf g$ is naturally included into the Lie algebra of an algebraic subgroup of $\operatorname{Bir}(S)$.
Moreover, we obtain a complete birational classification of birationally integrable Lie algebras that are of dimension two or semisimple, exhibiting holomorphic normal forms of them. We also characterize those birationally integrable algebras of rational vector fields that are maximal.
\end{abstract}

\section{Introduction}

Let $M$ be a smooth complex projective variety. The group $\mr{Aut}(M)$ of automorphisms of $M$ is a (finite dimensional) complex Lie group whose Lie algebra is the vector space $\mf{aut}(M)$ of holomorphic vector fields on $M$. 
Naturally associated with $M$ are also the group $\mr{Bir}(M)$ of birational transformations of $M$ and the Lie 
algebra $\mf X_{\mr{rat}}(M)$ of rational vector fields on $M$. Note that $\mf X_{\mr{rat}}(M)$ is not an infinitesimal counterpart of $\mr{Bir}(M)$ since a general
rational vector field does not generate a subgroup of $\mr{Bir}(M)$. In many cases the group $\mr{Bir}(M)$ coincides with $\mr{Aut}(M)$ but for some particular manifolds, such as the complex projective space $\p^n$ (with $n>1$) or the ruled surfaces, the group $\mr{Bir}(M)$ has infinite dimension. In contrast, $\mf X_{\mr{rat}}(M)$ is always an infinite dimensional vector space.

The study of the structure and the elements of the group $\mr{Bir}(M)$, and specially of the Cremona group $\mr{Bir}(\p^2)$, is a classical subject that was initiated by E.~de~Jonqui\`eres, L.~Cremona and M.~Noether at the end of the 19th century.
Nevertheless, the properties of the Cremona group, and more generally of the group $\mr{Bir}(M)$, are not completely well understood and their study continues to be an active field of research. An important source of motivation was the questions raised by I.R.~Shafarevich \cite{Sha} and D.~Mumford \cite{Mum} concerning the possibility of providing $\mr{Bir}(M)$ with a topological or algebraic structure, and of associating to it a Lie algebra of rational vector fields that generate birational transformations. In this direction we remark the result of J.~Blanc and J.-P.~Furter in \cite{BlaFur} showing that the $n$-dimensional Cremona group $\mr{Bir}(\p^n)$ (for $n\geq2$) does not admit a structure of algebraic group of infinite dimension. 
Algebraic subgroups of $\mr{Bir}(M)$ have also been a topic of interest for a long time, since the first results due to F.~Enriques \cite{Enriques}. We cite the recent work of J.~Blanc \cite{Blanc} who obtained the complete classification of the algebraic subgroups of $\mr{Bir}(\p^2)$ and of P.~Fong \cite{Fong} who classifies the maximal connected algebraic subgroups of $\mr{Bir}(S)$ for any surface $S$.

The goal of this article is to pursue the study of $\mr{Bir}(M)$, but from an infinitesimal point of view. Concretely, we are interested in determining those Lie subalgebras of $\mf X_{\mr{rat}}(M)$ that generate, by integration, subgroups of  $\mr{Bir}(M)$. We say that a vector field $X\in \mr{Bir}(M)$ is birationally integrable if it generates a one-parameter subgroup of $\mr{Bir}(M)$ (cf. Definition~\ref{0}). 
 We note that the subset  $\mf X_{\mr{bir}}(M)\subset \mf X_{\mr{rat}}(M)$ of birationally integrable vector fields
 is not a Lie subalgebra because it is neither closed under the addition nor under the Lie bracket.
Our main theorem concerns finite dimensional Lie subalgebras $\mf g$ of $\mf X_{\mr{rat}}(S)$ where $S$ is a smooth projective surface. It states that if the Lie algebra $\mf g$ is generated by birationally integrable vector fields, then $\mf g$  is regularizable, i.e., there is a minimal smooth surface $S'$ and a birational map $f\colon S'\rightarrow S$ such that $f^\ast {\mf g}$ is a Lie algebra of holomorphic vector fields of $S'$ (cf. Definition~\ref{regularizable}). It follows that $\mf g$ is naturally included into the Lie algebra of an algebraic subgroup of $\mr{Bir}(S)$.
Moreover, along the proof of this result we obtain a complete birational classification of birationally integrable Lie algebras that are of dimension two or semisimple.

To prove the main theorem we follow several steps. First we show that each birationally integrable vector field $X\in\mf X_{\mr{bir}}(S)$ is regularizable. 
This is done using the description of foliations on projective surfaces with an infinite group of birational symmetries given by S.~Cantat and C.~Favre \cite{CF} and an explicit criterion of integrability of  vertical rational vector fields on a surface of type $S=C\times \p^1$.  
The second step consists in showing that the semidirect sum of regularizable subalgebras of $\mf X_{\mr{rat}}(S)$ is also regularizable. Its proof uses two tools: Weil's regularization theorem, as it is stated by S.~Cantat in \cite{Cantat}, and the version of Lie's third theoren given by R.S.~Palais in \cite{Palais}. This already implies that solvable subalgebras of $\mf X_{\mr{rat}}(S)$ generated by birationally integrable vector fields are regularizable. At that point we also classify, up to birational equivalence, birationally integrable 2-dimensional subalgebras of $\mf X_{\mr{rat}}(S)$. That classification is the key for the third step, which consists in determining a complete list of normal forms of rationally integrable subalgebras of $\mf X_{\mr{rat}}(S)$ isomorphic to $\mf{sl}_2(\C)$. As each semisimple complex Lie algebra contains copies of $\mf{sl}_2(\C)$ the above classification allows us to show that each semisimple subalgebra of $\mf X_{\mr{rat}}(S)$ of rank greater than one and generated by birationally integrable vector fields is birationally equivalent to $\mf{aut}(\p^1\times\p^1)$ or $\mf{aut}(\p^2)$. This is the fourth step. Finally, the theorem follows from Levi's decomposition theorem, which states that any complex Lie algebra decomposes as the direct sum of its radical, which is solvable, and a semisimple subalgebra. 

As an application of the above results, and motivated by the study of maximal algebraic subgroups of $\mr{Bir}(S)$ carried out by J.~Blanc \cite{Blanc} and P.~Fong \cite{Fong}, we study maximal birationally integrable algebras of rational vector fields.

Unless otherwise stated, all algebraic varieties considered in this article are smooth. 
\color{black}

\section{Statement of the main results} 

Let $M$ be a smooth complex projective manifold. Its automorphism group $\mr{Aut}(M)$ is a complex Lie group
whose Lie algebra $\mathfrak{aut}(M)$ consists of the holomorphic vector fields on $M$. The connected component $\mr{Aut}_0(M)$ of $\mr{Aut}(M)$ containing the identity is an algebraic group. We denote by $\mr{Bir}(M)$ the group of birational transformations of $M$ and by $\mf X_{\mr{rat}}(M)$ the Lie algebra of rational vector fields of $M$. 
 
For some projective manifolds, these two transformation groups coincide. For instance, if $C$ is a curve then $\mr{Bir}(C)=\mr{Aut}(C)$. 
In the case of a surface $S$, which is the case we are dealing with in this article, the following can be said. If $S$ is a surface of non-negative Kodaira dimension, i.e. $\rm{kod}(S)\geq 0$, then $\mr{Bir}(S)= \mr{Bir}(S_{min})=\mr{Aut}(S_{min})$, where $S_{min}$ is the (unique) minimal model of $S$, and $\mr{Aut}(S)$ is naturally identified to a subgroup of $\mr{Aut}(S_{min})$ (cf. \cite{Hanamura}). If $\mr{kod}(S)=-\infty$ then $\mr{Bir}(S)$ has infinite dimension.

Let $X$ be a rational vector field on 
the projective manifold
$M$ and let $X_\infty$ denote its polar locus. The local flow $\varphi^X$ of $X$ can be multivalued. Nevertheless, $X$ is holomorphic on the Zariski open set $\Omega_X = M\setminus X_\infty$, and there exists an Euclidean neighborhood of $\{0\}\times \Omega_X$ in $\C\times M$ where $\varphi^X$ is well-defined and holomorphic. We call \emph{admissible open subset} of $\C\times M$ any Euclidean open neighborhood of $\{0\}\times V$ in $\C\times M$, where $V\subset M$ contains a non-empty Zariski open subset of~$M$. Thus, the local flow of $X$ is well-defined on a certain admissible open subset,  although it does not necessarily define a one-parameter subgroup of $\mr{Bir}(M)$. These considerations motivate the following definition.

\begin{defin}\label{0}
A rational vector field $X\in\mf X_{\mr{rat}}(M)$ on a complex projective manifold $M$ is called \emph{birationally integrable}  if there is a morphism of groups $\phi:\C\to\mr{Bir}(M)$ and an admissible open subset $U\subset\C\times M$  where the local flow $\varphi^X$ of $X$ is well-defined and fulfills $\varphi^X_t(p)=\phi(t)(p)$ for every $(t,p)\in U$. The set of birationally integrable vector fields on $M$ will be denoted by $\mf X_{\mr{bir}}(M)$.
\end{defin}

This definition precises and extends the one given by D. Cerveau and J. D\'eserti in \cite[\S 2.1.2]{CD} where the authors focus on quadratic birational flows  on the projective plane $\p^2$.

The one-parameter subgroup $\phi:\C\to\mr{Bir}(M)$ associated to a birationally integrable vector field $X$ 
is the flow of $X$ in its domain. One has to note, however, that a one-parameter subgroup of $\mr{Bir}(M)$ is not necessarily the flow of a vector field (cf. Exemple~\ref{no-flow}). Notice also that the action $\C\times M \rightarrow M$ induced by the morphism $\phi$ is not required to be rational. 
\smallskip

It is clear that $\mathfrak{aut}(M)\subset \mf X_{\mr{bir}}(M)$. However, $\mf X_{\mr{bir}}(M)$ is not a Lie algebra, nor even a vector space, since it is neither closed under addition nor under the Lie bracket. A simple example of such a situation is discussed in Example~ \ref{no-algebra}. 
We note that the set of birationally integrable vector fields is well-behaved with respect to the pull-back by birational transformations; that is, if $f\colon M' \rightarrow M$ is birational then $f^\ast \mf X_{\mr{bir}}(M)=\mf X_{\mr{bir}}(M')$.

A natural question is to determine the rational vector fields that are birationally integrable. We do not know of any general criterion, of computational  nature, for the integrability of rational vector fields. However, in this article we obtain a partial answer to that question in the case case where $M$ is a surface of the type $C\times \p^1$, where $C$ is an arbitrary curve and $\p^1$ is the projective line. More precisely, Theorem~\ref{fibration_2} provides an explicit criterion for the birational integrability of a vector field $X$ on $C\times \p^1$ when $X$ is tangent to the rational fibration. Such a characterization is key in the proof of Theorem~\ref{bir-hol}  stated below. From a qualitative point of view, birational integrability can be related to the notion of \emph{regularizability}, that we recall below. The main objective of this article is to show that  these two notions are equivalent in the case of surfaces. 
 
A subgroup $G$ of $\mr{Bir}(M)$ is said to be \emph{regularizable} if there is a smooth complex projective manifold $M'$ and a birational map $f\colon M' \to M$ such that $f^\ast G = \{f^{-1}\circ g \circ f \mid g\in G\} \subset \mr{Aut}(M')$. In that case the elements of the group $G$ are of bounded degree. We extend that notion to the context of rational vector fields as follows.

 \begin{defin}\label{regularizable}
A set of rational vector fields $A\subset\mf X_{\mr{rat}}(M)$  is called \emph{regularizable} if there exists a birational map $f\colon M'\to M$ such that $f^*A\subset\mathfrak{aut}(M')$.
\end{defin}

A regularizable vector field generates, by integration, a one-dimensional subgroup of $\mr{Bir}(M)$. Hence, every regularizable vector field is birationally integrable. If $\dim M=1$ one has $\mf X_{\mr{bir}}(M) =\mf{aut}(M)$, and the two notions coincide. Our first result establishes that, in dimension two, the converse statement is also true.
 
\begin{thmm}\label{bir-hol}
A birationally integrable vector field $X$ on a projective surface $S$ is regularizable. That is, there exist a projective surface $S'$ and a birational map $f\colon S'\rightarrow S$ such that $f^\ast X$ is holomorphic.
Moreover, the surface $S'$ can be chosen to be minimal. 
\end{thmm}

The vector field $X$ in the above theorem defines a holomorphic foliation which is invariant by the flow. The study of holomorphic foliations on surfaces with birational symmetries carried out by S.~Cantat and C.~Favre in \cite{CF} allows us to concentrate the proof of the theorem to the case of vector fields tangent to a rational fibration. Then the statement follows from the integrability criterion given by Theorem~\ref{fibration_2}. The possibility of taking $S'$ minimal is due to the inclusion $\mathfrak{aut}(S)\hookrightarrow  \mathfrak{aut}(S_{min})$ induced by the natural projection $\pi\colon S\rightarrow S_{min}$.

Note that, as it is shown in Example~\ref{fin-dim}, integrable vector field can be regularized in different minimal surfaces.
\smallskip

From the above theorem it follows that a birationally integrable vector field on a projective surface, which is not a torus, preserves a rational or elliptic fibration (cf. Corollary~\ref{preserva-fibracion}). 
We also show that birationally integrable vector fields 
are semicomplete in the sense defined by A.~Guillot and J.-C.~Rebelo in \cite{Guillot-Rebelo} (cf. Proposition~\ref{semi-comp}).
\medskip

Note that if a subset $A\subset\mf X_{\mr{rat}}(S)$ is regularizable 
then $A$ 
is contained in a regularizable Lie subalgebra of $\mf X_{\mr{rat}}(S)$. Hence, our next goal is to extend Theorem~\ref{bir-hol} to Lie algebras in order to determine, for a given projective surface $S$, which Lie subalgebras $\mf g \subset\mf X_{\mr{rat}}(S)$ are regularizable. The obvious necessary conditions are that the Lie algebra $\mf g$ has finite dimension and that each element $X$ of $\mf g$ is regularizable or, what is the same, birationally integrable. 
We show that these two conditions are also sufficient to ensure that $\mf g$ is regularizable,
the difficulty lying in obtaining a simultaneous regularization of all the elements of the algebra. 
We note that there exist infinite dimensional Lie subalgebras of $\mf X_{\mr{rat}}(S)$ that are contained in $\mf X_{\mr{bir}}(S)$ although not all the elements of $\mf g$ can be regularized in the same surface (cf. Example~\ref{fin-dim}). These infinite dimensional algebras lead to the existence of regularizable subalgebras that are not contained in a maximal one (cf. Definition~\ref{def-max} and Theorem~\ref{maximal}).

\medskip

Our aim is to prove that a Lie subalgebra $\mf g \subset\mf X_{\mr{rat}}(S)$ is regularizable if it is finite-dimensional and generated by birationally integrable vector fields. For this purpose, we first show two results that hold in arbitrary dimension. Let $M$ be a smooth complex projective manifold and let $\mf g$ be a finite-dimensional Lie subalgebra of $\mf X_{\mr{rat}}(M)$. Under the hypothesis that the Lie subalgebra is generated (as a vector space) by birationally integrable vector fields, Theorem~\ref{XbirM} states that $\mf g$ is completely contained in $\mf X_{\mr{bir}}(M)$.  Such a subalgebra of $\mf X_{\mr{rat}}(M)$ is said to be a \emph{birationally integrable Lie algebra}. The second result concerns the case where $\mf g$ decomposes as the semidirect sum of regularizable subalgebras. Their proof uses Weil's regularization theorem, as stated by S.~Cantat in \cite{Cantat}, and can be formulated as follows.

\begin{thmm}\label{semi}
Let $\mf g$ be a finite dimensional Lie subalgebra of rational vector fields on a projective manifold $M$. Assume that $\mf h\subset\mf g$ is an ideal and $\mf s\subset\mf g$ is a Lie subalgebra fitting in a split exact sequence of Lie algebras $0\to \mf h\to \mf g\to \mf s\to 0$. If $\mf h$ and $\mf s$ are regularizable then $\mf g$ is also regularizable.
\end{thmm}

From these two results we deduce that, for a given projective surface $S$, the solvable subalgebras of $\mf X_{\mr{rat}}(S)$ generated by rationally integrable vector fields are regularizable. Next, we consider the case of semisimple subalgebras. In Theorems~\ref{sl2} and \ref{r>1}, we show that any semisimple Lie subalgebra of $\mf X_{\mr{rat}}(S)$ generated by birationally integrable vector fields is regularizable and necessarily isomorphic to $\mf{sl}_2(\C)$, $\mf{sl}_2(\C)\times \mf{sl}_2(\C)$ or $\mf{sl}_3(\C)$. 

Finally, using Levi's decomposition theorem, which states that any complex Lie algebra decomposes as a semidirect sum of its radical and a semisimple subalgebra, we obtain our main result:

\begin{thmm}\label{reg} 
Let $S$ be a smooth complex projective surface. Every finite dimensional Lie subalgebra of $\mf X_{\mr{rat}}(S)$ generated, as a vector space, by birationally integrable vector fields is regularizable.
\end{thmm}

A consequence of Theorem~\ref{reg}, in the spirit of Lie's third theorem, is that a finite-dimensional birationally integrable Lie subalgebra $\mf g$ of $\mf X_{\mr{rat}}(S)$ generates, by integration, a subgroup $G$ of $\mr{Bir}(S)$. That group is not necessarily closed but it is a Lie subgroup of an algebraic subgroup of $\mr{Bir}(S)$ (cf.  Definition~\ref{Lie-type} and Corollary~\ref{algebraic_subgroup}). In fact $\mf g$ is birationally equivalent to the Lie algebra of a Lie subgroup of $\operatorname{Aut}(S')$ for a certain minimal surface $S'$.

\smallskip

In addition to the above general results, we also obtain holomorphic normal forms of birationally integrable vector fields and algebras, when we classify them under birational equivalence. Besides being interesting in themselves, these normal forms turn out to be a useful tool throughout the procedure followed to complete the proof of Theorem~\ref{reg}. Specifically, normal forms of $2$-dimensional regularizable Lie algebras are the key ingredient in the study of the algebras isomorphic to $\mf{sl}_2(\C)$. And the normal forms of $\mf{sl}_2(\C)$ are the starting point in the determination of semisimple regularizable Lie algebras of higher rank. In this direction, Theorems~\ref{X1} and~\ref{pd} provide a precise description of the birational classes of holomorphic vector fields on projective surfaces.
And Theorem~\ref{2dim} collects the holomorphic normal forms of the $2$-dimensional regularizable Lie algebras. 
\smallskip

As another application of the above results, and motivated by the study of maximal algebraic subgroups of $\mr{Bir}(S)$ carried out by J.~Blanc \cite{Blanc} and P.~Fong \cite{Fong}, we study maximal birationally integrable algebras of rational vector fields.
We say that a birationally integrable Lie algebra $\mf g$ of rational vector fields on a surface $S$ is maximal if there is no birationally integrable Lie algebra $\mf g'\subset \mf X_{\mr{rat}}(S)$ which strictly contains $\mf g$. We have seen that the algebra $\mf g$ is birationally equivalent to a Lie subalgebra of  $\mathfrak{aut}(S')$ for a certain minimal surface $S'$. Therefore it is sufficient to decide, for a given minimal surface $S$, whether the Lie algebra of holomorphic vector fields  $\mathfrak{aut}(S)$ is maximal in the above sense. The result that we obtain is the following. (A more complete statement is given in Theorem~\ref{maximal}.)

\begin{thmm}\label{maximal_0}
 Let $S$ be a minimal projective surface. The Lie algebra $\mathfrak{aut}(S)$ is not maximal (as birationally integrable Lie subalgebra of $\mf X_{\mr{rat}}(S)$) if and only if 
 $\mathfrak{aut}(S)$ is tangent to a rational fibration and $S$ is not biholomorphic to a product $C\times\p^1$.
\end{thmm}

\section{Holomorphic vector fields on projective surfaces}\label{Shol}

Our first result, Theorem~\ref{bir-hol}, states that every birationally integrable vector field on a projective surface is birationally conjugated to a holomorphic vector field on a suitable surface. For this reason, we dedicate this section to recalling which projective surfaces have non-trivial holomorphic vector fields and to giving a detailed description of these vector fields in the way that will be used in the following sections.

In fact, we will focus this description to the case of minimal surfaces, i.e. surfaces $S$ without rational curves of self-intersection $-1$. The reason is that any such a curve $C$ is holomorphically rigid and any holomorphic vector field $X$ on $S$ is necessarily tangent to it. The blow-down of the curve $C$ induces a holomorphic vector field $\bar X$ on a smooth surface $S'$.
Since a minimal model $S_{min}$ of $S$ is obtained by a finite sequence of such kind of blow-downs, the natural projection $\pi:S \to S_{min}$ induces an injective morphism of Lie algebras $\mf{aut}(S)\hookrightarrow \mf{aut}(S_{min})$. Therefore, under birational equivalence, every holomorphic vector field admits a holomorphic representative living on a minimal surface.

We recall that, according to Levi decomposition theorem, each complex Lie algebra $\mf g$ fits into a split exact sequence $0\rightarrow \operatorname{rad}(\mf g) \rightarrow \mf g \rightarrow \mf g/ \operatorname{rad}(\mf g) \rightarrow 0$, where $\operatorname{rad}(\mf g)$ is the radical of $\mf g$, that is, the maximal solvable ideal. Moreover, there is a semisimple Lie subalgebra $\mf s\subset\mf g$ which is sent isomorphically onto $\mf g/ \operatorname{rad}(\mf g)$ via the quotient map $\mf g \rightarrow\mf g/ \operatorname{rad}(\mf g)$. Then $\mf g$ is said to be the semidirect sum of $\operatorname{rad}(\mf g)$ and $\mf s$, which we denote by $\mf g= \operatorname{rad}(\mf g)\roplus\mf s$. 
We recall also that a Borel subalgebra of $\mf g$ is a maximal solvable subalgebra, which is not necessarily unique.
The Levi decomposition of the algebra $\mf{aut}(S)$ plays an important role in the proof of Theorem~\ref{reg}. For this reason we indicate the corresponding decomposition for each minimal surface $S$.
\medskip

We begin the description of holomorphic vector fields on surfaces by considering the case of projective surfaces whose Kodaira dimension, $\mr{kod}(S)$, is non-negative. We recall that, in that case, the surfaces have a unique minimal model. 

Let $S$ be a minimal surface of non-negative Kodaira dimension. According to Enriques' classification, there are the following possibilities:
\begin{itemize}
\item If $\mr{kod}(S)=0$, then $S$ is either a K3 or Enriques surface, and $\mf{aut}(S)=0$, or $S$ is an Abelian surface, and $\dim\mf{aut}(S)=2$, or $S$ is a bielliptic surface, and $\dim\mf{aut}(S)=1$.
\item If $\mr{kod}(S)=1$, then $S$ is an elliptic surface and $\dim\mf{aut}(S)\in\{0,1\}$. 
Moreover, when $\mf{aut}(S)\neq 0$, $S$ is a Seifert elliptic fibration and $\mf{aut}(S)$ is tangent to the elliptic fibration (cf. \cite{Lie2}).
\item If $\mr{kod}(S)=2$, then $S$ is of general type and $\mf{aut}(S)=0$. 
\end{itemize} 
We recall that a surface $S$ is said to be a Seifert elliptic fibration if $S= \tilde S/\Gamma$ where $\tilde S$ is a principal $T$-bundle, with $T$ an elliptic curve, and $\Gamma$ is a finite group whose action commutes with that of~$T$. If $S$ has non-negative Kodaira dimension and $\mf{aut}(S)\neq 0$  then $\tilde S = C\times T$ and it is an elliptic surface, if $\mr{kod}(S)=1$, or a bielliptic surface if $\mr{kod}(S)=0$.  Moreover, in that case, $\Gamma\subset T$ and $\mf{aut}(S)\cong\mf{aut}(T)$ (cf. \cite{Fong}).
Examples of Seifert elliptic fibrations with $\mr{kod}(S)=-\infty$ are discussed in Remark~\ref{el-fib}.

\begin{obs}
From the above description we see that, if $S$ has non-negative Kodaira dimension and non-trivial holomorphic vector fields, then $\mf{aut}(S)$ is an Abelian Lie algebra. Thus, $\mf{aut}(S)$ coincides with its radical. This algebra has dimension $2$ if $S$ is an Abelian surface, and dimension $1$ if $S$ is elliptic or bielliptic. In these two last cases $\mf{aut}(S)$ is tangent to an elliptic fibration of $S$. Moreover, all the vector fields $X$ of $S$ are non-singular (except if $X\equiv 0$).  
If $X$ is a holomorphic vector field on a surface $S$, and $\pi\colon S'\to S$ is the blow-up of a point of $p\in S$, then $\pi^*X$ is holomorphic if and only if $p$ is a singular point of $X$. We deduce that surfaces of non-negative Kodaira dimension having non-trivial holomorphic vector fields are necessarily minimal. 
\end{obs}

Let us now consider the class of minimal surfaces of negative Kodaira dimension.
We discuss separately $\p^2$, Hirzebruch surfaces and  non-rational ruled surfaces. 

We recall first that the Lie algebra $\mf{aut}(\p^2)$ is isomorphic to the simple Lie algebra $\mf{sl}_3(\C)$. If we fix an affine chart $(x,y)$ of $\p^2$, there is an isomorphism $\Phi:\mf{sl}_3(\C)\to\mf{aut}(\p^2)$ given by 
\begin{equation}\label{Phi}
\Phi(A)= (\partial_x,\partial_y,-x\partial_x-y\partial_y)A\left(\begin{array}{c}x\\ y \\ 1\end{array}\right).
\end{equation} 
Using that isomorphism, we see that
\begin{equation}\label{XP2}
 \mf{aut}(\p^2)=\big\langle \partial_x,\partial_y,x\partial_x,y\partial_y,y\partial_x,x\partial_y,x(x\partial_x+y\partial_y),y(x\partial_x+y\partial_y)\big\rangle.
 \end{equation}
We note that the pull-back of a vector field $\Phi(A)$ by the automorphism of $\p^2$ induced by $C\in\mr{SL}(3,\C)$ corresponds to  $\Phi(C^{-1}AC)$. Using the Jordan normal forms of the matrices of $\mf{sl}_3(\C)$ and the above observation, we obtain:
\begin{prop}\label{campos_P2}
Each vector field in $\p^2$ is holomorphically conjugated to one of the following:
\begin{enumerate}[ (a)]
 \item $T=\partial_y$,
 \item $N=\partial_x+x\partial_y$,
 \item $J=\partial_x+y\partial_y$,
 \item $H_\gamma=x\partial_x+ \gamma y\partial_y,\quad \text{with}\quad \gamma\in\C$.
\end{enumerate}
\end{prop}

\begin{obs}\label{JordanP2}
More precisely, we can see that $T= \Phi\left(\begin{array}{ccc}0 & 0 & 0\\ 0 & 0 &1\\ 0 & 0 & 0\end{array}\right)$, $N=\Phi\left(\begin{array}{ccc}0 & 0 & 1\\ 1 & 0 &0\\ 0 & 0 & 0\end{array}\right)$, $J=\Phi\left(\begin{array}{ccc}-\frac{1}{3} & 0 & 1\\ 0 & \frac{2}{3} &0\\ 0 & 0 & -\frac{1}{3}\end{array}\right)$ and $H_\gamma=\Phi\left(\begin{array}{ccc}\frac{2-\gamma}{3} & 0 & 0\\ 0 & \frac{2\gamma-1}{3} &0\\ 0 & 0 & -\frac{\gamma+1}{3}\end{array}\right)$.
We note also that $N$ is holomorphically conjugated to $\partial_y+y\partial_x=\Phi\left(\begin{array}{ccc}0 & 1 & 0\\ 0 & 0 &1\\ 0 & 0 & 0\end{array}\right)$.
\end{obs}

\begin{obs}\label{fibP2}
For any holomorphic vector field $X$ on $\p^2$ the following conditions are equivalent:
\begin{enumerate}[ (a)]
\item $X$ is holomorphically conjugated to a constant multiple of $T$, $N$ or $H_\gamma$ with $\gamma\in\Q$,
\item $X$ is tangent to a rational fibration,
\item the subgroup $\exp\C X\subset\mr{PSL}_3(\C)$ is Zariski closed.
\end{enumerate}
Furthermore, $\mr{Lie}\,\overline{\exp\C J}=\langle \partial_x,y\partial_y\rangle$ and $\mr{Lie}\,\overline{\exp\C H_\gamma}=\langle x\partial_x,y\partial_y\rangle$ if $\gamma\in\C\setminus\Q$.  Here, $\overline{\exp\C X}$ denote the Zariski closure in $\mr{Aut}(\p^2)$ of the flow associated to the vector field $X$.
Note that each of these two Abelian Lie algebras preserve the two rational fibrations defined by $dx$ and $dy$ respectively.
\end{obs}

\medskip

We consider now the case of rational ruled surfaces, that is, the Hirzebruch surfaces $\mb F_n=\p(\mc O_{\p^1}(n)\oplus \mc O_{\p^1})$ where $n\ge 0$. For the sake of completeness, we include in our discussion the Hirzebruch surface $\mb F_1$, which is the blow-up of a point in $\p^2$ and therefore is not minimal.
We denote by $\pi:\mb F_n\to\p^1$ the natural projection and we consider an affine chart $(x,y)$ on $\mb F_n$, where $x$ is an affine coordinate in the basis and $y$ is an affine coordinate in the fiber such that $\{y=\infty\}$ has self-intersection $-n$. Coordinates $(x,y)$ of $\mb F_n$ with the above properties will be called {\sl standard}.
In standard coordinates, the spaces of holomorphic vector fields of Hirzebruch surfaces can be described as follows (cf. for instance  \cite[\S3.3.2]{Deserti}).

\begin{prop}\label{campos_Fn} The algebra $\mf{aut}(\mb F_n)$ of the Hirzebruch surface $\mb F_n$ is given by
\begin{equation}\label{XF0}
\mf{aut}(\mb F_0)=\mf{aut}(\p^1\times\p^1)=\C_2[x]\partial_x\oplus\C_2[y]\partial_y
\end{equation}
\begin{equation}\label{XFn}
\textstyle\mf{aut}(\mb F_n)=\big\langle \partial_x,x\partial_x+\frac{n}{2}y\partial_y,x^2\partial_x+nxy\partial_y\big\rangle\oplus\C y\partial_y\oplus\C_n[x]\partial_y,  \text{ if } n>0.
\end{equation}
\end{prop}

\begin{obs}\label{y-quadrado}
Note that  $y^2\partial_y$ is a birationally integrable vector field of $\mb F_n$, but it is not holomorphic if $n\geq 1$ because its flow does not preserve $y=\infty$, the section of self-intersection $-n<0$.
\end{obs}

 \begin{obs}\label{rad}
 We denote by $\mf{aut}^v(\mb F_n)$ the subalgebra of $\mf{aut}(\mb F_n)$ of vector fields tangent to the fibration. That is, $\mf{aut}^v(\mb F_n)=\ker\pi_*$.
 
We deduce from the above proposition that the radical of $\mf{aut}(\mb F_0)$ is trivial, because $\mf{aut}(\p^1\times\p^1)\cong\mf{sl}_2(\C)\oplus\mf{sl}_2(\C)$ is semisimple, and that
 \[
 \operatorname{rad}(\mf{aut}(\mb F_n)) = \mf{aut}^v(\mb F_n)=\C_n[x]\partial_y\oplus  \C y\partial_y
\] 
if $n>0$. Recall that
Blanchard's theorem \cite{B} implies that every vector field $X\in\mf{aut}(\mb F_n)$ projects onto a holomorphic vector field $\bar X=\pi_*X$ on the basis $\p^1$ of the rational fibration $\pi:\mb F_n\to\p^1$.
Therefore, there is a well-defined projection morphism $\pi_\ast\colon \mf{aut}(\mb F_n)\to\mf{aut}(\mb F_n)/\mf{aut}^v(\mb F_n)\cong\mf{aut}(\p^1)$. Moreover, the restriction of $\pi_\ast$ to the Lie subalgebra $\big\langle \partial_x,x\partial_x+\frac{n}{2}y\partial_y,x^2\partial_x+nxy\partial_y\big\rangle\cong\mf{sl}_2(\C)$ is an isomorphism.
In other words, Levi's decomposition of $\mf{aut}(\mb F_n)$, for $n>0$, is given by the semidirect sum of the solvable ideal $\mf{aut}^v(\mb F_n)$ with the simple algebra $\big\langle \partial_x,x\partial_x+\frac{n}{2}y\partial_y,x^2\partial_x+nxy\partial_y\big\rangle$.
 Notice also that, for all $n\ge 0$,  the direct sum
\begin{equation}\label{bn}
\mf B_n:=\C_1[x]\partial_x\oplus\C_n[x]\partial_y\oplus\C y\partial_y
\end{equation}
is a Borel subalgebra of $\mf{aut}(\mb F_n)$, and that the derived series of $\mf B_n$
is 
\begin{equation}\label{B_der}
\mf B_n^{(1)}=\C\partial_x\oplus\C_n[x]\partial_y, \quad \mf B_n^{(2)}=\C_n[x]\partial_y \quad \text{and} \quad \mf B_n^{(3)}=0.
\end{equation}
\end{obs}

Let us note here the following fact which will be useful later.

\begin{lema}\label{solv_S}
Let $S$ be a rational surface and let $\mf B$ be a Borel subalgebra of $\mf{aut}(S)$. Every solvable Lie subalgebra $\mf g$ of $\mf{aut}(S)$ is holomorphically conjugated to a subalgebra of $\mf B$, that is, there exists an automorphism $f\in\operatorname{Aut}(S)$ such that $f^\ast\mf g\subset \mf B$.
\end{lema}

\begin{proof}
It suffices to prove the statement in the case that $\mf g = \mf B'$ is another Borel subalgebra. Note that $\operatorname{Aut}(S)$ is a complex linear algebraic group (because of Chevalley's theorem), therefore $\mf B'= \operatorname{Lie}(B')$ and $\mf B= \operatorname{Lie}(B)$, where $B'$ and $B$ are Borel subgroups of $\operatorname{Aut}(S)$ (cf. \cite{Borel}).
As any two Borel subgroups are conjugated, there exists $f\in \operatorname{Aut}(S)$ that conjugates $B'$ and $B$ and therefore $f^\ast \mf B' = \mf B$.
\end{proof}

\begin{obs}\label{rad2}
The choice of standard coordinates on the Hirzebruch surfaces $\mb F_n$ induces birational maps $\Phi\colon\mb F_n \dashrightarrow \mb F_{n+1}$ such that $\Phi_\ast (\mf B_n)$ is included into $\mf B_{n+1}$. As an abuse of notation, we will write  $\mf B_n\subset\mf B_{n+1}$. Notice however that this inclusion does not extend to an inclusion of  $\mf{aut}(\mb F_n)$ into $\mf{aut}(\mb F_{n+1})$.
In a similar way, the identification of the standard coordinates $(x, y)$ on $\mb F_n$
with an affine coordinate system on the projective plane $\p^2$ induces a birational map $\Psi\colon \mb F_n \dashrightarrow \p^2$. The pull-back by $\Psi$ of the vector fields $T, N, J$ and $H_\gamma$ of $\p^2$ considered in Proposition~\ref{campos_P2} are well defined holomorphic vector fields on the Hirzebruch  surface $\mb F_n$. 
As an abuse of notation, we will consider that these vector fields are defined on $\p^2$ or on $\mb F_n$, depending on the context.
\end{obs}

\medskip
Finally, we consider  non-rational ruled surfaces, that is, ruled surfaces $S=\mb PE$ over projective curves $C$ of positive genus. The case of ruled surfaces over elliptic curves was studied by T.~Suwa in \cite[\S3]{Suwa} and the case of ruled surfaces over higher genus curves was studied by M.~Maruyama in \cite[Theorem~1]{Maruyama}. As above, we denote by $\mf{aut}^v(S)$ the space of holomorphic vector fields on the ruled surface $S$ tangent to the rational fibration $S\to C$. Using the birational identification of $S$ with $C\times\p^1$, the coordinates $(x,y)$ on this product surface, where $x$ is a coordinate on $C$ and $y$ is an affine coordinate of $\p^1$, will be considered a coordinate system on $S$. In the case where $C$ is an elliptic curve $C= \C/\Lambda$, the $x$ coordinate will be the linear coordinate of $\C$.

If $E$ is a rank two vector bundle over a curve $C$ we denote by $N(E)$ the invariant of the bundle $E$ defined as 
\begin{align*}
N(E)&=2\max\{\deg L: L\subset E \text{ line subbundle}\}-\deg E\\
&=-\min\{\sigma(C)^2:\sigma \text{ section of }\pi:\p E\to C\},
\end{align*}
(cf. \cite[Proposition 5.12]{Friedman}). A line subbundle $L\subset E$ of maximal degree is called \emph{maximal}.

Let $C= \C/\Lambda$ be a given elliptic curve. As usual, $\wp(x)$ stands for the Weierstrass  elliptic function with periods $\Lambda\subset\C$. We denote by $\wp_\tau$ the well-defined rational function on $C$ given by 
\begin{equation}\label{wptau}
\wp_\tau(x):=\frac{1}{2}\frac{\wp'(\tilde\tau)+\wp'(x)}{\wp(\tilde\tau)-\wp(x)},
 \end{equation} 
 where $\tilde\tau\in\C$ projects onto $\tau\in C$. We denote by $\mr{o}$ the class of $0\in\C$ in $C=\C/\Lambda$.

\smallskip 

Then, the results of Suwa and Maruyama can be stated as follows.

\begin{teo}[Suwa]\label{Xhol-ruled_1}
Let $S=\p E\to C$ be a ruled surface over an elliptic curve $C$.
Then, up to biholomorphism, we are in one of the following cases:
\begin{enumerate}[ (1)]
\item $S=\p A_{1}$, where $A_{1}$ is the stable undecomposable rank $2$ vector bundle over $C$ (whose invariant is $N(A_{1})=-1$) and $\mf{aut}(S)=\C\big(2\partial_x+(y^2-3\wp(x))\partial_y\big)$,
\item $S=\p A_0$, where $A_0$ is the semi-stable undecomposable rank $2$ vector bundle over $C$ (whose invariant is $N(A_0)=0$) and $\mf{aut}(S)=\C\big(\partial_x-\wp(x)\partial_y\big)\oplus\C\partial_y$,
\item $S=\p E_\tau$ where $E_\tau=\mc O_C(\tau-\mr{o})\oplus \mc O_C$ with $\tau\in C\setminus\{\mr{o}\}$
and $\mf{aut}(S)=\C\big(\partial_x-\wp_\tau(x)y\partial_y\big)\oplus\C y\partial_y$,
\item $S=\p(\mc O_C(np)\oplus\mc O_C)$ with $n>0$ and $\mf{aut}(S)=\mf{aut}^v(S)=\C y \partial_y\oplus H^0(C,\mc O_C(np))\partial_y$,
\item $S=C\times\p^1$ and $\mf{aut}(S)=\C\partial_x\oplus\C_2[y]\partial_y$.
\end{enumerate}
\end{teo}

The presentation of case 3) in the above theorem requires a more detailed explanation. According to \cite[\S3]{Suwa},
if $E=\mc O_C(p_1-p_0)\oplus \mc O_C$ with $p_1\neq p_0$, then 
$\mf{aut}(\p E)=\big\langle\partial_x+p(x)y\partial_y,y\partial_y\big\rangle$
with $p(x)=\zeta(x-p_1)-\zeta(x-p_0)$, where $\zeta=-\int\wp$ is the Weierstrass zeta function. 
 According to \cite[Lemma 2]{Suwa}, $\p(\mc O_C(p_1-p_0)\oplus\mc O_C)\cong\p(\mc O_C(p_1'-p_0')\oplus\mc O_C)$ if and only if there is $f\in\mr{Aut}(C)$ such that $f(p_i)=p_i'$, for $i=0,1$.
Since $\mr{Aut}(C)$ acts transitively on $C$, we can always take  $p_0=\mr{o}\in C$. In this case, due to the addition formula of $\zeta$ (see for instance \cite[\S18.4.3]{Abramowitz-Stegun}), we have
\begin{equation}\label{wp}
p(x)=\zeta(x-p_1)-\zeta(x)=\zeta(-p_1)+\frac{1}{2}\frac{\wp'(x)-\wp'(-p_1)}{\wp(x)-\wp(-p_1)}=-\zeta(p_1)+\frac{1}{2}\frac{\wp'(x)+\wp'(p_1)}{\wp(x)-\wp(p_1)},
\end{equation}
so that $\big\langle\partial_x+p(x)y\partial_y,y\partial_y\big\rangle=\big\langle\partial_x-\wp_\tau(x)y\partial_y), y\partial_y \big\rangle$, according to (\ref{wptau}).

\begin{teo}[Maruyama]\label{Xhol-ruled_2}
Let $S=\p E\to C$ be a ruled surface over a curve $C$ of genus $g(C)>1$.
Then $\mf{aut}(S)=\mf{aut}^v(S)$ and 
\begin{enumerate}[ (1)]
\item if $S=\p E$ with invariant $N(E)<0$, then $\mf{aut}^v(S)=0$,
\item if $S=\p E$ with invariant $N(E)\ge 0$, $E$ is indecomposable and $L\subset E$ is the  unique maximal line subbundle of $E$, then $\mf{aut}^v(S)=H^0(C,(\det E)^{-1}\otimes L^2)\partial_y$,
\item if $S=\p E$ with $E= L\oplus L'$, $\deg L\ge \deg L'$ and $L\not\cong L'$ then $\mf{aut}^v(S)= H^0(C,(\det E)^{-1}\otimes L^2)\partial_y\oplus\C y\partial_y$,
\item if $S=C\times\p^1$, then $\mf{aut}^v(S)=\C_2[y]\partial_y$.
\end{enumerate}
\end{teo}

\begin{obs}
If the ruled surface $S=\p E$ is different from $C\times\p^1$, then $\mf{aut}(\p E)$ is solvable. If $C$ is a curve of genus $g(C)>1$, then $\mf{aut}(C\times\p^1)\cong \mf{sl}_2(\C)$ is simple. Finally, if $C$ is an elliptic curve, then the Levi decomposition of $\mf{aut}(C\times\p^1)$ is just the direct sum of the its radical $\mf{aut}(C)$ and the simple subalgebra $\mf{aut}(\p^1)\cong\mf{sl}_2(\C)$.
\end{obs}

\section{Birationally integrable vector fields} 

In this section we prove Theorem~\ref{bir-hol} showing that each birationally integrable vector field on a surface is regularizable.

\smallskip
We begin by discussing some examples of birationally integrable vector fields. The first one arises from a geometric construction and illustrates the behavior that the flow generated by such a vector field can have.

\begin{ex}	We choose two different points $p_\pm\in\p^2$ and consider the spaces $\check{p}_\pm\cong\p^1$ of lines through $p_\pm$. Given two holomorphic flows $\varphi_t^\pm$ on $\check p_\pm$, we can construct, for each $t\in\C$, the rational map
	\[\phi(t)(p)=\varphi_t^+(p\vee p_+)\cap\varphi_t^-(p\vee p_-),\] where $p\vee p_\pm$ is the straight line passing through $p$ and $p_\pm$. The map $\phi(t)$ is birational, has two fixed indeterminacy points $p_\pm$ and a movable indeterminacy point 
	\[I_t=\varphi_{-t}^-(p_-\vee p_+)\cap\varphi^+_{-t}(p_-\vee p_+).\] 	We take affine coordinates $(x,y)$ on $\p^2$ such that $p_\pm=(\pm 1,0)$, and an affine coordinate $m$ on $\p^1$ such that $\check p_\pm=\{y=m(x\mp 1)\}_{m\in\p^1}$. If we choose the flows $\varphi_t^-(m)=m+t$, $\varphi^+_t(m)=e^tm+(e^t-1)$, then the map $\phi:\C\to\mr{Bir}(\p^2)$ is given by $\phi(t)(x,y) = \big(\phi_1(t,x,y), \phi_2(t,x,y)\big)$ where
	\[
	\begin{aligned}
		\phi_1(t,x,y)&=
		-\frac{{\mathrm e}^{t} x^{2}+{\mathrm e}^{t} y x +t \,x^{2}-{\mathrm e}^{t} y -x^{2}+x y -{\mathrm e}^{t}-t +y +1}{{\mathrm e}^{t} x^{2}+{\mathrm e}^{t} y x -t \,x^{2}-{\mathrm e}^{t} y -x^{2}-x y -{\mathrm e}^{t}+t -y +1}, \\[2mm]
		\phi_2(t,x,y)&=
		-\frac{2 \left(t x -t +y \right) \left({\mathrm e}^{t} x +{\mathrm e}^{t} y +{\mathrm e}^{t}-x -1\right)}{{\mathrm e}^{t} x^{2}+{\mathrm e}^{t} y x -t \,x^{2}-{\mathrm e}^{t} y -x^{2}-x y -{\mathrm e}^{t}+t -y +1}.
	\end{aligned}	
	\] 
	It is easy to check that $\phi$ is a morphism of groups. Indeed,
	it is the one-parameter group of the rational vector field
	\[	X=\frac{d}{dt}\Big|_{t=0}\phi(t)(x,y)=\frac{(x^2-1)(y+2)}{2y}\partial_x+\Big(\frac{y}{2}+2x+\frac{xy}{2}\Big)\partial_y,	\] 
 which is birationally integrable by contruction.
 As every birationally vector field on a surface with negative Kodaira dimension, it preserves a fibration (cf. Corollary~\ref{preserva-fibracion}). In our case $X$ preserves the two rational fibrations given by the  pencils of lines through $p_-$ and $p_+$.

 It can be checked that the transcendental function $F(x,y)=\frac{(x+1)\mr{e}^{\frac{y}{x-1}}}{x+y+1}$ is a first integral of $X$ and the movable indeterminacy point $I_t=\big(\frac{\mr{e}^t(t+1)-1}{\mr{e}^t(t-1)+1},\frac{-2t(\mr{e}^t-1)}{\mr{e}^t(t-1)+1}\big)$ of $\phi(t)$ describes the Zariski dense subset of~$\p^2$
	\[
	F^{-1}(1)=\big\{(x,y)\in\p^2\mid (x+1)\mr{e}^{\frac{y}{x-1}}=x+y+1
	\big\}.
	\] 	
	We observe however that the flow $\phi$ is defined on the complement of the image of the rational map $\sigma:\C^2\to\C\times\p^2$ defined by 
	
	\[
	\sigma(t,\tau)=\Big(t,\frac{\tau(t+1)-1}{\tau(t-1)+1},\frac{-2t(\tau-1)}{\tau(t-1)+1}\Big),
	\]
	which is contained in  a proper Zariski closed subset of $\C\times\p^2$.
	
The geometric behavior of the birational transformation $\phi(t):\p^2\to\p^2$ is illustrated in Figure~\ref{phi}: it contracts the polar locus $X_\infty$ of $X$, that is, the line $\ell_0=\{y=0\}=p_-\vee p_+$,  but also the lines 
	$\ell_t^-=\{y+tx-t=0\}=p_+\vee I_t$ and $\ell_t^+=\{(\mr{e}^t-1)x+\mr{e}^ty+\mr{e}^t-1=0\}=p_-\vee I_t$.
	That is, the  lines $\ell_0$, $\ell_t^-$ and $\ell_t^+$ are respectively contracted by $\phi(t)$ onto the indeterminacy points $I_{-t}$, $p_-$ and $p_+$ of $\phi(-t)$. The map $\phi(t)$ blows-up its indeterminacy points into the lines contracted by the inverse map $\phi(-t)$.
	
	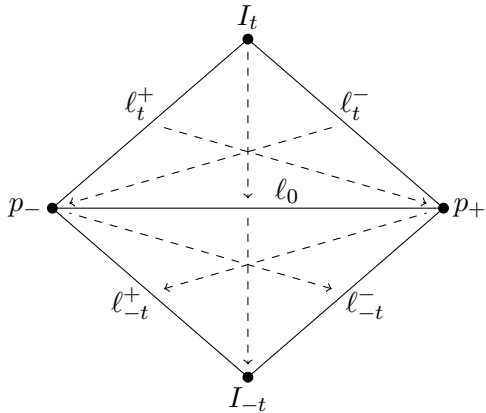
\begin{figure}[t]
		\begin{center}
			\begin{tikzpicture}[scale=1.3]
				\draw [fill,black] (-2,0) circle  [radius=0.05];
				\draw [fill,black] (2,0) circle  [radius=0.05];
				\draw [fill,black] (0,1.73) circle  [radius=0.05];
				\draw [fill,black] (0,-1.73) circle  [radius=0.05];
				\node at (-2,0) [left] {$p_-$};
				\node at (2,0) [right] {$p_+$};
				\node at (0,1.73) [above] {$I_t$};
				\node at (0,-1.73) [below] {$I_{-t}$};
				\node at (-1.1,.8) [above] {$\ell^+_t$};
				\node at (1.1,.8) [above] {$\ell^-_t$};
				\node at (-1.2,-.7) [below] {$\ell^+_{-t}$};
				\node at (1.2,-.7) [below] {$\ell^-_{-t}$};
				\node at (.4,-0.05) [above] {$\ell_0$};
				\begin{scope}[yscale=.5]
					\draw (-2,0) to (2,0) to (0,1.73*2) to (-2,0)  to (0,-1.73*2) to (2,0);
					\draw [dashed,->] (0,1.6*2) to (0,.1*2);
					\begin{scope}[xshift=2cm,rotate=60]
						\draw [dashed,->] (0,1.65*2) to (0,.1*2);
					\end{scope}
					\begin{scope}[xshift=-2cm,rotate=-60]
						\draw [dashed,->] (0,1.65*2) to (0,.1*2);
					\end{scope}
					\begin{scope}[cm={1,0,0,-1,(0,0)},xshift=2cm,rotate=60]
						\draw [dashed,<-] (0,1.65*2) to (0,.1*2);
					\end{scope}
					\begin{scope}[cm={1,0,0,-1,(0,0)},xshift=-2cm,rotate=-60]
						\draw [dashed,<-] (0,1.65*2) to (0,.1*2);
					\end{scope}
					\draw [dashed,->] (0,-.1*2) to (0,-1.6*2);
				\end{scope}
			\end{tikzpicture}
		\end{center}
		\caption{Geometric behavior of the birational transformation $\phi(t)$.}\label{phi}
	\end{figure}
	
	\color{black}
	Note that $X$ is regularizable, since it is the pull-back of the holomorphic vector field $(x+1)\partial_x+\partial_y$ on $\p^2$ by the birational map  $f(x,y)=\big(\frac{y}{x+1},\frac{y}{x-1}\big)$. 
The regularizing map $f$ contracts only the polar locus $X_\infty=\{y=0\}$  and blows up its indeterminacy points $p_-$ and $p_+$ getting 
$f(p_+)=\{x=0\}$ and $f(p_-)=\{y=0\}$.
\end{ex}

This example shows that the restriction of a birationally integrable vector field to the complementary of its polar set, which is always invariant (cf. Lemma~\ref{lema_2}),  is not necessarily complete. 
A related notion is the semicomplete\-ness of vector fields  introduced by E.~Ghys and J.-C.~Rebelo in \cite{Rebelo,Ghys-Rebelo}, meaning that the associated flow is univaluated.  More precisely, a \emph{holomorphic} vector field $X$ on a complex manifold $U$ is said to be semicomplete (in $U$) if there exists an open set $\Omega\subset\C\times U$ where the flow $\varphi$ of $X$ is defined and the following condition is fulfilled: for every $p\in U$ and any sequence $(T_n)_n\subset \C$, such that $(p,T_n)$ belong to $\Omega$ and converge to the boundary of $\Omega$, we have that $\varphi(T_n,p)$ leaves any compact subset contained in~$U$.

It is also worth to recall the generalization of this notion to the context of \emph{meromorphic} vector fields given by A.~Guillot and J.-C.~Rebelo in \cite{Guillot-Rebelo}.  A meromorphic vector field $X$ on a complex manifold $M$ is semicomplete if its restriction to $M\setminus X_\infty$  is a holomorphic semicomplete vector field.
By \cite[Corollary 12]{Guillot-Rebelo}, $X$ is semicomplete on $M$ if it is semicomplete in restriction to a non-empty Zariski open subset of $M$. Hence, semicompleteness of a meromorphic vector field is a birational invariant. According to Theorem~\ref{bir-hol}, this clearly implies that birational integrability implies semicompleteness in the context of rational vector fields on projective surfaces:

\begin{prop}\label{semi-comp}
Every birationally integrable vector field on a projective surface $S$ is semicomplete on $S$.
\end{prop}

We note that the converse of the above proposition is not true. This is shown by our second example.

\begin{ex}
In \cite{Ghys-Rebelo} it is shown that the vector fields 
\[
X_{m}\in\mf X_{\mr{rat}}(\p^2), \quad m\in\{(1,1,1),(1,1,2),(1,2,3)\}\cup\{(n+1,1,-1):n\in\Z\},
\] 
defined by
\[
X_{1,1,1}=x(x-2y)\partial_x+y(y-2x)\partial_y,\  X_{1,1,2}=x(x-3y)\partial_x+y(y-3x)\partial_y,\ X_{1,2,3}=x(2x-5y)\partial_x+y(y-4x)\partial_y
\]	
and
\[
X_{n+1,1,-1}=x^2\partial_x-y(nx-(n+1)y)\partial_y,
\]
having first integrals $f_m(x,y)=x^{m_1}y^{m_2}(x-y)^{m_3}$, are all semicomplete.
The vector fields $X_{n+1,1,-1}$, $n\in\Z$, are birationally integrable and tangent to a rational fibration. In fact, they are all birationally equivalent to $\partial_x$.
The remaining vector fields are tangent to an elliptic fibration and therefore they are not birationally integrable (see Remark~\ref{no-elliptic}).
\end{ex}

A third example is provided by the  Loud family of quadratic vector fields with real  centers. Some of the vector fields of the family are birationally integrable and others are not.

\begin{ex}
The linear isochronous real center 
is given by the singularity of the vector field	$-y\partial_x+x\partial_y$, which in the complex plane $\p^2$ is holomorphically conjugated to $iy\partial_y$. It is well known that every isochronous center is  linearizable, i.e. locally analytically conjugated to the linear one.
In the dehomogenized Loud family $\mc L$ of reversible quadratic centers
\[
	-y(1-x)\partial_x+(x+Dx^2+Fy^2)\partial_y,	
\]
where $(D,F)\in\R^2$,
there are exactly four  centers that are isochronous. They correspond to the parameters
$(0,\frac{1}{4})$, $(0,1)$, $(-\frac{1}{2},\frac{1}{2})$ and $(-\frac{1}{2},2)$.
The local flow of an arbitrary vector field of the family $\mc L$ is not known. However, the four vector fields having isochronous centers can be integrated explicitly. The first three are birationally integrable on $\p^2$ but the fourth is not. Indeed, its local flow is the following
\[
	\varphi_t(x,y)=\left(1+\frac{x-1}{\sqrt{(x-1)^2+x(2-x)\cos t-2y\sin t}},\frac{\frac{1}{2}x(2-x)\sin t+y\cos t}{(x-1)^2+x(2-x)\cos t-2y\sin t}\right),
\]
which is multivalued.
The first three vector fields are tangent to a rational fibration (and birationally equivalent to $iy\partial_y$). The vector field with parameters $(D,F)=(-\frac{1}{2},2)$ is tangent to the elliptic fibration defined by  the rational function $\frac{4y^2-2x^2+4x-1}{(x-1)^4}$. 
In \cite[Table I]{MRT} the authors provide explicit rational linearizations  of these four centers, the first three, being in fact birational, define their regularizations. 
\end{ex}

We recall that the Kodaira dimension of surfaces is a birational invariant and that for minimal surfaces of non-negative Kodaira dimension one has $\operatorname{Bir}(S) = \operatorname{Aut}(S)$. Each surface $S$ with negative Kodaira dimension is birationally equivalent to a product $C\times \p^1$, where $C=\p^1$ if $S$ is rational, and $g(C)\geq 1$ if $S$ is a non-rational ruled surface. Moreover, if two ruled surfaces are birationally equivalent then their bases are biholomorphic.

 If $S=C\times\p^1$ we denote by $\mr{Bir}_\pi(S)$ (resp. $\mr{Bir}_C(S)$) 
the subgroup of $\mr{Bir}(S)$ preserving  the rational fibration $\pi:S\to C$ (resp. each fibre of the rational fibration). 
The following result is well-known. 
\begin{prop}\label{Birpi}
 Let $S$ be the ruled surface $C\times \p^1$ with $C$ a smooth projective curve. Then 
 $\operatorname{Bir}_C(S)\cong \operatorname{PGL}_2(\C(C))$
 where $\C(C)$ is the field of rational functions of $C$, and 
 $\operatorname{Bir}_{\pi}(S) = \operatorname{Aut}(C)\ltimes \operatorname{Bir}_C(S)$.
 Moreover, 
 if $g(C) > 0$, then $\operatorname{Bir}(C\times \p^1) = \operatorname{Bir}_{\pi}(S) = 
 \operatorname{Aut}(C)\ltimes \operatorname{Bir}_C(S)$.
\end{prop}

\smallskip

The following two lemmas show general properties of birationally integrable vector fields. It follows from them that birationally integrable vector fields on a projective curve or on a minimal surface of non-negative Kodaira dimension are holomorphic. We note that a rational vector field on a projective manifold defines a one-dimensional holomorphic foliation.

\begin{lema}\label{lema_2}
 Let $X$ be a rational vector field on a projective manifold $M$ (of arbitrary dimension) and let $\F$ be the foliation defined by $X$.
 Assume that $X$ is birationally integrable. Then the polar locus $X_\infty$ of $X$ is tangent to $\F$.
\end{lema}
\begin{proof}
Suppose there is a component $H$ of  $X_\infty$ that is transverse to $\F$ at a given point $p\in H$, which can be assumed to be a smooth point of both $H$ and $\F$. We can find local coordinates $(u_1, \dots, u_n)$ in an Euclidean neighborhood $W$ of $p$  that are adapted to the foliation, i.e. $\F$ is defined in $W$ by $u_i=\rm{ct.}$ for $i=1, \dots, n-1$, and such that $H\cap W$ is given by $u_n =0$.  In these coordinates
$$
X|_W = \frac{h(u_1, \dots, u_n)}{u_n^k} \partial_{u_n}
$$
with $h$ a holomorphic function not divisible by $u_n$ and $k$ a positive integer. The restriction of $X$ to a (generic) local leaf $L$ of $\F$ inside $W$ is of the form
$$
X|_L = \frac{h_1(u_n)}{u_n^k} \partial_{u_n},
$$
with $h_1$ a holomorphic function not divisible by $u_n$. In fact, by means of a coordinate change, $X|_L$ can be written as $X|_L = \frac{1}{u_n^k} \partial_{u_n}$ (cf. \cite{G-G-J}). But such a vector field generates a local flow that is multivalued. Hence $X$ cannot be birationally integrable unless $k\leq 0$, leading to a contradiction.
\end{proof}

The proof of the above lemma implies in particular that, for a projective curve $C$, one has $\mf X_{\rm{bir}}(C) = \mf{aut}(C)$. In higher dimension we have the folllowing:
\smallskip

\begin{lema}\label{todo_hol}
Let $X$ be a birationally integrable vector field on a projective manifold $M$ of dimension $n\geq 2$. If the group morphism $\phi:\C\to\mr{Bir}(M)$ associated to $X$ takes values in $\operatorname{Aut}(M)$ then $X\in \mf{aut}(M)$.
\end{lema}

\begin{dem}
The general case being similar, we suppose that $n=2$. We argue by contradiction assuming that the polar locus of $X$ is non-empty.
We denote by ${\mb D}_{\varepsilon}$ the open disc of $\C$ centered at zero and of radius $\varepsilon$.
The case $X=0$ being trivial, we suppose that $X\neq 0$. Then, the restriction of $\phi$ to ${\mb D}_{\varepsilon}$ of $\C$ is injective if $\varepsilon$ is small enough. 

Let $U$ be an admissible open subset of $\mb C\times M$ where the local flow $\varphi^X$ of $X$ is well-defined and fulfills $\varphi^X_t(p) = \phi(t)(p)$ for all $(t,p)\in U$.  Using analytic continuation, we can suppose that $U\cap\big(\{0\}\times M\big) = \Omega_X = M \setminus X_\infty$. By the previous Lemma, $X$ is tangent to the polar locus $X_\infty$, which we assume is non-empty. Let $p$ be a point in the regular part of  $X_\infty$, which is not a singular point of the foliation $\mc F$ induced by $X$. There is a local chart $\{W, (x,y)\}$ centered at $p$ and such that, on $W$, the foliation $\mc F$ is defined by $x=\rm{ct.}$ and $X_\infty$ is given by $x=0$. 

We fix positive numbers $\varepsilon, \delta, r, r_1, r_2$, with $r_1< r_2$, which will be taken as small as necessary. We consider the sets $P\subset Q$ defined by
\[
\begin{aligned}
 P &= \{(x,y)\in W \mid r_1\leq |x|\leq r_2, |y| \leq r\} \\
 Q &= \{(x,y)\in W \mid |x|\leq r_2, |y| \leq r\}.
\end{aligned}
\]
Our first requirements on the above numbers are that $P_\varepsilon:={\mb D}_{\varepsilon}\times P$ is included in the admissible open subset $U$, and that $\varphi^X_t(P_\varepsilon)\subset W$ if $|t|<\varepsilon$. Then, we can write
\[
\varphi^X_t(x,y) = \big(x, f(t,x,y)  \big) = \phi(t)(x,y),
\]
where $(t, x, y)$ belongs to an open neighborhood of $\bar{\mb D}_\varepsilon\times P$ in $U\cap \big(\C\times W\big)$, and $f$ is holomorphic in its domain.

We endow the Lie group $\mr{Aut}(M)$ with the distance induced by a left-invariant Riemannian metric. 
Since $\mr{Aut}(M)$ is the union of a countable number of compact subsets, there is an open ball $B_{\delta/2}(g)$ of $\mr{Aut}(M)$ (centered at a certain automorphism $g$) such that $T'= \phi({\mb D}_{\varepsilon/2})\cap B_{\delta/2}(g)$ is uncountable. The restriction $\phi_\varepsilon$ of $\phi$ to the disc $\mb D_{\varepsilon}$ is injective and the set $\phi_\varepsilon ^{-1}(T')$ is also uncountable. We denote by $T$ the subset of $\mb D_{\varepsilon}$ defined as $T= \{ t=s-s'\mid s, s'\in\mb D_{\varepsilon} \text{ and } \phi(s), \phi(s') \in T'\}$. Then 
$$
\phi(T)=\{\phi(t) = \phi(s-s')= \phi(s) \circ \phi(s')^{-1} \mid t=s-s'\in T\} 
$$ 
is an uncountable subset of $\mr{Aut}(M)$ contained in the ball $B_{\delta}({\rm Id})$. Note that the identity $\rm Id\in\phi(T)$ is an accumulation point of $\phi(T)$. 
By taking $\delta$ small enough, we can assure that $\phi(t)(Q)\subset W$ for each $t\in T$. In particular, for each $t\in T$, we can write $\phi(t)$ in the local coordinates $(x,y)$ as
\[
\phi(t) = \big( \phi^1(t,x,y), \phi^2(t,x,y) \big),
\]
where the components $(x,y) \mapsto \phi^i(t,x,y)$ are holomorphic.
The function given by
\[
F(t,x,y) = \frac{1}{2i\pi}\int_{|\zeta|=r_2} \frac{f(t,\zeta,y)}{\zeta - x}\, d\zeta
\]
is well-defined in $\mb D_\varepsilon \times \mathring{Q}= \mb D_\varepsilon \times \mb D_{r_2} \times \mb D_r$. 
Moreover, $F$ is separately holomorphic as can be seen by derivation under the integral sign. Then, Hartogs' theorem implies that $F$ is holomorphic. We now observe that, for a given $t\in T$, we have $\phi^2(t,x,y) = f(t,x,y)$, and, since $(x,y)\mapsto \phi^2(t,x,y)$ is holomorphic, the following equality holds
\[
F(t,x,y) = \phi^2(t,x,y) = f(t,x,y), \qquad \forall t\in T, \quad \forall(x,y)\in \mathring{P}.
\]
As $T$ is infinite, we deduce that $F(t,x,y) = f(t,x,y)$ on $\mb D_\varepsilon\times \mathring{P}$. That is, $f(t,x,y)$ extends holomorphically to $\mb D_\varepsilon \times \mathring{Q}= \mb D_\varepsilon \times \mb D_{r_2} \times \mb D_r$ as the function $F(t,x,y)$, but this contradicts $p$ being a point in the polar set of $X$, ending the proof.
\end{dem}

\begin{cor}
Let $M$ be a projective manifold such that $\operatorname{Bir}(M) = \operatorname{Aut}(M)$, then $\mf X_{\rm{bir}}(M) = \mf{aut}(M)$. 
\end{cor}

In particular we obtain:

\begin{cor}\label{tot_hol}
Let $S$ be a minimal surface of non-negative Kodaira dimension, then $\mf X_{\rm{bir}}(S) = \mf{aut}(S)$. Moreover, two vector fields $X_1, X_2\in\mf{aut}(S)$ are birationally equivalent if and only if they are holomorphically equivalent.
\end{cor}

\medskip

We now consider rational vector fields on the ruled surface $C\times\p^1$ that are tangent to the rational fibration. The following theorem characterizes those that are birationally integrable.

\begin{teo}\label{fibration_2}
Let $S$ be the ruled surface $C\times \p^1$ with $C$ a smooth projective curve and let $y$ be an affine coordinate of $\p^1$.  A rational vector field $X$ on $S$ tangent to the rational fibration $S\to C$ is birationally integrable if and only if $X=(a(x)y^2+2b(x)y+c(x))\partial_y$, with $a,b,c\in\C(C)$, and 
$\Delta(X):=b(x)^2-a(x)c(x)\equiv\kappa^2$ is constant. 
In that case 
there is a birational automorphism 
$\phi\in \operatorname{Bir}_C(S)$ such that $\phi^* X=2\kappa y\partial_y$
if $\kappa\in\C^*$,
or
$\phi^* X= \partial_y$ if $\kappa =0$, and  therefore $X$ is regularizable.
\end{teo}

\begin{proof}
As the rational vector field $X$ is tangent to $\F$
it can be written in the form $X=h(x,y)\partial_y$, where $y$ is an affine coordinate of $\p^1$ and $h=h(x,y)$ a rational function on $S$. Since $X$ is 
birationally integrable, Lemma~\ref{lema_2} implies that the polar 
locus $X_\infty$ of $X$ is a finite union of fibres. Let $P\subset C$ be the projection of these fibres and set $U_0= C\setminus P$ and $U=\pi^{-1}(U_0)$.
Then $X|_U$ is a holomorphic vector field that can be written
$$
X|_U=\big(a(x) y^2 + 2b(x) y +c(x)\big)\, \partial_y
$$
with $a, b, c$ holomorphic functions on $U_0$. In fact, $a, b, c$ belong to the field $\C(C)$ of rational functions of the curve $C$.
To see this,
we note first that $\{y=0\}$ is not contained in $(h)_{\infty}=(X)_{\infty}$, because of the above lemma, and therefore $c(x) = h(x,0)$ is a 
well-defined rational function on $C$. Now  
$h(x,y) - h(x,0)$ is a rational function on $C\times\p^1$ that coincides with $\big(a(x) y + 2b(x)\big) y$ on $U$ and whose polar set does not contain $\{y=0\}$.
Therefore $h_1(x,y) :=\big( h(x,y) - h(x,0) \big)/y$ is a rational function on $C\times\p^1$
and  $2b(x) = h_1(x,0)$ is a rational function on $C$. Repeating the argument we see that $a(x)$ is also an element of $\C(C)$.

We will use the natural isomorphism $\mathfrak{sl}_2(\C)\stackrel{\cong}{\longrightarrow}\mf{aut}(\p^1)$ given by 
\[\left(
\begin{array}{cc}
\beta/2 & \gamma \\[0mm]
-\alpha & -\beta/2
\end{array}
\right) \mapsto (\alpha y^2+\beta y+\gamma)\partial_y.\]
We note that under that isomorphism, conjugation in $\mf{sl}_2(\C)$ by a matrix $A$ in $\mr{SL}_2(\C)$ corresponds to pull-back in $\mf{aut}(\p^1)$ by the automorphism of $\p^1$ defined by $A$.
Using the above identification the vector field $X$ can be thought as the matrix
$$
X = 
\left(
\begin{array}{cc}
b & c \\[0mm]
-a & -b
\end{array}
\right) 
\in\mathfrak{sl}_2(\C(C)).
$$
There is a finite field extension $\C(C)\subset K$ and $Q\in\operatorname{SL}_2( K)$ such that $X = Q D Q^{-1}$ 
with $D\in\mathfrak{sl}_2(K)$ in Jordan normal form. Let $\tilde C$ be the smooth curve associated to the field $K$,
i.e. such that $K = \C(\tilde C)$. There are two possibilities
$$
1) \;\; D =
\left(
\begin{array}{cc}
\delta & 0 \\[0mm]
0 & -\delta
\end{array}
\right) \quad \text{with } \; \delta=\sqrt{\Delta(X)}\in \C(\tilde C)\setminus\{0\}  \qquad \text{or} \qquad 
 2) \;\; D =
\left(
\begin{array}{cc}
0 & 1 \\[0mm]
0 & 0
\end{array}
\right).
$$
Since $\exp(tX)= Q \exp(tD) Q^{-1}$ corresponds to the birational flow of the vector field $(a(x)y^2+2b(x)y+c(x))\partial_y$ and 
$\mr{PGL}_2(\C(C))\cong \operatorname{Bir}_C(C\times\p^1)$, we deduce that 
$\exp(tX)\in\mr{GL}_2(\C(C))$ and, using that $Q\in\mr{SL}_2(\C(\tilde C))$, we obtain that
$\exp(tD)\in \operatorname{SL}_2( \C(\tilde C))$.

In the first case this implies in particular that $\exp(t\delta)\in \C(\tilde C)$ for each $t\in \C$. 
We set $f= \exp(\delta)$. Then $f^t\in \C(\tilde C)$ for each $t\in \C$. If $\delta$ is not constant the divisor $(f)$ is not empty and there is a 
point $p\in (f)\subset\tilde C$ such that $f$ can be written $f(x) = x^k$ in a neighborhood of $p$, where $x$ is a local coordinate of $C$ at $p$ and $k$ is an integer different from zero. But then 
$f^t(x)= x^{tk}$ would be multivalued for $t\in \C\setminus\Q$, leading to a contradiction. Hence, we deduce that $\delta\equiv \kappa$ is constant. 
This implies that $D\in\mathfrak {sl}_2(\C)$ and $Q\in\operatorname{SL}_2(\C(C))$. Then $Q$ defines a birational
conjugacy $\Phi$ of $C\times\p^1$ which transforms the vector field $X$ into the holomorphic vector field defined by $D$, that is 
$2\kappa y\partial_y$. 

In the second case we already have $D\in\mathfrak{sl}_2(\C)$ and consequently $Q\in\operatorname{SL}_2(\C(C))$. Therefore $Q$ defines a birational transformation
of $C\times\p^1$ which transforms $X$ into the holomorphic vector field~$\partial_y$.
\end{proof}

\medskip

The following result, due to S. Cantat and C. Favre (cf. {\cite[Theorem~1.2]{CF}}), reduces the class of surfaces and vector fields to be considered in order to prove Theorem~\ref{bir-hol}. 
If $\F$ is a holomorphic foliation 
on a projective manifold we denote by  $\mr{Aut}(\F)$ (resp. $\mr{Bir}(\F)$) the group of holomorphic (resp. birational) automorphisms of the 
manifold preserving $\F$.
We recall that a foliation $\F$ on a projective surface $S$ is called rational fibration  if their leaves are the fibers of a rational fibration $\pi\colon S\dashrightarrow C$ over a curve $C$.
It is well known that every rational fibration is birationally equivalent to a $\p^1$-bundle over a curve $C$, which is also birationally equivalent to the product $C\times\p^1\to C$.

\begin{teo}[Cantat-Favre]\label{CF0}
Let $\F$ a foliation on a projective surface  such that the strict inclusion $\mr{Aut}(\F)\subsetneq\mr{Bir}(\F)$ holds for every birational model of $\F$. Then $\mr{Bir}(\F)$ possesses an element of infinite order and either $\F$ is a rational fibration, or $\F$ is induced by a linear vector field $x\partial_x+\alpha y\partial_y$ on $\p^1\times\p^1$ or on the desingularization of the quotient of $\p^1\times\p^1$ by the involution $\sigma(x,y)=(1/x,1/y)$.
\end{teo}

\medskip

From the above result we deduce the following consequence.

\begin{cor}\label{CF2}
Let $\F$ be a foliation on a projective surface which is not a rational fibration. Then there is a birational model $\F'$ of $\F$ such that $\mf{aut}(\F')=\mf X_{\mr{bir}}(\F')\cong\mf X_{\mr{bir}}(\F)$.
\end{cor}

\begin{proof}
According to Theorem~\ref{CF0}, if $\F$ is not birationally equivalent to the foliation defined by a linear vector field on $S$, then there is a birational map $f:\bar S\to S$ such that $\bar\F=f^*\F$ fulfills
 $\mr{Aut}(\bar\F)=\mr{Bir}(\bar\F)$. In this case, using Lemma~\ref{todo_hol}, we conclude that $f^*\mf X_{\mr{bir}}(\F)=\mf{aut}(\bar\F)$. Therefore we can assume that $\F$ is the foliation defined by
$x\partial_x+\alpha y\partial_y$, with  $\alpha\in \C\setminus\mb{Q}$, in $S=\p^1\times\p^1$ or in the desingularization $S'$  of the quotient of $S=\p^1\times\p^1$ by the involution $\sigma(x,y)=(1/x,1/y)$. 

In the first situation, and using that  $\F$ has no dicritical singularities, it follows (cf. \cite[\S5.1]{CF}) that any $\psi\in\mr{Bir}(\F)$ restricts to a biholomorphism of $\C^*\times\C^*$. It is well-known 
that $\operatorname{Aut}(\C^*\times\C^*)$ is the Lie group $G$ of transformations of the form $\psi(x,y)=(\lambda x^ay^b,\mu x^cy^d)$ with $\lambda,\mu\in\C^*$ and $A=\left(\begin{array}{cc}a & b\\ c & d\end{array}\right)\in\mr{GL}(2, \Z)$.  Then $G_0= \operatorname{Aut}_0(\C^*\times\C^*)\cong \C^*\times\C^*$. Let $X$ be a birationally integrable vector field belonging to $\mf X_{\mr{bir}}(\F)$, and let $\phi\colon\C\rightarrow\mr{Bir}(\p^1\times\p^1)$ be the  group morphism associated to $X$ according to Definition~\ref{0}. It follows from the above considerations that $\phi$ induces a morphism $\bar\phi\colon\C\rightarrow G/G_0=\mr{GL}(2, \Z)$.
Since each Abelian subgroup of $\mr{GL}(2, \Z)$ is isomorphic to a subgroup of $\Z\oplus\Z_2$, or $\Z_2\oplus\Z_2\oplus\Z_2$, or $\Z_m\oplus\Z_2$, for a certain $m$, we deduce that the morphism $\bar\phi$ must be the identity. Hence the restriction to $\C^*\times\C^*$ of each $\phi(t)$ has the form $\phi(t)(x,y)=(\lambda(t) x, \mu(t) y)$. The fact that $\phi$ coincides with $\varphi^X$ in a admissible open subset of $\C\times\p^1\times\p^1$ implies that the functions $\lambda, \mu$ are smooth. Therefore the restriction of $\phi(t)$ to $\C^*\times\C^*$ defines a one parameter subgroup of the Lie group $G_0$. This implies that the restriction to the surface $\C^*\times\C^*$ of the vector field is a linear combination of $x\partial_x$ and $y\partial_y$. But all these vector fields extend holomorphically to $\p^1\times\p^1$.
Hence $X\in\mf{aut}(\F)$.

Let us consider now the foliation $\F'$ induced by $\F$ in the desingularization $S'$ of $S/\sigma$ and let $\psi$ be a birational transformation of $S'$ preserving $\F'$. The same argument as before implies that $\psi$ induces  a biholomorphism of $((\C^*\times\C^*)\setminus\{(1,1),(-1,1),(1,-1),(-1,-1)\})/\sigma$ which lifts to a biholomorphism of $(\C^*\times\C^*)\setminus\{(1,1),(-1,1),(1,-1),(-1,-1)\}$ extending holomorphically to $\C^*\times\C^*$ by Hartogs' theorem. We conclude as in the previous case.
\end{proof}

\bigskip

Combining all the above results, we are able to prove Theorem~\ref{bir-hol}.

\begin{proof}[Proof of Theorem~\ref{bir-hol}]
We have to prove that a given birationally integrable vector field $X$ on a projective surface $S$ is regularizable. Since $X$ is tangent to the foliation $\F$ defined by $X$, the image of the morphism $\phi\colon\C\to \operatorname{Bir}(S)$ associated to $X$ is a subgroup of $\operatorname{Bir}(\F)$.

According to Corollary~\ref{CF2}, it only remains to deal with the case where $\F$ is a rational fibration.
By taking an appropriate birational model we can suppose that $S = C\times\p^1$,
where $C$ is a curve, and the leaves of $\F$ are the fibres $\pi^{-1}\{x_0\}\cong \p^1$, where $\pi\colon S \rightarrow C$ is the natural projection. But this case follows directly from Theorem~\ref{fibration_2}. Finally, we can assume that $S$ is minimal due to the inclusion $\mathfrak{aut}(S)\hookrightarrow  \mathfrak{aut}(S_{min})$, where $S_{min}$ is a minimal model of $S$.
\end{proof}

\smallskip
As a consequence of Theorem~\ref{bir-hol} and some results due to D.~Lieberman \cite{Lie2}, we obtain:
\begin{cor}\label{preserva-fibracion}
Let $X\neq 0$  be a birationally integrable vector field on a  projective surface $S$. 
Then,  either $X$ is birationally equivalent to a linear vector field on a torus without non-trivial subtorus or $X$ preserves a rational or elliptic fibration on $S$.
\end{cor}

\begin{proof}[Proof of Corollary~\ref{preserva-fibracion}]
 The preservation of a rational or elliptic fibration by a vector field is a property that is invariant under birational transformations. Therefore, using Theorem~\ref{bir-hol}, we can assume that $X$ is holomorphic. From \cite[Theorems~4.4 and~4.6]{Lie2} it follows that there are the following four possibilities: (i) $X$ is tangent to a Seifert elliptic fibration,  (ii) $X$ is a linear vector field on a torus, (iii) $S$ is a ruled surface $S=\p E$ and $X$ commutes with a vector field $Y\in\mr{Lie}\,\overline{\exp\C X} \setminus\{0\}\subset\mf{aut}(\p E)$ which is tangent to the rational fibration $\p E\to C$, or (iv) $S=\p^2$ and $X$ commutes with a vector field $Y \in\mr{Lie}\,\overline{\exp\C X}\setminus\{0\}\subset\mf{aut}(\p^2)$ which can be taken tangent to a rational fibration, as follows from Remark~\ref{fibP2}.
\end{proof}

\section{Algebras of birationally integrable vector fields}

In this section we prove some general properties of the set $\mf X_{\mr{bir}}(M)$, for an arbitrary projective manifold $M$. In particular, we prove Theorem~\ref{semi}, which states that a semidirect sum of regularizable subalgebras of $\mf X_{\mr{rat}}(M)$ is also regularizable. We also recall the notions of algebraic subgroups and Lie-type subgroups of $\mr{Bir}(M)$, and we recall Weil's regularization theorem in the formulation used in this article.

The subset $\mf X_{\mr{bir}}(M)$ of the Lie algebra $\mf X_{\mr{rat}}(M)$ is not necessarily closed under the sum or under the Lie bracket. This is shown, in the case of surfaces, by the following example.
\begin{ex}\label{no-algebra}
Let $(x,y)$ be affine coordinates of the projective plane $\p^2$ and consider the rational vector fields $X=y^2\partial_y$, $Y=x\partial_y$ on $\p^2$. They are birationally integrable, and therefore regularizable, but a direct computation, or the criterion given by Theorem~\ref{fibration_2}, shows that 
$X+Y = (y^2+x)\partial_y\notin\mf X_{\mr{bir}}(\p^2)$ and $[X,Y]=[x\partial_y,y^2\partial_y]=2xy\partial_y\notin\mf X_{\mr{bir}}(\p^2)$. 
 We note that $X$ and $Y$ are not simultaneously regularizable, otherwise they would be contained in a finite dimensional regularizable Lie subalgebra of $\mf X_{\mr{bir}}(\p^2)$, which is not the case.
In fact $X$ and $Y$ generate the infinite dimensional Lie algebra $\mf{sl}_2(\C[x])\cong (\C[x]\oplus\C[x]y\oplus\C[x]y^ 2)\partial_y\not\subset\mf X_{\mr{bir}}(\p^2)$.
\end{ex}

The above example also shows that a vector subspace of $\mf X_{\mr{rat}}(M)$ spanned by birationally integrable vector fields is not necessarily contained is $\mf X_{\mr{bir}}(M)$. 
Our first result in this section (Theorem~\ref{XbirM}) states that if the vector subspace is  in fact a finite-dimensional Lie algebra then it is actually contained in  $\mf X_{\mr{bir}}(M)$.

\medskip

Recall that an admissible open subset of $\C\times M$ is an open Euclidean neighborhood of $\{0\}\times V$, where $V\subset M$ contains a  non-empty Zariski open subset of $M$. Notice that any finite intersection of admissible open subsets of $\C\times M$ is admissible. As before, we denote by $\mb D_\varepsilon$ the open disk of $\C$ centered at $0$ of and of radius $\varepsilon>0$.

\begin{lema}\label{Lema_admisible}
Let $U$ be an admissible open subset of $\C\times M$ and let $\varphi:U\to M$ be a holomorphic map such that $\varphi(0,p)=p$ for any $p\in U\cap(\{0\}\times M)$. Let $\mu:\mb D_\varepsilon\to\mb D_{\varepsilon'}$ be a holomorphic map with $\eta(0)=0$ and consider the map $\psi:U\cap(\mb D_\varepsilon\times M)\to \C\times M$ defined by $\psi(t,p)=(\mu(t),\varphi(t,p))$. If $W$ is an admissible open subset of $\C\times M$ then $\psi^{-1}(W)$ is also admissible.
\end{lema}
\begin{proof}
Clearly $U\cap(\mb D_\varepsilon\times M)$ and $W\cap(\mb D_{\varepsilon'}\times M)$ are admissible open subset of $\C\times M$ and $\psi^{-1}(W)$ is an Euclidean open subset of $\C\times M$ which contains $\{0\}\times V\subset U\cap W\cap(\{0\}\times M)$ for a suitable non-empty open Zariski subset $V$ of $M$.
\end{proof}

\begin{teo}\label{XbirM}
Let $M$ be a projective manifold.
Any finite dimensional Lie subalgebra $\mf g$ of $\mf X_{\mr{rat}}(M)$ generated, as a  vector space, by birationally integrable vector fields is contained in $\mf X_{\mr{bir}}(M)$.
\end{teo}

\begin{proof}
Let $G$ be the simply connected Lie group with Lie algebra $\mf g\subset\mf X_{\mr{rat}}(M)$, and let $\exp\colon \mf g\to G$ denote the exponential map. 
There is a non-empty Zariski open subset $\Omega_{\mf g}\subset M$ where all the elements of $\mf g$ are well-defined and holomorphic. Indeed, if 
$X_1,\ldots,X_n$ is a basis of $\mf g$ then $\Omega_{\mf g}=\bigcap_i \Omega_{X_i}$ and we note that $\Omega_{\mf g}$ does not depend on the chosen basis. According to a result of R. Palais (cf. \cite[Chapter II]{Palais}), the algebra $\mf g|{\Omega_{\mf g}}$ (obtained by restricting the elements of $\mf g$ to $\Omega_{\mf g}$) generates a local action of $G$ on $\Omega_{\mf g}$. More precisely, there is an Euclidean open neighborhood $U_G$ of $\{e\}\times \Omega_{\mf g}$ in $G\times M$ and a holomorphic map $\eta:U_G\to M$ 
such that $\eta(e,p)=p$, $\eta(g,\eta(h,p))=\eta(gh,p)$ and $d\eta(\cdot,p)_e(X)=X_p\in T_pM$ for each $X\in\mf g$ and $p\in \Omega_{\mf g}$. Therefore, for any $X\in\mf g$ there is an admissible open subset $U_X$ of $\C\times M$ such that the local flow $\varphi^X$ of $X$ is defined on $U_X$ and fulfills 
\begin{equation}\label{phiEta}
 \varphi^X_t(p)=\eta(\exp(tX),p)
\end{equation}
for every $(t,p)\in U_X$.

Let  $X_1,\ldots,X_n$ be a basis of $\mf g$ consisting of birationally integrable vector fields. We denote by $\varphi^i$ their local flows, which are defined on admissible open subsets $U_i\subset\C\times M$, and by $\phi^i:\C\to\mr{Bir}(M)$ their associated morphisms. Given $a=(a_1,\dots,a_n)\in \C^n$, we set $X_a=a_1X_1+\cdots+a_nX_n$. We note that, to prove the theorem, it is sufficient to show that $X_a$ is birationally integrable for each $a$ belonging to a certain polydisc $\mb D_\varepsilon^n$, where $\varepsilon>0$.

We claim that there exist $\varepsilon,\varepsilon'>0$ and a biholomorphism  $\xi:\mb D_\varepsilon^n\to\xi(\mb D_\varepsilon^n)\subset\mb D_{\varepsilon'}^n$, with $\xi(0)=0$, such that for every $a\in\mb D_\varepsilon^n$ we have
\begin{equation}\label{expG}
\exp(a_1X_1+\cdots +a_nX_n)=\exp(\xi_1(a)X_1)\cdots\exp(\xi_n(a)X_n).
\end{equation}
Indeed, the maps $\psi,\psi':(\C^n,0)\to (G,e)$ defined by $\psi(b)=\exp(b_1X_1)\cdots\exp(b_nX_n)$ and $\psi'(a)=\exp(a_1X_1+\cdots+a_nX_n)$ are local biholomorphisms because the matrices (with respect to the canonical basis of $\C^n$ and the basis $X_1,\ldots,X_n$ of $\mf g$) of their differentials at $0$ are the identity. Therefore $\xi=\psi^{-1}\circ\psi':(\C^n,0)\to(\C^n,0)$ is a germ of biholomorphism for which identity~\eqref{expG} holds.

By recursively applying Lemma~\ref{Lema_admisible} and relation~\eqref{phiEta}, we deduce that for each $a\in\mb D^n_\varepsilon$ there exists an admissible open subset $U_a\subset \C\times M$ such that
\begin{align*}
\eta(\exp(\xi_1(ta)X_1)\cdots\exp(\xi_n(ta)X_n),p)&=\varphi^1_{\xi_1(ta)}\circ\cdots\circ\varphi^n_{\xi_n(ta)}(p)\\
&=\phi^1(\xi_1(ta))\circ\cdots\circ\phi^n(\xi_n(ta))(p)
\end{align*}
for any $(t,p)\in U_a$.
We fix $a\in\mb D_\varepsilon^n$ and we chose a real number $r_a>1$ such that $ta\in \mb D_\varepsilon^n$ for each $t\in \mb D_{r_a}$. Now, we define $\phi^a(t)\in \mr{Bir}(M)$, for $t\in \mb D_{r_a}$, as the birational map
\begin{equation}\label{phiXa_0}
 \phi^a(t) = \phi^1(\xi_1(ta))\circ\cdots\circ\phi^n(\xi_n(ta)).
\end{equation}
As a consequence of the above identities, if $(t,p)$ belongs to the admissible open subset $W_a:= U_a\cap U_{X_a}\cap\big(\mb D_{r_a}\times M \big)\subset\C\times M$ we have 
\begin{equation}\label{phiXa}
\phi^a(t)(p)=\eta(\exp(ta_1X_1+\cdots+ta_nX_n),p)=\varphi^{X_a}_t(p).
\end{equation}

If $p\in M\setminus(X_a)_\infty$ there exists $\varepsilon_p>0$, and a neighborhood $V_p$ of $p$ in $M$, such that $\mb D_{\varepsilon_p}\times V_p\subset W_a$. If $t,s,t+s\in\mb D_{\varepsilon_p}$ and $q\in V_p$ then
$\varphi_t^{X_a}\circ\varphi_s^{X_a}(q)=\varphi^{X_a}_{t+s}(q)=\varphi_s^{X_a}\circ\varphi_t^{X_a}(q)$. Using analytic continuation, we deduce that the following equalities of birational maps, 
\begin{equation}\label{t+s}
\phi^a(t)\circ\phi^a(s)=\phi^a(t+s)=\phi^a(s)\circ\phi^a(t),
\end{equation} 
are fulfilled whenever $t,s,t+s\in\mb D_{r_a}$. Therefore, if $t\in\mb D_{r_a}$ and $k\in\mb N$ then $\phi^a(t)=\left(\phi^a\left({\frac{t}{k}}\right)\right)^{\circ k}.$
We now globalize $\phi^a$ defining $\phi^a(T)=(\phi^a(T/n))^{\circ n}\in\mr{Bir}(M)$, for arbitrary $T\in\C$, taking $n\in\mb N$ large enough so that $|T/n|<r_a$. This definition does not depend on the choice of $n$ because if $k\in \mb N$ then
\[\left(\phi^a\left(\frac{T}{kn}\right)\right)^{\circ (kn)}=\left(\left(\phi^a\left(\frac{T}{kn}\right)\right)^{\circ k}\right)^{\circ n}=\left(\phi^a\left(\frac{T}{n}\right)\right)^{\circ n}.\]
The map $\phi^a\colon\C\to\mr{Bir}(M)$ is a group morphism. Indeed,  the equalities (\ref{t+s}) are fulfilled for $t,s\in\C$ small enough and consequently for any $T,S\in\C$ we have
\[
\begin{aligned}
\phi^a(T)\circ\phi^a(S)& = \phi^a(T/n)\circ\cdots\circ\phi^a(T/n)\circ \phi^a(S/n)\circ\cdots\circ\phi^a(S/n) \\
 & = (\phi^a(T/n)\circ\phi^a(S/n))\circ\cdots\circ (\phi^a(T/n)\circ\phi^a(S/n))\\
  & =\phi^a((T+S)/n)\circ\cdots\circ\phi^a((T+S)/n)=\phi^a(T+S).
\end{aligned}
\]
The identity~\eqref{phiXa} says now that, on  the admissible open subset $W_a$ of $\C\times M$,
 the local flow $\varphi^{X_a}$ of $X_a$ coincides with the 
one-parameter subgroup $\phi^a$ of $\rm{Bir}(M)$, and therefore the condition in Definition~\ref{0} is fulfilled
proving that 
$X_a\in\mf X_{\mr{bir}}(M)$. 
\end{proof}

The hypothesis that $\mf g$ is an algebra in the above theorem is essential. Example~\ref{no-algebra} shows that a vector subspace of $\mf X_{\mr{rat}}(M)$ generated by birationally integrable vector fields is not necessarily contained in $\mf X_{\mr{bir}}(M)$. 

\begin{obs}
The germ of biholomorphism $\xi$ only depends on the algebra $\mf g$ and the choice of the basis. 
 If this germ $\xi$ could be extended to a globally defined map, the equality~\eqref{phiXa_0} would already provide the flow of the vector field $X_a$, for each $t\in\C$, completing the proof of the theorem. 
This is the case, for instance, if $\mf g$ is the affine algebra generated by two birationally integrable vector fields $X, Y$, which fulfill $[X,Y] =Y$. Then 
$\xi(a,b)=(a,b\frac{e^a-1}{a})$
is globally defined. 
However, this is not true in general. If $\mf g$ is the Lie algebra $\mf{sl}_2(\C)$ generated by three birationally integrable vector field $X,Y,Z$, which satisfy the structure equations $[X,Y]=X$, $[X,Z]=2Y$, $[Y,Z]=Z$, then 
\[
\xi(a,b,c)=\left(\frac{2a\tan(d/2)}{d-b\tan(d/2)},\log\left(\frac{d^2}{2ac-bd\sin(d)+\frac{(d^2-b^2)}{2}\cos(d)}\right),\frac{2c\tan(d/2)}{d-b\tan(d/2)}
\right),
\]
where $d=\pm\sqrt{4ac-b^2}$.
Notice that the first and third components of $\xi$ are univaluated around $b^2-4ac=0$ because their expressions remains unchanged when we substitute $d$ by $-d$, but they have poles along $b=d\cot(d/2)$. Moreover, the second component of $\xi$ is multivaluated around $b^2-4ac=0$ because it can be checked that if $(a,b,c)$ tends to a point $(a_0,b_0,c_0)\in\{d=0\}$ then $\xi(a,b,c)\to(\frac{2a_0}{2-b_0},\log(\frac{1}{1-b_0}),\frac{2c_0}{2-b_0})$.
\end{obs}

\bigskip

Let $M$ be a projective manifold. We recall here the notion of \emph{total degree} of a birational tansformation of $M$ (cf.~\cite{Cantat}). The graph of $f\in\mr{Bir}(M)$ is the set $\Gamma_f\subset M\times M$ obtained as the Zariski closure of $\{(x, f(x))\in M\times M\mid x\in\operatorname{Dom}(f)\}$, where $\operatorname{Dom}(f)$ denotes the domain of $f$. We fix a K\"{a}hler metric on $M$. It induces a metric on $M\times M$. The total degree (or degree for short) of $f$ is then defined as the volume of $\Gamma_f$ with respect to that metric.

A subgroup $G$ of $\mr{Bir}(M)$ is said to be \emph{algebraic} if $G$ has a structure of algebraic group such that the map $G\times M \dashrightarrow M$ induced by the inclusion $G\subset\mr{Bir}(M)$ is rational. Then, the classical regularization theorem of A.~Weil~\cite {Weil} implies that $G$ is regularizable (cf.~\cite{Blanc}). This implies in particular that the elements of an algebraic subgroup $G$ of $\mr{Bir}(M)$ are of bounded degree. 

In our discussion we need to consider subgroups of $\mr{Bir}(M)$ that are regularizable, but not necessarily algebraic. This need motivates the following definition. 

\begin{defin}\label{Lie-type}
We say that a subgroup $G$ of $\mr{Bir}(M)$ is of \emph{Lie-type}, or that $G$ is a Lie subgroup of $\mr{Bir}(M)$, if $G$ is a Lie subgroup of an algebraic subgroup $G'$ of $\mr{Bir}(M)$.
\end{defin}

We note that if $f:M'\to M$ is a birational map and $G$ is a  subgroup of $\mr{Bir}(M)$ of Lie type then $f^*G=\{f^{-1}\circ g\circ f:g\in G\}$ is also a Lie subgroup of $\mr{Bir}(M')$.

Let $X$ be a birationally integrable vector field on a surface $S$ and let $\phi\colon\C \to \mr{Bir}(S)$ be the associated group morphism. It follows from Theorem~\ref{bir-hol} that $\phi(\C)$ is a Lie subgroup of $\mr{Bir}(S)$.
The following example, due to S. Cantat, shows an embedding of $\C$ as a subgroup of $\mr{Bir}(S)$ that is not of Lie-type.

\begin{ex}\cite[Example~3.5]{Cantat}. \label{no-flow}
The $\mb Q$-vector spaces $\C$ and $\C[x]$ are isomorphic because they both have the dimension of the continuum. In particular, there exists an isomorphism of Abelian groups $(\C,+)\stackrel{\cong}{\longrightarrow}(\C[x],+)$, $t\mapsto p_t(x)$. The image of the group morphism $\phi:\C\to\mr{Bir}(\p^2)$ given by $\phi(t)(x,y)=(x,y+p_t(x))$ has unbounded degree and can not be regularized. 
\end{ex}
\medskip

\begin{obs}\label{Lie_gr_alg}
 Let $G$ be a subgroup of $\mr{Bir}(M)$ of Lie-type. It is regularizable and there is a birational map $f:M'\to M$ such that $f^*G$ is a Lie subgroup of the Lie group $\mr{Aut}(M')$. Let $\mf g'$ be the Lie subalgebra of $\mf{aut}(M')$ associated to $f^*G$. Then $\mf g:=f_*\mf g'$ is a Lie subalgebra of $\mf X_{\mr{rat}}(M)$ contained in $\mf X_{\mr{bir}}(M)$. We say that $\mf g$ is the Lie algebra associated to $G$. This algebra is characterized by the fact that, for each $X\in \mf g$, the one-parameter group associated with $X$ is a Lie subgroup of~$G$.

Let $\mf g$ be subalgebra of $\mf X_{\mr{rat}}(M)$ that is regularizable, and let $f:M'\to M$ be a birational map such that $f^\ast \mf g \subset \mathfrak{aut}(M')$. Let $G'$ be the connected Lie subgroup of $\mr{Aut}(M')$ determined by the Lie subalgebra $f^\ast \mf g$. Then $G= f_\ast  G'=\{f\circ g\circ f^{-1}: g\in G'\}$ is a Lie subgroup of $\mr{Bir}(M)$ whose Lie algebra is $\mf g$ (in the sense defined above). We say that $G$ is the Lie subgroup of $\mr{Bir}(M)$ associated to the Lie algebra $\mf g$. 
\end{obs}

\medskip

We recall here Weil's regularization theorem as stated in \cite[Theorem~2.5]{Cantat}:
\begin{teo}[Weil]\label{Weil}
Let $M$ be a complex projective manifold. Let $G$ be a subgroup of $\mr{Bir}(M)$. If $G$ has bounded degree, then $G$ can be regularized, i.e. there exists a birational map $f:M'\to M$ such that $f^*G=\{f^{-1}\circ g\circ f: g\in G\}\subset\mr{Aut}(M')$.
\end{teo}

With the above notation, the previous theorem states that a subgroup of $\mr{Bir}(M)$ of bounded degree is of Lie type.

\bigskip

\begin{proof}[Proof of Theorem~\ref{semi}] 

Let $\mf g$ be a subalgebra of $\mf X_{\mr{rat}}(M)$ which fits in a split exact sequence of Lie algebras $0\to \mf h\hookrightarrow \mf g\to \mf s\to 0$, where $\mf h$ and $\mf s$ are subalgebras of $\mf X_{\mr{rat}}(M)$ that are assumed to be regularizable. Since $\mf s$ is regularizable we can suppose that $\mf s\subset\mf{aut}(M)$. Let $S$ denote the connected Lie subgroup of ${\rm Aut}(M)$ associated to $\mf s$. 
As the ideal  $\mf h$ of $\mf g$ is also regularizable, there is a birational map $F:M'\to M$ such that $\mf h':=F^*\mf h\subset\mf{aut}(M')$. Let $H'\subset\mr{Aut}(M')$ be the connected Lie subgroup associated to $\mf h'$. Then $H=F_\ast H'$ is a Lie subgroup of $\mr{Bir}(M)$ whose associated Lie algebra is $\mf h$. 
Note that the elements of both $S$ and $H$ are of bounded degree.

We want to construct a subgroup $G$ of $\mr{Bir}(M)$, obtained as a semidirect product of $H$ and $S$, which is a Lie subgroup of $\mr{Bir}(M)$ and whose associated Lie algebra is $\mf g$. 

Let $X\in\mf s$ and $Y\in\mf h$ be given. Note that $X$ and $Y$ are respectively holomorphic on $M$ and $\Omega_Y= M\setminus Y_\infty$. For each $t\in\C$, 
the vector field $Y(t)=\exp(tX)^*Y$ is rational and holomorphic on the Zariski open set $\Omega_{Y(t)}=\exp(tX)^{-1}(\Omega_Y)$. Here we denote by $\exp (tX)$ the holomorphic global flow of $X$. Fix $t_0\in\C$ and a point $p\in \Omega_{Y}$, and set $p_{t_0}= \exp(t_0X)^{-1}p\in \Omega_{Y(t_0)}$. Then
\begin{equation}\label{derivada}
\begin{aligned}
 \left.\frac{d}{dt}\right |_{t_0} \big(Y(t) \big)(p_{t_0}) & = 
 \lim_{h\to 0} \frac{d\big(\exp(-(t_0+h)X)\big)(Y_{\exp hX(p)}) - d\big(\exp(-t_0 X)\big)(Y_p)}{h} \\
 & = d\big(\exp(-t_0X)\big) \Big( \lim_{h\to 0} \frac{d\big(\exp(hX)\big)(Y_{\exp hX(p)}) - Y_p}{h} \Big) \\[2mm]
 & = d\big(\exp(-t_0X)\big) \big((L_XY)_p\big) = \big(\exp(t_0X)^\ast [X,Y]\big)(p_{t_0})
\end{aligned} 
\end{equation}
is a well-defined element of the tangent space $T_{p_{t_0}}M$. This proves that 
\[
\partial_t|_{t=t_0}\big(\exp(tX)^*Y\big)=\exp(t_0X)^*[X,Y]\in\exp(t_0X)^*\mf h,
\] 
because $\mf h$ is an ideal of $\mf g$. Therefore, for each $t_0\in\C$, the derivative $Y'(t_0)$ of $Y(t)$ at $t=t_0$ is a rational vector field,  which is holomorphic on $\Omega_{Y(t_0)}$, and that belongs to $\exp(t_0X)^*\mf h$.  In particular it is birationally integrable.

We claim now that, in fact, $Y(t)=\exp(tX)^* Y\in\mf h$ for each $Y\in \mf h$ and each $t\in \C$. Let us chose a basis $Y_1,\ldots,Y_m$ of $\mf h$. 
As above we denote $\Omega_{\mf h}=\bigcap_i \Omega_{Y_i}$, noticing that $\Omega_{\mf h}$ does not depend on the chosen basis of $\mf h$.
We can write $[X,Y_i]=\sum_j c_i^jY_j$. Computation~\eqref{derivada} applied to $Y_i(t)= \exp(tX)^* Y_i$ says that $Y_i'(t) = \sum_j c_i^j Y_j(t)$, where this identity holds at each vector space $T_{p_{t}}M$ with $p_t\in \exp(t X)^{-1}\Omega_{\mf h}$.
Therefore, if we denote $\mathcal Y=(Y_1,\ldots,Y_m)^t\in\mf h^m$ and we set $C=(c_i^j)$, the column vector $\mathcal Y(t)=(Y_1(t),\ldots,Y_m(t))^t$ satisfies the linear differential system with constant coefficients
\begin{equation}\label{eq_diff}
\mathcal Y'(t)=C\mathcal Y(t),
\end{equation}
on the Euclidian open subset $U=\big\{(t,p)\in\C\times M: p\in(\exp tX)^{-1}(\Omega_{\mf h})\big\}$.
We fix $p\in\Omega_{\mf h}$ once and for all. Since $(0,p)\in U$, there exists an Euclidean neighborhood $V$ of $p$ in $M$ and $\varepsilon>0$ such that $\mb D_\varepsilon\times V\subset U$. 
We consider equation (\ref{eq_diff}) with initial condition $\mc Y(0)=\mc Y$ at the vector space $(T_qM)^m$ where $q\in V$.
Solving it we deduce that  $\mc Y(t)(q)=\exp(tC)\mc Y(q)$ for every $(t,q)\in\mb D_\varepsilon\times V$. For each $t\in\mb D_\varepsilon$, the vector fields $\mc Y(t)$ and $\exp(t C)\mc Y$ are meromorphic and coincide on the Euclidian open subset $V\subset M$. By analytic continuation we deduce that  $\mc Y(t)=\exp(t C)\mc Y\in\mf h^m$ for every $t\in\mb D_\varepsilon$. Since the pull-back map $\exp(tX)^*$ is linear, the above identity implies that 
$Y(t)=\exp(tX)^* Y$ belongs to $\mf h$ for any given $Y\in\mf h$ and each $t\in\mb D_\varepsilon$. 

Now, we want to prove the same property for an arbitrary $T\in\C$. We choose $n\in\mb N$ such that $t_1=\frac{T}{n}\in\mb D_\varepsilon$, and we consider the new column vector $\mc Y^1:=\mc Y(t_1)\in\mf h^m$ whose components define another basis of $\mf h$. 
Then $\mc Y^1(t)=\exp(tX)^*\mc Y^1$ fulfills equation~\eqref{eq_diff} with a  matrix coefficient $C^1$ corresponding to the new basis. Applying the previous argument, we conclude that the solution to this new equation with initial values $\mc Y^1(t_1)=\mc Y^1$ is defined for all $t$ such that $t-t_1\in \mb D_\varepsilon$ and belongs to $\mf h^m$. This implies in particular that $\mc Y^1(t_1) = \exp(2t_1X)^*\mc Y\in \mf h^m$, and therefore $Y(2t_1)=\exp(2t_1X)^* Y$ belongs to $\mf h$ for any given $Y\in\mf h$. Iterating that procedure $n$ times we conclude that $Y(T)=\exp(TX)^* Y$ belongs to $\mf h$.

The birational flow $u\mapsto\exp(uY(t))$ generated by the vector field $Y(t)=\exp(tX)^* Y$ is obtained by conjugation with the flow of $X$. Therefore, for any fixed $t\in\C$ and for every $u\in\C$, we have
\begin{equation}\label{commutacion}
\exp(-tX)\exp(uY)\exp(tX)=\exp(uY(t))\in H,
\end{equation}

We want to prove that the subset $G\subset\mr{Bir}(M)$, whose elements are of the form $h\circ s$ with $h\in H$ and $s\in S$, is a Lie subgroup of $\mr{Bir}(M)$.
Since connected Lie groups are generated by neighborhoods of the identity, we can write $s=\exp(X_1)\circ\cdots\circ\exp(X_n)$, $h=\exp(Y_1)\circ\cdots\circ\exp(Y_k)$ for suitable vector fields $X_i\in\mf s$ and $Y_j\in\mf h$. Applying the relation (\ref{commutacion}) we see that $h\circ s=s\circ h'_{s,h}$ for a certain $h'_{s,h}\in H$. This implies that $G$ is a subgroup of $\mr{Bir}(M)$ and that the map
\begin{equation}\label{conjugation}
S\to\mr{Aut}(H),\quad
s\mapsto(h\mapsto s^{-1}\circ h\circ s)
\end{equation}
is a well-defined group morphism. Hence, $H$ is a normal subgroup of $G$,  $G/H\cong S$, and $G$ is the semidirect product $G=H\rtimes S$ defined by the conjugation morphism~\eqref{conjugation}.
Finally, since the subgroups $S$ and $H$ are of bounded degree the same is true for $G$.

Thanks to Weil's regularization Theorem~\ref{Weil} there is a birational map $f:M'\to M$ such that $f^*G=f^*H\rtimes f^*S\subset\mr{Aut}(M')$. This implies in particular that $G$ is a Lie subgroup of $\mr{Bir}(M)$ and that $f^*\mf g$, which is contained in $\mf{aut}(M')$, is the Lie algebra of $f^*G$. Therefore $\mf g$ is regularizable, concluding the proof.
\end{proof}

An immediate consequence of
Theorem~\ref{bir-hol}, Theorem~\ref{XbirM}  and Theorem~\ref{semi} is the following.

\begin{cor}\label{solv}
Any finite-dimensional solvable Lie algebra generated by birationally integrable rational vector fields on a projective surface is regularizable.
\end{cor}

The following example shows the importance of the finite dimensionality hypothesis in the above results. In particular, it exhibits an infinite-dimensional Lie algebra $\mf g$ entirely contained in $\mf X_{\mr{bir}}(\p^2)$, which is not regularizable, being infinite-dimensional, and with the property that every finite-dimensional subalgebra of $\mf g$ is regularizable in a suitable surface.

\begin{ex}\label{fin-dim}
Using the birational identification of $\p^2$ with the Hirzebruch surfaces $\mb F_n$, we can think that the affine coordinates $(x,y)$ of $\p^2$ are coordinates of  $\mb F_n$ as well, assuming also  that $y$ is an affine coordinate of the fiber of the natural projection $\mb F_n \to \p^1$ and that $x$ is an affine coordinate of the base space.
Let us consider the infinite dimensional Abelian Lie algebra $\mf g=\C[x]\partial_y\subset\mf X_{\mr{rat}}(\p^2)\cong \mf X_{\mr{rat}}(\mb F_n)$. 
For a given polynomial $p(x) \in \mb C_n[x]$, the vector field $p(x)\partial_y\in \mf X_{\mr{bir}}(\p^2)$ is birationally integrable. In fact, it extends to a well-defined holomorphic vector field on $\mb F_m$, for each $m\geq n$ (cf. Proposition~\ref{campos_Fn}). Therefore $\mf g$ is contained  in $\mf X_{\mr{bir}}(\p^2)\cong \mf X_{\mr{bir}}(\mb F_n)$ although it is not regularizable, since it has infinite dimension. In contrast, the finite dimensional Lie subalgebra $\mf g_n =\C_n[x]\partial_y$ is contained in $\mf {aut}(\mb F_m)$ for each $m\geq n$. Note that $\mf g_n$ regularizes in $\p^2$ only if $n\leq 1$.
\end{ex}

\section{Normal forms of birationally integrable vector fields}

 Our purpose in this section is to obtain a birational classification of the birationally integrable vector fields on projective surfaces. For each class we provide a normal form that is given by a holomorphic vector field defined on a minimal surface.

We discuss first the case of vector fields on a minimal surface $S$ of non-negative Kodaira dimension. According to Corollary~\ref{tot_hol} we have $\mf X_{\rm{bir}}(S) = \mf{aut}(S)$ and,  due to the uniqueness of the minimal model of such a surface,  the birational classification of the elements of $\mf{aut}(S)$ coincides with its holomorphic classification. This classification is trivial in the case where
 $\dim\mf{aut}(S)=1$, which correspond to bielliptic or elliptic surfaces with non-trivial vector fields. If $\dim\mf{aut}(S)>1$ then $S$ is an Abelian surface and the holomorphic classification reduces to studying the action of the discrete group $\mr{Aut}(S)/\mr{Aut}_0(S)$ on $\mf{aut}(S)=\C^2$, which we do not consider here. Therefore, we concentrate on minimal surfaces of negative Kodaira dimension.

\medskip

Before presenting the main result of this section, let us recall the construction of surfaces obtained by suspension of a representation. Let $C$ be a curve and let $\rho:\pi_1(C)\to\mr{Aut}(\p^1)\cong \mr{PSL}_2(\C)$ be a group morphism. The representation $\rho$ defines a flat  $\p^1$-bundle $S_\rho$ over the curve $C$, called the suspension of $\rho$, as well as a foliation $\F_\rho$ on this ruled surface $S_\rho$, which is transverse to the natural projection $\pi:S_\rho\to C$. Indeed, if $\tilde C$ is the universal covering of $C$ then $S_\rho=(\tilde C\times \p^1)/\sim_\rho$, where $\sim_\rho$ is the equivalence relation defined by $(\tilde c,y)\sim_\rho(\gamma\tilde c,\rho(\gamma^{-1})(y))$ for $\gamma\in\pi_1(C)$, and $\F_\rho$ is the quotient of the horizontal foliation on $\tilde C\times \p^1$. Given a holomorphic vector field $X$ on $C$ there is a unique holomorphic vector field $X_\rho$ on $S_\rho$ tangent to $\F_\rho$ and projecting onto $X$. Since $\pi_1(C)=0$ if $g(C)=0$ and $\mf{aut}(C)=0$ if $g(C)>1$, we only consider here the case of an elliptic curve $C=\C/\Lambda$. In that situation, there exists a non-trivial vector field $X_\rho$ tangent to $\F_\rho$. It is unique up to a multiplicative constant and it does not vanish. We say that this vector field $X_\rho$ is obtained by suspension of the representation $\rho$.  Two representations $\rho$ and $\rho'$ are conjugated if and only if the models $(S_\rho,X_\rho)$ and  $(S_{\rho'},X_{\rho'})$ are biholomorphic.

It is known (cf. \cite{Suwa,LM}) that the  flat $\p^1$-bundles $S_\rho\to C$ are of the form $\p A_1$, $\p A_0$, $\p E_\tau$ and $C\times\p^1$,  where $A_1$ and $A_0$ are the two undecomposable rank $2$ vector bundles over~$C$ and $E_\tau=\mc O_C(\tau-\mr{o})\oplus\mc O_C$ with $\tau\in C\setminus\{\mr{o}\}$. We notice that all the non-vertical vector fields described in Theorem~\ref{Xhol-ruled_1} are obtained by suspension.

\medskip

Now, we classify birationally integrable vector fields defined on rational and ruled surfaces under birational equivalence. The holomorphic models that we obtain are all defined on ruled surfaces of the types $C\times\p^1$ or $S_\rho$, for a certain representation $\rho$ of the fundamental group of an elliptic curve. We denote by $y$ an affine coordinate of the fiber $\p^1$.

\begin{teo}\label{X1}
Let $X$ be a birationally integrable vector field on a projective surface $S$ of negative Kodaira dimension. Then, up to a constant multiple of the vector field, the pair $(S,X)$ is birationally equivalent to one, and only one, of the following five holomorphic models defined on a ruled surface over a curve $C$:
\begin{enumerate}[ (a)]
\item $(C\times\p^1, T)$, where $T=\partial_y$,
\item $(C\times\p^1, L)$, where $L=y\partial_y$,
\item $(\p^1\times\p^1, J)$, where $J=\partial_x+y\partial_y$,
\item $(\p^1\times\p^1, H_\gamma)$, where $H_\gamma=x\partial_x+\gamma y\partial_y$ with $\gamma\in\C\setminus \mb Q$,
\item $(S_\rho, X_\rho)$, where $S_\rho$ and $X_\rho$ are obtained by suspension of a representation $\rho:\Lambda\to\mr{PSL}_2(\C)$ of the fundamental group $\Lambda=\pi_1(C)$ of an elliptic curve $C=\C/\Lambda$.
\end{enumerate} 
If $X$ is tangent to a rational fibration then it is birationally equivalent to $T$ or $L$.
\end{teo}

\begin{obs}\label{no-elliptic}
From the previous theorem it follows that there are no birationally integrable vector fields on rational surfaces that are tangent to an elliptic fibration. 
\end{obs}

\begin{obs}
The classification of the elements of $\mf{aut}(\p^2)$ under biholomorphic equivalence was discussed in Proposition~\ref{campos_P2}. It was noticed in Remark~\ref{rad2} that the vector fields $T,N,J$ and $H_\gamma$ considered there define holomorphic vector fiels in $\mb F_n$. We also note that the vector fields $T$ and $L$ are well-defined holomorphic vector fields on any ruled surface of type $C\times \p^1$. The vector field $N=\partial_x+x\partial_y$ of Proposition~\ref{campos_P2} does not appear in the above theorem because
$f^\ast N = T$, where $f\colon \p^1\times\p^1 \to \p^2$ is the birational map given by $f(x,y)=(y,x+{y^2/2})$.
\end{obs}

\medskip

The following statement provides holomorphic normal forms of the elements of $\mf{aut}(\mb F_n)$. It is the first step in proving Theorem~\ref{X1} and will also be used in the next section.
Recall that we denote by $\mf B_n$ the Borel subalgebra of $\mf{aut}(\mb F_n)$ given as  $\mf B_n=\C_1[x]\partial_x\oplus\C_n[x]\partial_y\oplus\C y\partial_y$.

\begin{lema}\label{fnQ}
For every $X\in\mf B_n\subset\mf{aut}(\mb F_n)$ there exists  $\phi\in\mr{Aut}(\mb F_n)$ such that $\phi^*\mf B_n=\mf B_n$ and
$\phi^*X$ is a constant multiple of one of the following vector fields of $\mb F_n$:
\[\partial_x+\varepsilon x^n\partial_y,\quad J=\partial_x+y\partial_y,\quad H_\gamma=x\partial_x+\gamma y \partial_y,\quad R_m=x\partial_x+(my+x^m)\partial_y,\quad L=y\partial_y,\quad p(x)\partial_y,\]
where $\varepsilon\in\{0,1\}$, $\gamma\in\C$, $m\in\{0,1,\ldots,n\}$ and $p\in\C_n[x]$. 

Moreover,
$\partial_x+\varepsilon x^n\partial_y$ and $p(x)\partial_y$ are birationally equivalent to $T=\partial_y$, and $R_m$ is birationally equivalent to $J$.
\end{lema}
\begin{proof}
If $X$ is not vertical, taking the pull-back of $X\in \C_1[x]\partial_x\oplus\C_n[x]\partial_y\oplus\C y\partial_y$ by the map $(x,y)\mapsto (x+a,y)$, we can assume that $X$ is a constant multiple of $x^\mu\partial_x+(\gamma y+p(x))\partial_y$, with $\mu\in\{0,1\}$, $\gamma\in\C$ and $p\in\C_n[x]$. The automorphism $\phi_q$ of $\mb F_n$ defined by $\phi_q(x,y) =(x, y+q(x))$, where $q\in\C_n[x]$, fulfills $\phi^*\mf B_n=\mf B_n$. And we have
\[
Y:=\phi_q^*\big(x^\mu\partial_x+(\gamma y+p(x))\partial_y\big)=x^\mu\partial_x+(\gamma y+p(x)+\Phi_{\mu,\gamma}(q(x)))\partial_y,
\]
where $\Phi_{\mu,\gamma}:\C_n[x]\to\C_n[x]$ is the linear map defined by $\Phi_{\mu,\gamma}(q(x))=\gamma q(x)-x^\mu q'(x)$. We want to see that, by choosing $q$ appropriately, the vector field $Y$ is in one of the normal forms listed in the statement of the lemma.

Notice first that if $\gamma\neq 0$ then $\Phi_{0,\gamma}$ is injective, hence surjective. This implies that in the case $\mu=0$ and $\gamma\neq 0$, choosing conveniently $q$ we can assume that $Y=\partial_x+\gamma y\partial_y$. By rescaling the variable $x$ we can additionally assume that $\gamma=1$, and we obtain $Y=J$.

Secondly, $\ker \Phi_{0,0}=\C$ and $\mr{Im}\,\Phi_{0,0}\oplus\langle x^n\rangle=\C_n[x]$. Hence, in the case $\mu=\gamma=0$ we can assume that $Y=\partial_x+\varepsilon x^n\partial_y$ with $\varepsilon\in\C$. Again, by rescaling the coordinate $x$ we can assume that $\varepsilon\in\{0,1\}$.

Let us analyze now the case $\mu=1$. The solutions of the differential equation $xq'(x)=\gamma q(x)$ are $q(x)=c x^\gamma$. Therefore,
if $\gamma\notin\{0,\ldots,n\}$ then $\ker\Phi_{1,\gamma}=0$ and we can assume that $Y=x\partial_x+\gamma y\partial_y = H_\gamma$. 
Finally, if $\gamma=m\in\{0,\ldots,n\}$ then $\ker\Phi_{1,m}=\langle x^m\rangle $ and $\mr{Im}\,\Phi_{1,m}\oplus\langle x^m\rangle=\C_n[x]$. In that case we can assume that $Y=x\partial_x+(my+cx^m)\partial_y$. If $c\neq 0$, up to rescaling the coordinate $x$, we can assume that $c=1$,  and that $Y=R_m$.

It only remains to deal with vertical vector fields $X=(\gamma y+p(x))\partial_y\in\mf B_n$. If $\gamma\neq 0$ we define $\phi(x,y)=(x,y-p(x)/\gamma)$. Then $\phi^*X=\gamma y\partial_y$ is a constant multiple of $L$.

Let us prove now the second assertion. We observe first that, if we define $\phi(x,y)=\big(x,y+\varepsilon\frac{x^{n+1}}{n+1}\big)$, then
$\phi^*(\partial_x+\varepsilon x^n\partial_y)=\partial_x$, which is birationally equivalent to $\partial_y=T$ by exchanging the coordinates $x$ and $y$. Finally, $f_1(x,y) = (x,yp(x))$ is a birational map fulfilling $f_1^*(p(x)\partial_y)=T$, and the birational map $f_2(x,y) = (y,y^mx)$ fulfills $f_2^*R_m=J$.
\end{proof}

\medskip

\begin{obs}\label{everywhere}
Let $X$ be a given holomorphic vector field on $\mb F_n$. Lemma~\ref{solv_S} implies that there exists $f\in\operatorname{Aut}(\mb F_n)$ such that $f^\ast X\in \mf B_n$. We deduce from Lemma~\ref{fnQ} that $X$ is birationally equivalent to one of the holomorphic vector fields $T, L, J, H_\gamma$ on $\mb F_n$. We note also that 
 $T, L, J, H_\gamma$ can be considered as a holomorphic vector fields on $\mb F_0=\p^1\times\p^1$ (cf. Remark~\ref{rad}).
\end{obs}

\medskip

\begin{proof}[Proof of Theorem~\ref{X1}]
By Theorem~\ref{bir-hol} we can assume that $X$ is a holomorphic vector field defined  on a minimal ruled surface over a curve $C$ of genus $g(C)\ge 0$ or on $\p^2$. By Proposition~\ref{campos_P2} and Remark~\ref{rad2} 
we can exclude the case of $\p^2$ and assume that $X$ is defined on a ruled surface. According to \cite[Chapter~6, Proposition 6]{Bru} we have the following possibilities depending on the value of $g$:
\begin{itemize}
\item If $g(C)>1$, then $X$ is tangent to the rational fibration. In this case, Theorem~\ref{fibration_2} implies that $X$ is birationally equivalent to a constant multiple of $y^\nu\partial_y$, with $\nu\in\{0,1\}$, on $C\times\p^1$.
\item If $g(C)=1$, then either $X$ is tangent to the rational fibration (in which case $X$ is birationally equivalent to $y^\nu\partial_y$, with $\nu\in\{0,1\}$, on $C\times\p^1$ due to Theorem~\ref{fibration_2} again) or $X$ is the suspension of a representation $\rho:\pi_1(C)\to\mr{PSL}_2(\C)$.
\item If $g(C)=0$, then $X$ is birationally equivalent to a holomorphic vector field on $\p^1\times\p^1$.  In this case, it follows from Lemma~\ref{fnQ} and Remark~\ref{everywhere} that $X$ is birationally equivalent to $T$, $J$ or $H_\gamma$ for some $\gamma\in\C$. 
If $\gamma=\frac{p}{q}\in\mb Q$ then $H_\gamma$ is tangent to the rational fibration defined by $x^p/y^q$. In fact, in this case $H_\gamma$ is birationally conjugated to $L$ (cf. Proposition~\ref{pe} below).
\end{itemize}
It remains to prove that the five models $T=\partial_y$, $L= y\partial_y$, $J=\partial_x+y\partial_y$, $H_\gamma=x\partial_x+\gamma y\partial_y$ with $\gamma\in\C\setminus\mb Q$ and $X_\rho$ are not pairwise birationally equivalent.
Clearly, the holomorphic vector fields $\partial_y$ and $ y\partial_y$ on $\p^1$ are not holomorphically equivalent. Since the holomorphic and birational equivalences coincide in dimension one, $T$ and $L$ are not birationally equivalent. Moreover, the vector fields $y^\nu\partial_y$, $\nu\in\{0,1\}$, are tangent to a rational fibration, contrarily to the vector fields $J$ and $H_\gamma$ with $\gamma\notin\mb Q$. On the other hand, $J$ possesses saddle-node singularities contrarily to $H_\gamma$. This property is invariant by birational equivalence, cf. \cite[Chapter~3]{CCD}.
Finally, the suspension over an elliptic curve is not tangent to a rational fibration and it is not birationally equivalent to any vector field on $\p^1\times\p^1$.
\end{proof}

Each of the model classes (d) and (e) in Theorem~\ref{X1} contain vector fields that are birationally equivalent. This is discussed in the following Proposition~\ref{pe} and Theorem~\ref{pd}.

\begin{prop}\label{pe} 
The vector field $H_\gamma$ is birationally equivalent to a constant multiple of $H_{\gamma'}$ if and only if there is $\left[\begin{array}{cc}a & b\\ c & d\end{array}\right]\in\mr{PGL}_2(\Z)$ such that $\gamma'=\frac{a\gamma+b}{c\gamma+d}$ in which case $H_{\gamma'}$ is a multiple of the pullback of $H_\gamma$ by the monomial birational map $(x,y)\mapsto(x^ay^{-c},x^{-b}y^d)$.
\end{prop}

\begin{proof}
If $\gamma'=\frac{a\gamma+b}{c\gamma+d}$ with $ad-bc=\pm 1$ then the monomial map $\phi(x,y)=(x^ay^{-c},x^{-b}y^d)$ is birational and $\phi^*H_\gamma$ is a constant multiple of $H_{\gamma'}$. If $\gamma\in\mb Q$ then $\gamma$ is related to $0$ via $\mr{PGL}_2(\Z)$. Indeed, there are $a,b,c,d\in\Z$ with $\gcd(b,d)=1$ such that $ad-bc=1$ and $\gamma=\frac{b}{d}=\frac{a\cdot 0+b}{c\cdot 0+d}$. Thus, if $\gamma\in\mb Q$ then $H_\gamma$ is tangent to a rational fibration.
On the other hand, looking at the first integral $xy^{-\gamma}$ of $H_{\gamma}$ we see that if $\gamma\notin\mb Q$ then $H_\gamma$ is not tangent to a rational fibration. 

Assume now that 
$\phi:\p^1\times\p^1\to\p^1\times\p^1$ is a birational map conjugating the foliations defined by  $H_\gamma$ and $H_{\gamma'}$. 
Hence if $\gamma\in\mb Q$ and $H_\gamma$ is birationally equivalent to a constant multiple of $H_{\gamma'}$ then $\gamma'\in\mb Q$ and $\gamma$ and $\gamma'$ are related to $0$ via $\mr{PGL}_2(\Z)$.
From now on we assume that  $\gamma$ and $\gamma'$ are not rational, so that the foliations defined by $H_\gamma$ and $H_{\gamma'}$ do not have dicritical singularities.
Then the only invariant algebraic curves of $H_\gamma$ and $H_{\gamma'}$ are $I=\{xy=0\}\cup\{xy=\infty\}$ and the critical locus of $\phi$ (resp. $\phi^{-1}$) is  invariant by $H_\gamma$ (resp. $H_{\gamma'}$), hence it is contained in $I$. Thus, the restriction of $\phi$  to $\p^1\times\p^1\setminus I=\C^*\times\C^*$ is a biholomorphism and
we deduce as in \cite[\S 5.1]{CF} that $\phi$ is a monomial birational map $\phi(x,y)=(x^ay^{-c},x^{-b}y^d)$ with $ad-bc=\pm1$. Consequently $\gamma'=\frac{a\gamma+b}{c\gamma+d}$.
\end{proof}

\begin{teo}\label{pd}
Let $C=\C/\Lambda$ be an elliptic curve and let $\rho,\rho':\Lambda\to \mr{PSL}_2(\C)$ be two representations. 
Then $X_\rho$ and $X_{\rho'}$ are birationally conjugated if and only if $\rho$ and $\rho'$ are biholomorphically conjugated, that is, there is $u\in\C^*$ such that $u\Lambda=\Lambda$ and $h\in\mr{PSL}_2(\C)$ such that $\rho(u\,\cdot)=h\circ\rho'(\cdot)\circ h^{-1}$. Moreover, a given suspension $(S_\rho,X_\rho)$
 is biholomorphically equivalent to one of the following models:
\begin{enumerate}[ (a)]
\item $\big(C\times\p^1,\partial_x+\partial_y\big)$ 
or $\big(\p A_0,\partial_x+(c-\wp(x))\partial_y\big)$ where $c\in\C$,
in case $\rho(\Lambda)$ is conjugated to a non-trivial subgroup of $(\C,+)$,
\item $\big(C\times\p^1,\partial_x+cy\partial_y\big)$
or $\big(\p E_\tau,\partial_x+(c-\wp_\tau(x))y\partial_y\big)$ where $c\in\C$ and $\tau\in C\setminus\{\mr{o}\}$, in case $\rho(\Lambda)$ is conjugated to a subgroup of $(\C^*,\cdot)$,
\item $\big(\p A_1,\partial_x+(\frac{y^2}{2}-\frac{3}{2}\wp(x))\partial_y\big)$, in case $\rho(\Lambda)$ is conjugated to the dihedral group $\big\langle -y,\frac{1}{y}\big\rangle\subset\mr{PSL}_2(\C)$.
\end{enumerate}
\end{teo}

\begin{proof}
We recall that the group of automorphisms of $C$ fixing the origin $\mr{o}\in C$ is the cyclic group
$\Gamma=\big\langle e^{\frac{2i\pi}{m_C}}\big\rangle\subset\C^*$, where $m_C\in\{2,4,6\}$.
Assume that $X_\rho$ and $X_{\rho'}$ are birationally conjugated by some $\phi:S_\rho\to S_{\rho'}$. Then
$\phi(x,y)=\left(ux+v,\frac{a(x)y+b(x)}{c(x)y+d(x)}\right)$, with $u\in\Gamma$, and $\phi$ is a biholomorphism conjugating $(\pi^{-1}(U),X_\rho)$ and $(\pi'^{-1}(U),X_{\rho'})$ over the Zariski open subset
\[
U=C\setminus\bigcup\limits_{k=1}^{ m_C}u^k\Big((a)_\infty\cup(b)_\infty\cup(c)_\infty\cup(d)_\infty\cup(ad-bc)_0\Big)
\] 
of $C$. Therefore, their holonomy representations $\rho\circ\iota$ and $\rho'\circ\iota$, where $\iota:\pi_1(U)\to\pi_1(C)$ is the morphism induced by the inclusion, are biholomorphically conjugated, i.e. there exists $h\in\mr{PSL}_2(\C)$ such that the following diagram is commutative:
$$\xymatrix{\pi_1(U,x_0)\ar@{->>}[r]^{\iota}\ar[d]^{\cdot u}&\pi_1(C,x_0)\cong\Lambda\ar[r]^{\rho}\ar[d]^{\cdot u} &\mr{PSL}_2(\C)\ar[d]^{h_*}\\ \pi_1(U,ux_0+v)\ar@{->>}[r]^{\iota\hphantom{AA}}&\pi_1(C,ux_0+v)\cong \Lambda\ar[r]^{\hphantom{AAA}\rho'}&\mr{PSL}_2(\C)}$$
where $h_*(g)=h\circ g\circ h^{-1}$.
Since $\iota$ is surjective we deduce that $\rho$ and $\rho'$ are biholomorphically conjugated, i.e.  there exist $u\in\Gamma$ such that $u\Lambda=\Lambda$ and $h\in\mr{PSL}_2(\C)$ such that $\rho(u\omega)=h\circ\rho'(\omega)\circ h^{-1}$ for each $\omega\in\Lambda$.

We set $\Lambda=\mb Z\omega_1\oplus\mb Z\omega_2\subset\C$ and we recall that the Weierstrass zeta function $\zeta(x)=-\int\wp(x)dx$ satisfies  $\zeta(x+\omega_j)=\zeta(x)+\eta_j$, with
\begin{equation}\label{determ}
\left|\begin{array}{cc}\omega_1 & \eta_1\\ \omega_2 & \eta_2\end{array}\right|=-2i\pi,
\end{equation}
see \cite[\S18.2.19 and \S18.3.37]{Abramowitz-Stegun}.
Recall also that $\zeta$ can be written as the quotient
$\zeta=\sigma'/\sigma$,
where $\sigma$ is the Weierstrass sigma function that fulfills 
\begin{equation}\label{sigma_add}
\sigma(x+m_1\omega_1+m_2\omega_2)=(-1)^{m_1+m_2+m_1m_2}\sigma(x)e^{\left(x+\frac{m_1\omega_1+m_2\omega_2}{2}\right)(m_1\eta_1+m_2\eta_2)}\end{equation}
for all $m_1,m_2\in\Z$ (cf. \cite[\S18.2.20]{Abramowitz-Stegun}).

According to \cite[Proposition~3.1]{LM} any representation $\pi_1(C)=\Lambda\to\mr{PSL}_2(\C)$ is conjugated to one of the following ones:
\begin{enumerate}[ (a)]
\item $\rho:\Lambda\to\C\cong\{y+c \mid c\in\C\}\subset \mr{PSL}_2(\C)$ with $\rho\not\equiv 0$ and either $S_\rho=\p A_0$ or $S_\rho=C\times\p^1$. Moreover, we are in the last situation if and only if $\rho$ is conjugated to the inclusion $\Lambda\subset\C$.
\item $\rho:\Lambda\to\C^*\cong\{ay \mid a\in\C^*\}\subset\mr{PSL}_2(\C)$ and $S_\rho=\p(\mc O_C\oplus L)$ with $L\in\mr{Pic}_0(C)$. Moreover, $L=\mc O_C$ if and only if $\rho(\omega)=\exp(c\omega)$.
\item  $\rho:\Lambda\twoheadrightarrow \big\langle -y,\frac{1}{y}\big\rangle\subset\mr{PSL}_2(\C)$ and
$S_\rho=\p  A_1$.
\end{enumerate}
Let us analyze each case separately:
\begin{enumerate}[ (a)]
\item  We notice first that  any two additive representations $\Lambda\to \C$ differing by a multiplicative non-zero constant are conjugated so that they give rise to biholomorphic models. We consider two cases.
If $\rho$ is conjugated to the inclusion $\Lambda\subset\C$ we can assume that $\rho(\omega)=\omega$,
so that $S_\rho=C\times\p^1$ and $X_\rho=\partial_x+\partial_y$. Otherwise 
the vectors $(\omega_1,\omega_2),(\rho(\omega_1),\rho(\omega_2))\in\C^2$ are linearly independent and there exist $c,\lambda\in\C$ such that $\lambda\rho(\omega_j)=c\omega_j+\eta_j$ for $j=1,2$, where $\eta_j$ are defined by  $\zeta(x+\omega_j)=\zeta(x)+\eta_j$.
Since the determinant in~\eqref{determ} is different from zero we deduce that $\lambda\neq 0$. Hence $\lambda\rho$ is conjugated to $\rho$ and we can assume that $\lambda=1$.
In this case $S_\rho=\p A_0$ and $X_\rho=\partial_x+(c-\wp(x))\partial_y$ because its flow $\varphi_t(x,y)=\big(x+t,y+ct+\zeta(x+t)-\zeta(x)\big)$ satisfies $\varphi_{\omega_j}(x,y)=\big(x,y+\rho(\omega_j)\big)$.
\item Since the determinant in \eqref{determ} is $-2i\pi\neq0$ there exists $k,\tilde\tau\in\C$ such that
\[(\log\rho(\omega_1),\log\rho(\omega_2))=k(\omega_1,\omega_2)-\tilde\tau(\eta_1,\eta_2).\]
In fact,
\[\left(\begin{array}{c}k\\-\tilde\tau\end{array}\right)=-\frac{1}{2i\pi}\left(\begin{array}{cc}\eta_2 & -\eta_1\\ -\omega_2 &\omega_1\end{array}\right)\left(\begin{array}{c}\log\rho(\omega_1)\\ \log\rho(\omega_2)\end{array}\right)\]
and $\tilde\tau\in\Lambda$ if and only if
$\log\rho(\omega_j)=c\omega_j+2i\pi n_j$ with $c\in\C$ and $n_j\in\Z$. 
Indeed, if $\tilde\tau=n_2\omega_1-n_1\omega_2$ with $n_j\in\Z$, denoting $\ell_j=\log\rho(\omega_j)$ we have that
\[-n_2\omega_1+n_1\omega_2=-\tilde\tau=-\frac{1}{2i\pi}(-\omega_2\ell_1+\omega_1\ell_2)\]
and 
$\omega_2(\ell_1-2i\pi n_1)=\omega_1(\ell_2-2i\pi n_2)$.
The complex value $c:=\frac{\ell_1-2i\pi n_1}{\omega_1}=\frac{\ell_2-2i\pi n_2}{\omega_2}$
satisfies the equalities $\ell_j=c\omega_j+2i\pi n_j$.
In that case
 $\rho(\omega)=\exp(c\omega)\in\C^*$, $S_\rho=C\times\p^1$ and $X_\rho=\partial_x+cy\partial_y$ because its flow $\varphi_t(x,y)=(x+t,ye^{tc})$ satisfies $\varphi_{\omega_j}(x,y)=(x,\rho(\omega_j)y)$. 
On the other hand, according to  (\ref{wptau}) and (\ref{wp}), the flow of the vector field $\partial_x+(c-\wp_\tau(x))y\partial_y$
 is given by
\begin{equation}\label{phi_t}
\varphi_t(x,y)=\left(x+t,ye^{(c+\zeta(\tilde\tau))t}\frac{\sigma(x-\tilde\tau+t)\sigma(x)}{\sigma(x-\tilde\tau)\sigma(x+t)}\right),
\end{equation}
where $\sigma$ is the Weierstrass sigma function. 
Indeed, since $-\wp_\tau(x)=\zeta(x-\tilde\tau)-\zeta(x)+\zeta(\tilde\tau)$ and  $x_t=x+t$ we have that $\frac{d\log y_t}{dt}=c-\wp_\tau(x+t)=c+\zeta(\tilde\tau)+\zeta(x+t-\tilde\tau)-\zeta(x+t)$. Using that $\zeta=\frac{\sigma'}{\sigma}$ we obtain the equality $\log(\frac{y_t}{y})=(c+\zeta(\tilde \tau))t+\log\left(\frac{\sigma(x+t-\tilde\tau)}{\sigma(x-\tilde\tau)}\right)-\log\left(\frac{\sigma(x+t)}{\sigma(x)}\right)$.
If $\tilde\tau\notin\Lambda$  then we can 
evaluate $\zeta(\tilde\tau)$ and taking $c:=k-\zeta(\tilde\tau)\in\C$ the  flow (\ref{phi_t}) satisfies $\varphi_{\omega_j}(x,y)=\big(x,y\exp((c+\zeta(\tilde\tau))\omega_j-\tilde\tau\eta_j)\big)=(x,y\rho(\omega_j))$, due to the formula~\eqref{sigma_add}  fulfilled by $\sigma$.
Hence, in this situation $S_\rho=\p E_\tau$ and $X_\rho=\partial_x+(c-\wp_\tau(x))y\partial_y$.

\item If $\rho:\Lambda\twoheadrightarrow \big\langle -y,1/y\big\rangle\subset\mr{PSL}_2(\C)$ then
$S_\rho=\p  A_1$ and $X_\rho= \partial_x+\frac{1}{2}(y^2-3\wp(x))\partial_y$ according to the first assertion of Theorem~\ref{Xhol-ruled_1}. 
\end{enumerate}
\end{proof}

Next remark discusses whether a  vector field obtained by suspension is tangent to an elliptic fibration. 

\begin{obs}\label{el-fib}
The vector field $X_\rho$ induced by a representation $\rho\colon\Lambda \to \mr{PSL}_2(\C)$ is tangent to a Seifert elliptic fibration if and only if $\rho(\Lambda)$ is finite. This certainly occurs in case (c), and never occurs in case (a). Let us now consider the case~(b). There are two possibilities, if $\rho(\omega)=\exp(c\omega)$ then $\rho(\Lambda)=\exp(c\Lambda)$ is finite if and only if $c=0$, hence the vector field $\partial_x+cy\partial_y$ is tangent to an elliptic fibration if and only if $c=0$. The second case corresponds to the vector field $X_{\tau,c}:=\partial_x+(c-\wp_\tau(x))y\partial_y$ on $\p E_\tau$. According to  \cite[Theorem 4]{Maruyama}, $\p E_\tau$ admits an elliptic fibration if and only if $\tau$ is a torsion point of the elliptic curve $C=\C/\Lambda$. Moreover, for each torsion point $\tau$ of $C$ there is at most one value of $c$ for which $X_{\tau,c}$ is tangent to the elliptic fibration. Indeed,  if  $(c+\zeta(\tau))\omega_j-\tau\eta_j$ and $(c'+\zeta(\tau))\omega_j-\tau\eta_j$ belong to $2i\pi\Q$ then $(c-c')\omega_j\in 2i\pi\Q$ and consequently $c-c'=0$.
\end{obs}

\section{Two-dimensional Lie algebras of birationally integrable vector fields}

In this section we consider two-dimensional birationally integrable algebras of vector fields on a projective surface. We classify these algebras up to birational equivalence and, since they are regularizable, we provide holomorphic normal forms of them.

We will write $S\simeq S'$ to indicate that two surfaces $S$ and $S'$ are birationally equivalent and we will keep the notation $S\cong S'$ to denote that these surfaces are biholomorphic. In a similar way we will write $X\simeq X'$ (resp. $\mf g\simeq\mf g'$) to indicate that the vector fields $X$ and $X'$ (resp. Lie algebras) are birationally conjugated.

\smallskip

Let $\mf g$ be a two-dimensional Lie subalgebra of $\mf{X}_{\mr{rat}}(S)$ generated by birationally integrable vector fields. Since $\mf g$ is solvable, Corollary~\ref{solv} implies that it is regularizable in a surface that can be chosen minimal. Whence, we can assume that $\mf g\subset \mf{aut}(S)$. As discussed in Section~\ref{Shol}, the only surfaces of non-negative Kodaira dimension with $\dim\mf{aut}(S)\geq 2$ are tori, and in that situation $\mf g$ is just the Abelian Lie algebra $\mf{aut}(S)$ itself. Therefore, we will further assume that the surfaces we are dealing with have Kodaira dimension $-\infty$. 

We notice that the vector fields of a solvable Lie subalgebra $\mf g$ of $\mf{aut}(\p^2)$ have a common zero. The blow-up of the point defines a birational map between $\mb{F}_1$ and $\p^2$ that identifies $\mf g$ with a subalgebra of $\mf{aut}(\mb{F}_1)$. By that reason,
Theorem~\ref{2dim} presents the models of two-dimensional algebras of vector fields as defined on ruled surfaces, omiting $\p^2$ but including $\mb{F}_1$, although this is not a minimal surface. Anyway, we first detail both the holomorphic and the birational classifications of the two-dimensional subalgebras of  $\mf{aut}(\p^2)$.

\medskip

\begin{lema}\label{Lie}
Let $\mf g$ be a Lie subalgebra of dimension $2$ of $\mf{aut}(\p^2)$.
\begin{enumerate}[ (a)]
\item If $\mf g$ is Abelian then it is linearly conjugated to one 
of the following types:
\begin{enumerate}[(a1)]
\item[(a1)] $\langle x^\mu\partial_x,y^\nu\partial_y\rangle$,  $0\le\mu\le\nu\le 1$;
\item[(a2)] $\langle \partial_x+x\partial_y,\partial_y\rangle$;
\item[(a3)] $\langle \partial_y,x\partial_y\rangle$.
\end{enumerate}
\item If $\mf g $ is affine then it is linearly conjugated to one 
of following types:
\begin{enumerate}[(b1)]
\item $\langle x\partial_x+\gamma y\partial_y,\partial_y\rangle$ with $\gamma\in\C^*$;
\item $\langle x\partial_x+2y\partial_y,\partial_x+x\partial_y\rangle$;
\item $\langle \varepsilon\partial_x+y\partial_y,\partial_y\rangle$ with $\varepsilon\in\{0,1\}$;
\item $\langle x\partial_x+(x+y)\partial_y,\partial_y\rangle$.
\end{enumerate}
\end{enumerate}
\end{lema}
\begin{proof}
Assertion (b) is proved in  \cite[Lemma~2.2]{BM}.
In a similar way one can prove assertion (a)  using the isomorphism (\ref{Phi}) and the Jordan normal forms of matrices.
\end{proof}
\begin{obs}\label{Nalg}
The birational map
$\phi(x,y)=(y-x^2/2,x)$ gives a conjugation between the algebras of types (a1), with $\mu=\nu=0$, and (a2), respectively of types (b1), with $\gamma=\frac{1}{2}$, and (b2). 
It follows from (\ref{XFn}) that all the algebras in the above lemma can be seen as
subalgebras of $\mf{aut}(\mb F_1)\hookrightarrow\mf{aut}(\p^2)$.
\end{obs}

\begin{teo}\label{2dim}
Let $\mf g$ be a $2$-dimensional birationally integrable Lie subalgebra of $\mf X_{\mr{rat}}(S)$, where $S$ is a projective surface of Kodaira dimension $-\infty$. Then the following assertions hold:
\begin{enumerate}[ (1)]
\item If $\mf g$ is spanned by non-collinear vector fields then 
 $(S,\mf g)$ is birationally equivalent 
 to one 
 of the holomorphic models 
 listed in  Table~\ref{tabla}.
\item If all the vector fields of $\mf g$ are collinear then there is a projective curve $C$ such that
\begin{enumerate}[(a)]
\item $\mf g\simeq \langle \partial_y,h(x)\partial_y\rangle\subset \mf{aut}(S_h)$, for some $h\in\C(C)$ and where $S_h=\p\big(\mc O_C((h)_\infty)\oplus\mc O_C\big)$, in the case where $\mf g$ is Abelian,
\item $\mf g\simeq \langle \partial_y,y\partial_y\rangle\subset\mf{aut}(C\times\p^1)$, in the case where $\mf g$ is affine. 
\end{enumerate}
\end{enumerate}
\end{teo}

\begin{table}[h]
\begin{center}
\renewcommand{\arraystretch}{1.2}
\begin{tabular}{|c|| c | c |}
\hline
surface & $\mb F_n$ with $n\ge 0$ & $\C/\Lambda\times\p^1$,  $\p A_0$ or $\p E_\tau$  \\
\hline
$\begin{array}{c}
\text{Abelian}\\ \text{Lie}\\ \text{algebra}\end{array}$ 
&$\begin{array} {l}
\big( \mb F_0\,,\, \mf a_{00}^0=\langle \partial_x,\partial_y\rangle \big)\\ 
\big( \mb F_0\,,\, \mf a_{01}^0=\langle \partial_x,y\partial_y\rangle\big)\\ 
\big( \mb F_0\,,\, \mf a_{11}^0=\!\langle x\partial_x,y\partial_y\rangle\big)
\end{array}$ 
& $\begin{array}{l} \big( \C/\Lambda\times\p^1 \,,\,   \mf a_{00}^1=\langle \partial_x,\partial_y\rangle\big)\\ 
\big( \C/\Lambda\times\p^1  \,,\, \mf a_{01}^1=\langle \partial_x,y\partial_y\rangle\big)\\ 
\big( \p A_0   \,,\, \mf b_{\mr{o}}= \langle \partial_x-\wp(x)\partial_y,\partial_y\rangle \big)\\ 
\big(\p E_\tau  \,,\,  \mf b_\tau= \langle \partial_x-\wp_\tau(x)y\partial_y,y\partial_y\rangle \big)
\end{array}$ 
\\
\hline
$\begin{array}{c}\text{affine}\\ \text{Lie}\\ \text{algebra}\end{array}$  
& $\begin{array}{ll} \big( \mb F_0\,,\,\mf c_\gamma=\langle x\partial_x+\gamma y\partial_y,\partial_y\rangle\big)
&\\ \big( \mb F_0\,,\, \mf d_{\phantom{\gamma}}=\langle \partial_x+y\partial_y,\partial_y\rangle\big)
&\\ \big( \mb F_n\,,\, \mf f_n=\langle x\partial_x+(ny+x^n)\partial_y,\partial_y\rangle \big)
&\end{array}$ 
& $\big( \C/\Lambda\times\p^1\,,\,\mf e_c=\langle \partial_x+c y\partial_y,\partial_y\rangle\big)$ 
\\
\hline
\end{tabular}
\end{center}
\caption{Holomorphic normal forms of $2$-dimensional birationally integrable Lie algebras on projective surfaces with $\mr{kod}(S)=-\infty$ generated by non-collinear vector fields.}
\label{tabla}
\end{table}

In Table~\ref{tabla}, $\wp$ stands for the Weierstrass  elliptic function with periods $\Lambda\subset\C$, and
 $\wp_\tau$ is the elliptic function defined in (\ref{wptau}). The parameter $\tau$ varies in $C\setminus\{\mr{o}\}$ up to the action of  $\mr{Aut}(C,\mr{o})\cong\big\langle e^{\frac{2i\pi}{m_C}}\big\rangle$, $m_C\in\{2,4,6\}$. Moreover, $E_\tau$ is the vector bundle $\mc O_C(\tau-\mr{o})\oplus\mc O_C$ and $A_0$ denotes the semistable undecomposable rank $2$ vector bundle.
 The parameters $c$ and $\gamma$ vary in $\C^*$ and $n\ge 0$ is an integer.

\begin{obs}\label{everywhere2}
The birational identification $\mb F_0\simeq\mb F_m$ induced by the choice of standard coordinates on the Hirzebruch surfaces (cf. Remark \ref{rad2})
enables us to consider the algebras $\mf a_{ij}^0$,
with $i, j \in \{0,1\}$, $\mf c_\gamma$ and $\mf d$ as subalgebras of $\mf{aut}(\mb F_0)$ as well as subalgebras of $\mf{aut}(\mb F_m)$. We also notice that $\langle \partial_y,y\partial_y\rangle$ is a well-defined affine Lie algebra of holomorphic vector fields on each ruled surface of type $C\times \p^1$, where $C$ is a curve of arbitrary genus.
\end{obs}
	
\begin{obs}\label{fib-2dim}
 Let $S$ be a projective surface of Kodaira dimension $-\infty$. Each 2-dimensional birationally integrable Lie subalgebra of $\mf X_{\mr{rat}}(S)$ preserves a rational fibration.
\end{obs}

\begin{proof}[Proof of Theorem~\ref{2dim}]
Let $\mf g$ be a $2$-dimensional Lie subalgebra of $\mf{aut}(S)$, where $S$ is a projective surface with $\mr{kod}(S)=-\infty$. 
As already pointed out, we can assume that $\mf g\subset\mf{aut}(S)$, where $S$ is $\p^2$ or a ruled surface (cf. Corollary~\ref{solv}). Furthermore, the case $S=\p^2$ can be reduced to $\mb F_1$, thanks to Lemma~\ref{Lie} and Remark~\ref{Nalg}.
Therefore, we suppose that $S$ is a ruled surface $\p E$ over a curve $C$ and we denote by $\pi\colon \p E\to C$ the natural projection. We analyze separately the cases in which $\mf g$ is, or is not, contained in the algebra $\mf{aut}^v(\p E)$ of vertical vector fields on the ruled surface.

\medskip

{\bf 1.} We suppose first that $\mf g\not\subset\mf{aut}^v(\p E)$, and we distinguish the two possibilities, $g(C)=0$ and $g(C)=1$.
\item 
Assume that $g(C)=0$, that is, $S$ is a Hirzebruch surface $\mb F_n$, with  $n\ge 0$.
Using Lemma~\ref{solv_S},  
we can suppose that $\mf g\subset\mf B_n$, just replacing $\mf g$ by $\phi^*\mf g$, where $\phi$ is an appropriate element of $\mr{Aut}(\mb F_n)$. By the hypothesis made, the projection $\pi_\ast\mf g$ is different from zero. According to the dimension of $\pi_\ast\mf g$, which can be either $1$ or $2$, the algebra $\mf g$ is of one of two types: $\mf g_1$ or $\mf g_2$. Using Lemma~\ref{fnQ}, we can chose $\phi$ such that 
\begin{align*}
\mf g_1&=\big\langle x^\mu\partial_x+(cy+\varepsilon x^m)\partial_y\,,\,(\bar c y+\bar p(x))\partial_y\big\rangle,\\[-5mm]
\intertext{and} \\[-10mm]
\mf g_2&=\big\langle \partial_x+(cy+p(x))\partial_y\,,\,x\partial_x+(\bar cy+\bar\varepsilon x^{\bar m})\partial_y\big\rangle,
\end{align*} 
where $\mu,\varepsilon,\bar\varepsilon\in\{0,1\}$, $m,\bar m\in\{0,\ldots,n\}$, $c,\bar c\in\C$  and $p,\bar p\in\C_n[x]$ with $\bar\varepsilon(\bar c-\bar m)=0$. Here, $(x,y)$ are standard coordinates of $\mb F_n$ introduced in Section~\ref{Shol}.

Imposing that $\mf g_1$ is a Lie algebra, we obtain the condition
\begin{equation}\label{edo_0}
\big[x^\mu\partial_x+(cy+\varepsilon x^m)\partial_y,(\bar c y+ \bar p(x))\partial_y\big]
=(x^\mu \bar p'(x)-c \bar p(x)+\bar c\varepsilon x^m)\partial_y=\lambda(\bar cy+ \bar p(x))\partial_y,
\end{equation}
for some $\lambda\in\C$ such that $\lambda \bar c=0$. We look at~\eqref{edo_0} as a differential equation fulfilled by the polynomial function $\bar p$ and the constant $\bar c$.
Using Lemma~\ref{fnQ} again in the choice of $\phi$, we see that there are four possibilities for the algebra $\mf g_1$, depending on the values that the parameters $\mu,c,\varepsilon,m$ can take. Solving the equation~\eqref{edo_0} separately in each of these four cases, and imposing that $\bar p$ must  belong to $\C_n[x]$, we obtain:
\begin{enumerate}[(i)]
\item If $\mu=0$, $c=0$ and $m=n$ then $\mf g_1=\langle \partial_x+\varepsilon x^n\partial_y,\partial_y\rangle$. In that case the birational map $f\colon \mb F_0\dashrightarrow \mb F_n$ given by $f(x,y) = \big(x,y+\frac{\varepsilon}{n+1}x^{n+1}\big)$ fulfills $f^\ast\mf g_1 = \langle\partial_x,\partial_y\rangle =\mf a^0_{00}$.

\item If $\mu=0$, $c=1$ and $\varepsilon=0$ then
 $\mf g_1=\langle \partial_x+y\partial_y,\partial_y\rangle=\mf d$ if $\bar c=0$, or $\mf g_1=\langle\partial_x+y\partial_y,y\partial_y\rangle=\langle\partial_x,y\partial_y\rangle=\mf a_{01}^0$ if $\lambda=0$.

\item If $\mu=1$ and $\varepsilon=0$ then $\mf g_1=\big\langle x\partial_x+cy\partial_y,x^{\lambda+c}\partial_y\big\rangle$ with $\lambda+c\in\{0,\ldots,n\}$ if $\bar c=0$, or $\mf g_1=\langle x\partial_x+cy\partial_y,(\bar c y+kx^c)\partial_y\rangle $ if $\lambda=0$.
In the first case the birational map $f\colon \mb F_0\dashrightarrow \mb F_n$ given by $f(x,y) = (x,yx^{\lambda+c})$ fulfills $f^\ast\mf g_1 = \langle x\partial_x-\lambda y\partial_y, \partial_y\rangle=\mf c_{-\lambda}$. In the second one  $f(x,y) = \big(x,yx^{c}\big)$ fulfills $f^\ast\mf g_1 = \langle x\partial_x, (\bar cy+k)\partial_y\rangle$, which is birationally conjugated to $\mf a_{11}^0$ when $\bar c\neq 0$ and to $\mf a_{01}^0$ when $\bar c=0$.

\item If $\mu=1$, $c=m$ and $\varepsilon=1$ then necessarily $\bar c=0$, and $\mf g_1=\big\langle x\partial_x+(my+x^m)\partial_y,x^{\lambda+m}\partial_y\big\rangle$
with $\lambda+m\in\{0,\ldots,n\}$. In the case $\mp\lambda\in\{1,\ldots,n\}$,
the birational map $f\colon \mb F_m\dashrightarrow \mb F_n$ given by $f(x,y) =\big(x^{\pm 1},\pm yx^{\pm(m+\lambda)}\big)$ fulfills $f^\ast\mf g_1 =\big\langle x\partial_x+(\mp\lambda y+x^{\mp\lambda})\partial_y,\partial_y\big\rangle=\mf f_{\mp\lambda}$. And in the case $\lambda=0$, the holomorphic map $f(x,y) = (x,yx^m)$ fulfills $f^\ast \mf g_1=\langle x\partial_x,\partial_y\rangle$, which is birationally conjugated to $\mf a_{01}^0$.
\end{enumerate}

Imposing that $\mf g_2$ is a Lie algebra we obtain the condition
\begin{align*}
\big[\partial_x+(cy+p(x))\partial_y,x\partial_x+(\bar cy+\bar\varepsilon x^m)\partial_y\big]&=\partial_x+(\bar\varepsilon \bar m x^{\bar m-1}-xp'(x)-c\bar\varepsilon x^{\bar m} +\bar c p(x))\partial_y\\
&=\partial_x+(cy+p(x))\partial_y.
\end{align*}
Clearly $c=0$ and the equation becomes $\bar\varepsilon \bar m x^{\bar m-1}-xp'(x)=(1-\bar c)p(x)$. We distinguish two possibilities:
\begin{enumerate}[(i)]
\item If $\bar\varepsilon = 0$ the solution is $p(x)=k x^{\bar c -1}$. Since $p(x)$ must be a polynomial function of degree $\leq n$, we obtain
\[
\mf g_2=\langle \partial_x+k x^{\bar c-1}\partial_y,x\partial_x+\bar c y\partial_y\rangle\quad\text{with $k=0$ if $\bar c-1\notin\{0,\ldots,n\}$,}
\]
If $k=0$ then either $\mf g_2=\langle\partial_x,x\partial_x\rangle\simeq \langle\partial_y,y\partial_y\rangle$ if $\bar c=0$  (this is the affine collinear case with base curve $C=\p^1$), or $\mf g_2=\langle \partial_x,x\partial_x+\bar c y\partial_y\rangle\simeq\mf c_{1/\bar c}$ if $\bar c\neq 0$. 
If $k\neq 0$ then $\bar c\neq 0$ and
$f^*\mf g _2=\langle\partial_x,x\partial_x+\bar c y\partial_y\rangle\simeq \mf c_{1/\bar c}$, where $f\colon \mb F_0 \dashrightarrow \mb F_m$ is the birational map given by $f(x,y) = \left(x,y+\frac{k}{\bar c}x^{\bar c}\right)$.
\item If $\bar\varepsilon \neq 0$ we see, taking into account the condition $\bar\varepsilon(\bar c-\bar m)=0$, that the general solution is $p(x)=x^{\bar m-1}(\bar\varepsilon\bar m\ln x+k)$. This implies that $\bar\varepsilon\bar m=0$, because $p(x)$ must be a polynomial. We deduce that $\bar m=0$. Hence, the only possibility is that $k=0$, that is $p(x)=0$.
Since $\bar c=\bar m=0$, we obtain
\[
\mf g_2=\langle \partial_x, x\partial_x+\bar\varepsilon\partial_y\rangle.
\]
Therefore, in that case $\mf g_2\simeq\mf d$ because $\bar \varepsilon\neq0$.
\end{enumerate}

We consider now the second possibility, namely $g(C)=1$. According to Theorem~\ref{Xhol-ruled_1} there are five types of minimal ruled surfaces $\p E$ over an elliptic curve $C$. The cases (1) $E=A_1$ and (4) $E=\mc O_C(np)\oplus O_C$ can be excluded in this discussion because, for the first one, the Lie algebra $\mf{aut}(\p E)$ is one-dimensional and, for the second one, $\mf{aut}(\p E)=\mf{aut}^v(\p E)$. 
The cases (2) and (3) provide the Lie algebras $\mf b_{\mr{o}}$ and $\mf b_\tau$ with $\tau\neq\mr{o}$, respectively.
It only remains to treat case (5) $\p E=C\times\p^1$. 

The Lie algebra $\mf{aut}(C\times\p^1)=\mf{aut}(C)\oplus\mf{aut}(\p^1)$ can be projected onto $\mf{aut}(C)$ and onto $\mf{aut}(\p^1)$. Therefore, $\mf g$ projects onto the one-dimensional Lie algebra $\mf{aut}(C)=\C\partial_x$ and onto a Lie subalgebra $\mf h\subset\mf{aut}(\p^1)=\C_2[y]\partial_y$ of dimension one or two. After a coordinate change of $\p^1$ we can assume that $\mf h\subset\C_1[y]\partial_y$.
From that, we can see that there exists $\phi\in\operatorname{Aut}(C\times\p^1)$, of the form $\phi(x,y)=(x,\phi_0(y))$, fulfilling $\mf g':=\phi^*\mf g=\langle \partial_x+(p_0+p_1y)\partial_y,y^\nu\partial_y\rangle$ with $\nu\in\{0,1\}$ and $p_i\in\C$. Imposing that $\mf g'$ is a Lie algebra gives the condition
\[
\big[\partial_x+(p_0+p_1y)\partial_y,y^\nu\partial_y]=(\nu p_0 y^{-1}+(\nu-1)p_1)y^\nu\partial_y
= \lambda y^\nu \partial_y,
\]
and we deduce that $\nu p_0=0$. If $\nu=0$
then $\mf g'=\langle \partial_x+p_1y\partial_y,\partial_y\rangle\subset\mf{aut}(C\times\p^1)$ is either 
 the affine Lie algebra $\mf e_{p_1}$
when $p_1\neq 0$ or the Abelian subalgebra $\mf a^1_{00}=\langle \partial_x,\partial_y\rangle$ of $\mf{aut}(C\times\p^1)$ when $p_1=0$. 
If $\nu=1$ then $p_0=0$ and
$\mf g'=\langle\partial_x,y\partial_y\rangle=\mf a^1_{01}\subset\mf{aut}(C\times\p^1)$ is Abelian.

\medskip 
{\bf 2.} We deal now with case $\mf g\subset\mf{aut}^v(\p E)$, where $\p E\to C$ is a ruled surface over a curve $C$ of arbitrary genus.
We choose a basis $X,Y$ of $\mf g$ such that $[X,Y]=aY$ with $a\in\C$.
According to Theorem~\ref{fibration_2}, there is a birational map $f:C\times\p^1\dashrightarrow  \p E$ such that $f^*Y=y^\nu\partial_y$, where $\nu\in\{0,1\}$ (up to a constant multiple of $Y$). 
Since $X$ is collinear to $Y$ there exists $g\in\C( C\times\p^1)$ such that $f^*X=g(x,y)\partial_y\in\mf X_{\mr{bir}}(C\times\p^1)$. As
$[f^*X,f^*Y]=af^*Y$ we deduce that $g(x,y)\nu y^{-1}-\partial_yg(x,y)=a$.

If $\nu=1$ then $g(x,y)=y(-a\ln y+h(x))$ is rational if and only if $a=0$ and $h\in\C(C)$. In this case $f^*X=h(x)y\partial_y$, which is not 
birationally integrable by Theorem~\ref{fibration_2}, because $h$ can not be constant as $X,Y$ are linearly independent.

If $\nu=0$ then $g(x,y)=-ay-h(x)$ with $h\in\C(C)$.
If $a\neq 0$ then the pull-back of $f^*\mf g=\big\langle \partial_y,(ay+h(x))\partial_y\big\rangle$ by the birational map $f(x,y)= \big(x,y-\frac{h(x)}{a}\big)$ is $\langle \partial_y,y\partial_y\rangle\subset\mf{aut}^v(C\times\p^1)$, which is the affine collinear Lie algebra. If $a=0$ then we claim that the vertical abelian Lie algebra $f^*\mf g=\langle h(x)\partial_y,\partial_y\rangle$ is regularizable on the surface $S_D=\p(\mc O_C(D)\oplus\mc O_C)$, where $D=(h)_\infty$.
Let $\{U_i\}$ be a sufficiently fine open covering of $C$ and $f_i,g_i\in\mc O_C(U_i)$ coprime sections such that $h_{|U_i}=f_i/g_i$. 
Then $g_i/g_j\in\mc O_C^*(U_i\cap U_j)$ are the transition functions of the line bundle $L=\mc O_C(D)$ and 
\[\{f_i\}\in H^0(C,\mc O_C(D))=H^0(C,\det(E)^{-1}\otimes L^2), \]
where $E=L\oplus \mc O_C$. We notice that $L$ is the maximal line subbundle (that is the line subbundle of maximal degree) of $E$.
According to \cite[\S 1]{Maruyama} the cocycle $\left\{\left(\begin{array}{cc}1 & f_i\\ 0 & 1\end{array}\right)\right\}$ defines an automorphism of $E$ and its class in $\mr{Aut}(E)/\C^*$  an element $f$ of $\mr{Aut}(S_D)$, see \cite[Lemma 3]{Maruyama}. Taking $U_0=C\setminus D$, $g_0=1$, $f_0=h_{|U_0}$ and a local trivialization
$(x,y):U_0\times\p^1\stackrel{\cong}{\longrightarrow}\p E_{|U_0}$ we deduce that  $f(x,y)=(x,y+h(x))$. Since the birational flow $(t,x,y)\mapsto (x,y+th(x))$ of the rational vector field $h(x)\partial_y$ in $C\times\p^1$ is holomorphic on $\p E=S_D$, we deduce from Lemma~\ref{todo_hol} that $h(x)\partial_y\in\mf{aut}(S_D)$ as we wanted.
This concludes the proof of Theorem~\ref{2dim}.
\end{proof}

\begin{obs}\label{hol-uni}
One can see that two algebras $\mf c_\gamma$ and $\mf c_{\gamma'}$ are birationally equivalent if $\gamma' =\pm \gamma$ and that $\mf e_c$ and $\mf e_{c'}$ are birationally equivalent if  $c'= \nu c$, where $\nu\in\big\langle e^{\frac{2i\pi}{m_C}}\big\rangle$.
It can be verified, through a laborious case-by-case analysis, that the above ones are the only possible equivalences. That is, any other pair of the Lie algebras appearing in Table~1 are not birationally equivalent. A particular case used in 
Theorem~\ref{maximal} is proved by the Proposition~\ref{Lab} below.
\end{obs}

\begin{prop}\label{Lab}
Assume that $C$ is an elliptic curve, and let be  $\tau\in C$ and $\nu\in\{0,1\}$. Then 
there is no birational map $f:\p E_\tau\to C\times\p^1$ such that $f^*\mf a^1_{0\nu}=\mf b_\tau$, where $E_{\mr{o}}=A_0$.
\end{prop}
\begin{proof}
Any birational map $f:\p E_\tau\to C\times \p^1$ preserves the rational fibration.
Therefore, if $f^*\mf a_{00}^1=\mf b_\tau$ with $\tau\neq\mr{o}$ (resp. $f^*\mf a^1_{01}=\mf b_{\mr{o}})$ then
$f^*\partial_y\in\langle y\partial_y\rangle$ (resp. $f^*(y\partial_y)\in\langle\partial_y\rangle$ and consequently $(f^{-1})^*\partial_y\in\langle y\partial_y\rangle$).
According to Proposition~\ref{Birpi} we can write $f^{\pm 1}(x,y)=(f_1(x),f_2(x,y))$. Then $(f^{\pm 1})^*\partial_y=\frac{1}{\partial_y f_2}\partial_y=\lambda y\partial_y$. This would imply that $f_2(x,y)=\frac{\log y}{\lambda}+\beta(x)$, which is not rational.
	
If there is a birational map $f:\p E_{\mr{o}}\to C\times\p^1$ such that $f^*\mf a^1_{00}=\mf b_\mr{o}$ then $f^*\langle\partial_y\rangle=\langle\partial_y\rangle$ and $f(x,y)=(f_1(x),f_2(x,y))$, where $f_1$ is an automorphism of $C$, that is, of the form $f_1(x)=ux+v$ with $u\in\langle e^{\frac{2i\pi}{m_C}}\rangle\cong \mr{Aut}(C,\mr{o})$. Since $f^*\partial_y=\frac{1}{\partial_y f_2}\partial_y=\lambda\partial_y$	we deduce that $f_2(x,y)=\lambda^{-1}y+\beta(x)$ with $\lambda\in\C^*$ and $\beta\in \C(C)$. Hence
\[
f^*\partial_x=\frac{1}{u}\left[\partial_x-\lambda\beta'(x)\partial_y\right]\in\langle\partial_x-\wp(x)\partial_y,\partial_y\rangle
\] 
so that $\beta'(x)=r-s\wp(x)$ for some $r,s\in\C$ with $r\neq 0$ (otherwise $\beta$ would not be periodic). 
As $\beta(x)=rx+s\zeta(x)\in\C(C)$, it satisfies $\beta(x+m_1\omega_1+m_2\omega_2)=\beta(x)$ for all $m_1,m_2\in\Z$. Imposing that condition to $(m_1,m_2)=(1,0)$ and 
$(m_1,m_2)=(0,1)$ we deduce that
\[
\left(\begin{array}{ll} \eta_1 & \omega_1\\ \eta_2 & \omega_2\end{array}\right)\left(\begin{array}{l}s\\ r\end{array}\right)=\left(\begin{array}{l}0\\ 0\end{array}\right).
\]
Since the determinant of the above  $2\times 2$ matrix is nonzero, see (\ref{determ}),
we obtain that $r=s=0$, which is a contradiction.
	
If there is a birational map $f:\p E_\tau\to C\times\p^1$ such that $f^*\mf a^1_{01}=\mf b_\tau$ with $\tau\neq\mr{o}$ then $f^*\langle y\partial_y\rangle=\langle y\partial_y\rangle$ and consequently $f(x,y)=(ux+v,\alpha(x)y^{\pm 1})$ with $\alpha\in\C(C)$. We deduce that
\[
f^*\partial_x=\frac{1}{u}\left[\partial_x\mp\frac{\alpha'(x)}{\alpha(x)}y\partial_y\right]\in\langle\partial_x-\wp_\tau(x)y\partial_y,y\partial_y\rangle
\] 
so that 
$\alpha'/\alpha=\pm \wp_\tau+r$ for some $r\in\C$. By~(\ref{wptau}) and~(\ref{wp}), $\wp_\tau(x)+\zeta(\tilde\tau)=\zeta(x)-\zeta(x-\tilde\tau)$ for some $\tilde\tau\in\C$ projecting onto $\tau\in\C/\Lambda=C$. Recall that $\zeta$ can be written as the quotient
$\zeta=\sigma'/\sigma$,
where $\sigma$ is the Weierstrass sigma function that fulfills the formula (\ref{sigma_add}).
Taking $\hat r= r\mp\zeta(\tilde \tau)\in\C$, we have that
\begin{equation*}
\alpha(x)=\exp\int(\pm\wp_\tau(x)+ r)dx=\left(\frac{\sigma(x)}{\sigma(x-\tilde\tau)}\right)^{\pm 1}e^{\hat rx+s}.
\end{equation*}
As $\alpha(x)\in\C(C)$, it
satisfies $\alpha(x+m_1\omega_1+m_2\omega_2)=\alpha(x)$ for all $m_1,m_2\in\Z$. Then
$\pm(m_1\eta_1+m_2\eta_2)\tilde\tau+\hat r(m_1\omega_1+m_2\omega_2)\in 2i\pi\Z$.
This implies that 
\[
\left(\begin{array}{cc}\eta_1 & \omega_1\\ \eta_2 & \omega_2\end{array}\right)\left(\begin{array}{c} \pm\tilde\tau\\ \hat r\end{array}\right)=2\pi i\left(\begin{array}{l}k_1\\ k_2\end{array}\right)
\]
with $k_1,k_2\in\Z$, and consequently $\tilde\tau=\mp k_2\omega_1\pm k_1\omega_2\in\Lambda= \Z\omega_1+\Z\omega_2$, i.e. $\tau=\mr{o}$ in $\C/\Lambda=C$, contradiction.
\end{proof}

\section{Semisimple subalgebras of birationally integrable vector fields}

In this section we classify, up to birational equivalence, semisimple birationally integrable subalgebras of $\mf X_{\mr{rat}}(S)$, where $S$ is a given projective surface. We prove that they are regularizable and we provide holomorphic models of all of them. We treat first the case of algebras of rank one, i.e., isomorphic to $\mf{sl}_2(\C)$, using the classification of two-dimensional affine algebras given in Theorem~\ref{2dim}. As every semisimple Lie algebra contains copies of $\mf{sl}_2(\C)$, that first result allows us to deal with the general case, showing that the rank of the algebra must be at most two, and concluding that the only possible birationally integrable semisimple algebras are isomorphic to $\mf{sl}_2(\C)$, $\mf{sl}_2(\C)\times\mf{sl}_2(\C)$ or $\mf{sl}_3(\C)$.

\subsection{Algebras of rank one}

The following result provides holomorphic normal forms, with respect to birational equivalence, for semisimple Lie subalgebras $\mf g\subset\mf X_{\mr{rat}}(S)$ of rank one that are birationally integrable. Examples of Lie subalgebras of $\mf{aut}(S)$ isomorphic to $\mf{sl}_2(\C)$ are: the algebra of vertical vector fields of $C\times \p^1$, the simple subalgebra arising in the Levi decomposition of 
 $\mf{aut}(\mb F_n)$ or the subalgebra $\mf{so}_3(\C)\subset\mf{sl}_3(\C)\cong\mf{aut}(\p^2)$. Another example is the image of the diagonal morphism $\Delta=(\pi_1^*,\pi_2^*):\mf{aut}(\p^1)\to\mf{aut}(\mb F_0)=\mf{aut}(\mb \p^1)\oplus\mf{aut}(\mb \p^1)$, where $\pi_i:\mb F_0=\p^1\times\p^1\to\p^1$, $i=1,2$, are the two natural projections. We prove that, up to birational equivalence, the algebra $\mf g$ is one of the above examples.

\begin{teo}\label{sl2}
Let $S$ be a projective surface and let $\mf g$ be a Lie subalgebra of $\mf X_{\mr{rat}}(S)$ isomorphic to $\mf{sl}_2(\C)$ and generated by birationally integrable vector fields. Then 
$S$ has Kodaira dimension $-\infty$. Moreover, 
$\mf g$ is regularizable and birationally equivalent to one 
of the following holomorphic models:
\begin{enumerate}[ (1)]
\item $\mf g_0=
\big\langle \partial_y,y\partial_y,y^2\partial_y\big\rangle  =\mf{aut}^v(C\times\p^1) \subset\mf{aut}(C\times\p^1)$, where $C$ is a curve of arbitrary genus,
\item $\mf g_n=\big\langle\partial_x,x\partial_x+\frac{n}{2}y\partial_y,x^2\partial_x+nxy\partial_y\big\rangle\subset\mf{aut}(\mb F_n)$, $n\ge 1$, 
\item $\tilde{\mf g}_2=	
\big\langle \partial_x+\partial_y,x\partial_x+y\partial_y,x^2\partial_x+y^2\partial_y\big\rangle =\Delta(\mf{aut}(\p^1))
\subset\mf{aut}(\mb F_0)$,
\item $\tilde{\mf g}_4=
\big\langle 
(x^2+1)\partial_x+xy\partial_y,xy\partial_x+(y^2+1)\partial_y,y\partial_x-x\partial_y
\big\rangle=\mf{so}_3(\C)
\subset\mf{aut}(\p^2)$.
\end{enumerate}
\end{teo}

\begin{obs}
	The  birational maps $h_2(x,y)=\big(\frac{x+y}{2},\frac{x-y}{2}\big)$ 
	and $h_4(x,y)=\big(\frac{2x}{x^2-y-1},\frac{i(x^2-y+1)}{x^2-y-1}\big)$ 
	satisfy that
	\begin{align*}
		\mf g_2^1&:=h_2^*\tilde{\mf g}_2=\langle\partial_x,x\partial_x+y\partial_y,(x^2+y^2)\partial_x+2xy\partial_y\rangle,\\
		\mf g_4^1&:=h_4^*\tilde{\mf g}_4=\big\langle \partial_x,x\partial_x+2y\partial_y,(x^2+y)\partial_x+4xy\partial_y\big\rangle,
	\end{align*}
 and we notice that the Lie algebras $\mf g_2$ and $\mf g_4$  are degenerations of $\mf g_2^1$ and $\mf g_4^1$ respectively. More precisely, $\lim\limits_{\lambda\to 0}(x,\lambda y)^*\mf g_n^1=\mf g_n$ for $n=2,4$.
	
The Lie algebra $\mf g_n$ can be obtained from $\mf g_1$ by pull-back by the branched covering $\phi_n(x,y)=(x,y^n)$ and $\mf g_2^1=\phi_2^*\mf g_4^1$.	
Nevertheless one can prove that any pair of the Lie algebras $\mf g_0$, $\mf  g_n$, $\mf g_2^1$ and $\mf g_4^1$ are not birationally equivalent. 
\end{obs}
	
		
\begin{obs}
The Lie algebras $\mf g_0$, $\mf g_n$ and $\tilde{\mf g}_2$ preserve the vertical rational fibration. It can be proved that the Lie algebra $\tilde{\mf g}_4$ does not preserve any rational fibration. Instead, it preserves an irreducible $2$-web.

One can prove that every element $X$ of each one of the above Lie algebras is tangent to a rational fibration, although the fibration varies with the vector field $X$. This implies in particular that certain $2$-dimensional birationally integrable Lie algebras (like $\mf c_\gamma$ with $\gamma\notin\mb Q$) can not be included in a birationally integrable Lie algebra isomorphic to $\mf{sl}_2(\C)$.
\end{obs}

\begin{proof}
We fix a basis $X, Y, Z$ of the Lie algebra $\mf g  \cong\mf{sl}_2(\C)$ that satisfies the structure equations
\begin{equation}\label{eq_lie}
	[X, Y] = X, \quad [Z, Y] = - Z, \quad [X, Z] = 2 Y.
\end{equation}
Since $\mf g$ is birationally integrable the $2$-dimensional affine subalgebra $\mf h=\big\langle X, Y\big\rangle$, 
is regularizable.
Therefore, $\mr{kod}(S)=-\infty$  because there is $S'\simeq S$ such that $\mf{aut}(S')$ is not Abelian. Theorem~\ref{2dim} states that  $\mf h$ is birationally equivalent to the Lie algebra $\langle \partial_y,y\partial_y\rangle\subset\mf{aut}(C\times\p^1)$, if the elements of $\mf h$ are collinear, or to one of the algebras $\mf c_\gamma$, $\mf d$, $\mf e_c$, $\mf f_n$, otherwise. In each case, we look for a birationally integrable vector field $Z$ such that the relations~\eqref{eq_lie} hold. We recall that $\lambda, c\in\C^\ast$.  In the search for holomorphic normal forms of the Lie algebra $\mf g$, we can assume that $S$ is a minimal surface or $\mb F_1$.

Consider first the case of the Lie algebra $\mf c_{1/\lambda}=\langle X,Y_\lambda\rangle$, with $X=\partial_y$, $Y_\lambda=\frac{x}{\lambda}\partial_x+y\partial_y$ and $\lambda\in\C^*$.
We look at $\mf c_{1/\lambda}$ as a subalgebra of $\mb F_m$,  where $m\geq 0$
(cf. Remark~\ref{everywhere2}). A straightforward computation shows that, if there is a third vector field $Z_\lambda$ such that
$[X,Y_\lambda]=X$, $[Z_\lambda,Y_\lambda]=-Z_\lambda$ and $[X,Z_\lambda]=2Y_\lambda$,
then
\[
Z_\lambda=\Big(\frac{2xy}{\lambda}+c_1x^{\lambda+1}\Big)\partial_x+\big(y^2+c_2x^{2\lambda}\big)\partial_y
\]
for some $c_1,c_2\in\C$. Since the vector field $Z_\lambda$ must be rational, we can assume that $c_1=0$. Indeed, in case $c_1\neq 0$, then $\lambda\in\Z$ necessarily, and the birational transformation   $f(x,y)=(x,y-\lambda c_1/2 x^\lambda)$ satisfies that $f^*X=X$, $f^*Y_\lambda=Y_\lambda$ and  $f^*Z_\lambda=\frac{2xy}{\lambda}\partial_x+\big(y^2+c_2x^{2\lambda}\big)\partial_y$.

Suppose that $c_2=0$. It can be verified that the vector field $Z_\lambda$ is birationally integrable if and only if $2/\lambda\in\Z$. In that case we set $n=2/\lambda$ and we have
$f^\ast\mf g = \big\langle\partial_y,\frac{n}{2}x\partial_x +y\partial_y,nxy\partial_x + y^2\partial_y\big\rangle$. By composing $f$  with the birational transformation $(x,y)\mapsto (y,x)$ if $n>0$, or $(x,y)\mapsto (y,1/x)$ if $n<0$, we see that $\mf g$ is birationally equivalent to $\mf g_n\subset \mf{aut}(\mb F_n)$. This gives model (2) of the theorem.

Suppose that $c_2\neq 0$. Since $Z_\lambda$ must be rational, $2\lambda\in\Z$. Moreover, taking the pull-back of $\mf g$ by the birational map $(x,y)\mapsto(1/x,y)$ if it is necessary, we can assume that $2\lambda$ is a positive integer. 
We consider now the family of vector fields  $Z_\lambda^c=\frac{2xy}{\lambda}\partial_x+(y^2+cx^{2\lambda})\partial_y$, which depends on the parameter $c$. Their local flow $\varphi^c$ can be written as $\varphi^c(t,x,y)=(x\varphi^c_1(t,x,y),\varphi^c_2(t,x,y))$, 
 where  $\varphi_1^c(t,x,y)$ is holomorphic along  the line  $x=0$ because this line is invariant under the flow.

The rescaling map $g_\mu(x,y)=(\mu x,y)$ satisfies $g_\mu^*\left(\frac{2xy}{\lambda}\partial_x+(y^2+x^{2\lambda})\partial_y\right)=\frac{2xy}{\lambda}\partial_x+(y^2+\mu^{2\lambda}x^{2\lambda})\partial_y$. This implies that all the vector fields $Z_\lambda^c$ are birationally equivalent for $c\neq 0$.  Moreover, the flows $\varphi^1$ an $\varphi^{c^{2\lambda}}$ are conjugated by $g_c$ and we deduce
\[
\varphi^{c^{2\lambda}}(t,x,y)=(x\varphi_1^1(t,cx,y),\varphi^1_2(t,cx,y)).
\]
We also have that
\[
\left(x(1-ty)^{-\frac{2}{\lambda}},\frac{y}{1-ty}\right)=
\varphi^0(t,x,y)=\big(x\varphi_1^1(t,0,y),\varphi^1_2(t,0,y)\big),
\]
where the first equality is obtained by integration of $Z^c_\lambda$, and the second one follows from the smooth dependence of $\varphi^c$ on the parameter $c$.
If $\varphi^1(t,x,y)$ is rational in $x,y$ then so is 
$\varphi_1^1(t,x,y)$ which has no poles along $x=0$.
Hence $\varphi_1^1(t,0,y)=(1-ty)^{-\frac{2}{\lambda}}$ is rational in $y$ and consequently $2/\lambda\in\mb Z$.

Therefore, $\lambda=2/n=m/2$ for certain positive integers $n,m $ satisfying $m n=4$. Since $c_2\neq 0$, making the rescaling $(x,y)\mapsto(\mu x,y)$ we can assume that, in fact, $c_2=1$. Hence, there are three possible cases:
\begin{enumerate}[ (a)]
\item $(n,m)=(4,1)$, $\lambda=2$.  In this case the flow of the vector field $Z_2$,
\[
\varphi^{Z_2}_t (x,y) =
\Big(\frac{x}{\sqrt{1-2yt+(y^2-  x^4)t^2}},\frac{y+(x-y^2)t}{1-2yt+(y^2- x^4)t^2}\Big),
\]
is not birational.
\item  $(n,m)=(2,2)$, $\lambda=1$.  In this case 
$f^*\mf g= \langle\partial_y, x\partial_x+y\partial_y,2xy\partial_x+(y^2+x^2)\partial_y\rangle$ which is birationally equivalent to 
\[
\tilde{\mf g}_2=\langle \partial_x+\partial_y,x\partial_x+y\partial_y,x^2\partial_x+y^2\partial_y\rangle\subset\mf{aut}(\mb F_0)
\]
via the map  $(x,y)\mapsto\left(\frac{x-y}{2},\frac{x+y}{2}\right)$.
 This gives model (3) of the theorem.
\item $(n,m)=(1,4)$, $\lambda=1/2$. In this case 
$f^*\mf g=\langle\partial_y,2x\partial_x+y\partial_y,4xy\partial_x+(y^2+x)\partial_y\rangle$.
The birational map $g:\p^2\to\mb F_m$ given by $g(x,y)=(x^2-2y,x)$ satisfies that  
\[\mf h=g^*f^*\mf g=\langle \partial_x+x\partial_y,x\partial_x+2y\partial_y,(y-x^2)\partial_x-xy\partial_y\rangle
\subset\mf{aut}(\p^2).\] 
Notice that $\mf h$ consists in the holomorphic vector fields on $\p^2$ whose flow leaves invariant the nondegenerate conic $x^2-2y=0$. Since all nondegenerate conics are projectively equivalent to $x^2+y^2+z^2=0$ we deduce that $\mf h$ is holomorphically conjugated to $\mf{so}_3(\C)=\tilde{\mf g}_4$.  This gives model (4) of the theorem.
\end{enumerate}

\smallskip

Consider now the Lie algebras $\mf d \subset\mf{aut}(\mb F_0)$ and $\mf e_c\subset\mf{aut}(\C/\Lambda\times\p^1)$, generated by $X=\partial_y$ and $Y_c=\frac{1}{c}\partial_x+y\partial_y$ (with $c=1$ in the case of $\mf d$). A straightforward computation shows that, if a vector field $Z_c$ satisfies the structure equations $[X,Y_c]=X$, $[Z,Y_c]=-Z_c$ and $[X,Z_c]=2Y_c$, then 
\[
Z_c=\Big(\frac{2y}{c}+c_1e^{cx}\Big)\partial_x+(y^2+ c_2e^{2cx})\partial_y,
\]	which is rational on $C\times\p^1$ with $g(C)\in\{0,1\}$ only for $c_1=c_2=0$. But in this case, its flow 
\[
\varphi_t^{Z_c} =
\Big(x-\frac{2}{c}\log(1-ty),\frac{y}{1-ty}\Big)
\]
is not birational.

In the case of the Lie algebra $\mf f_n \subset\mf{aut}(\mb F_n)$ generated
by $X=\partial_y$ and $Y=\frac{x}{n}\partial_x+\left(y+\frac{x^n}{n}\right)\partial_y$, it can be verified that there is no rational vector field $Z$ that satisfies the structure equations~\eqref{eq_lie}. 

Finally, in the case of the Lie algebra $\langle X=\partial_y,Y=y\partial_y\rangle\subset\mf{aut}(C\times\p^1)$,  it can be verified that the only vector field satisfying the equations~\eqref{eq_lie} is $Z=y^2\partial_y$.
This gives model (1) of the theorem ending the proof.
\end{proof}

\subsection{Algebras of higher rank}

The above Theorem~\ref{sl2} allows us to deal with the case of semisimple algebras of rank $\geq 2$.
\begin{teo}\label{r>1}	
Let $S$ be a projective surface and let $\mf g$ be a birationally integrable semisimple Lie subalgebra of $\mf X_{\mr{rat}}(S)$ that has $\operatorname{rank}(\mf g)\geq 2$.
Then $\mf g$ is regularizable and birationally equivalent to $\mf{aut}(\p^1\times\p^1)$ or $\mf{aut}(\p^2)$.
\end{teo}

We first prove the following preliminary result.
\begin{lema}
Let $S$ be an algebraic surface such that $\mf X_{\mr{rat}}(S)$ contains a birationally integrable semisimple subalgebra $\mf g$ of rank $\geq 2$. Then $S$ is rational. 
\end{lema}

\begin{dem}
Since $\mf g$ contains a subalgebra isomorphic to $\mf{sl}_2(\C)$, Theorem~\ref{sl2} implies that $\operatorname{kod}(S) = -\infty$. Therefore, $S$ is birationally equivalent to a ruled surface $C\times\p^1$, where $C$ is a certain curve. Assume that $g(C)>0$. Then
the flow of every birationally integrable vector field on $C\times\p^1$ preserves the rational fibration $\pi:C\times\p^1\to C$ and projects through $\pi_*$ onto a holomorphic vector field on $C$. Since $\dim\mf{aut}(C)\le 1$ and $\mf g$ is semisimple 
$\pi_\ast\mf g$ is zero, i.e. $\mf g$ is tangent to the rational fibration.

Lemma~\ref{lema_2} implies that the polar locus $\mf g_{\infty}$ of $\mf g$ is a finite union of fibres of $\pi$. Hence, there is a Zariski open subset $U\subset C$ such that the restriction of the elements of $\mf g$ to $U\times\p^1$ are holomorphic. Given a point $x_0\in U$ the map $\xi_{x_0}\colon \mf g\to \mf{aut}(\pi^{-1}(x_0))\cong\mf{sl}_2(\C)$ that restricts the elements of $\mf g$ to the fibre $\pi^{-1}(x_0)$ is a morphism of Lie algebras. Since $\operatorname{rank}(\mf g)\geq 2$, there are two linearly independent elements $h_1, h_2\in\mf g$ which commute (for instance two elements of a Cartan subalgebra). For a generic point $x_0$ the map $\xi_{x_0}$ is injective and therefore 
$\xi_{x_0}(h_i)\in\mf{sl}_2(\C) $ are also linearly independent and commute as well. This gives a contradiction, showing that $g(C) = 0$ and therefore that $S$  must be rational.
\end{dem}

\begin{proof}[Proof of Theorem \ref{r>1}]
By the above lemma, we can assume that $\mf g$ is an algebra of vector fields on $\p^2$ or $\mb F_n$ with $n\geq 0$ (we include $\mb F_1$ in our discussion even though it is not minimal).
We fix a Borel subalgebra $\mf B(\mf g)$ of $\mf g$. On $\p^2$, we chose the Borel subalgebra $\mf B\subset\mf{aut}(\p^2)\cong\mf{sl}_3(\C)$ of those vector fields whose flows preserve the projective flag $[0:1:0]\in L_\infty=\{z=0\}$, that is
$\mf B =\big\langle\partial_x, x\partial_x, \partial_y, x\partial_y, y\partial_y \big\rangle$. And, on $\mb F_n$, we consider the Borel subalgebra $\mf B_n\subset\mf{aut}(\mb F_n)$ defined in~\eqref{bn}, that is $\mf B_n=\C_1[x]\partial_x\oplus\C_n[x]\partial_y\oplus\C y\partial_y$. Since $\mf B$ is solvable, it is regularizable (cf. Corollary~\ref{solv}) on a surface  that we can assume to be $\p^2$ or $\mb F_n$, and due to Lemma~\ref{solv_S} we can also assume that $\mf B(\mf g)$ is included in $\mf B$ or $\mf B_n$.
Let $\Psi\colon\mb F_n\dashrightarrow \p^2$ be the birational map induced by the choice of affine coordinates on $\p^2$ (cf. Remark~\ref{rad2}). Then $\Psi^\ast\mf B = \mf B_1$. Thus, in the following discusion, we will identify
$\mf B$ with $\mf B_1$.

We denote by $\beta:\mf{B(g)}\hookrightarrow \mf B_n$ the inclusion map and
we fix a Cartan subalgebra $\mf H$ of $\mf g$ contained in $\mf B(\mf g)$. The elements $h\in \mf H\setminus\{0\}$ are not $\mr{ad}$-nilpotent in $\mf B(\mf g)$, thus their images $\beta(h)$ are not $\mr{ad}$-nilpotent in $\mf B_n$ either. Since the elements of $\mf B_n^{(1)}$ are $\mr{ad}$-nilpotent,
we deduce that the composition map $\bar\beta:\mf H\stackrel{\beta}{\longrightarrow} \mf B_n\to\mf B_n/\mf B_n^{(1)}$ is injective and, consequently, that the rank, $r=\dim \mf H$, of $\mf g$ is $r\le\dim\big(\mf B_n/\mf B_n^{(1)}\big)=2$ (cf. Remark~\ref{rad}).
Therefore $r=2$ and $\mf g$ is one of the semisimple Lie algebras of rank $2$, i.e. $\mf g\in\{A_1\times A_1,A_2,B_2,G_2\}$. Here we are using the standard notation for simple complex Lie algebras, that is $A_k\cong \mf{sl}_{k+1}(\C)$, $B_2 \cong \mf{so}_5(\C)$ and $G_2$ is the exceptional Lie algebra of rank $2$.

We analyze the four possibilities for the algebra $\mf g$ separately, but introduce the following conventions that will be used in all cases. The map $\bar\beta:\mf H\to\mf B_n/\mf B_n^{(1)}$ is an isomorphism and the classes of $x\partial_x,y\partial_y$ span  $\mf B_n/\mf B_n^{(1)}$. Therefore
$\beta(\mf H)=\langle h_1,h_2\rangle$ where
\[
h_1= (x+\lambda)\partial_x+u(x)\partial_y, \quad h_2=\mu\partial_x+(y+v(x))\partial_y,
\] 
with
$\lambda,\mu\in\C$ and $u,v\in\C_n[x]$. Imposing that $[h_1,h_2]=0$ we deduce that $\mu=0$ and $u(x)=-(x+\lambda)v'(x)$. 
Then the automorphism $f(x,y)=(x-\lambda,y-v(x-\lambda))$ of  $\mb F_n$ satisfies
$f^*h_1=x\partial_x$ and $f^*h_2=y\partial_y$. 
Replacing $\mf g$ by $f^*\mf g$ we can assume that $\beta(\mf H)=\langle x\partial_x,y\partial_y\rangle$. The eigenspaces of $\beta(\mf H)$ are $L_{-1}=\langle\partial_x\rangle$ and $L_i=\langle x^i\partial_y\rangle$ with roots $\rho_{-1}(ax\partial_x+by\partial_y)=-a$, $\rho_i(ax\partial_x+by\partial_y)=ia-b$, $i=0,\ldots,n$, satisfying $\rho_i=\rho_n+(n-i)\rho_{-1}$. We now consider each of the four possibilities, seeing whether the map $\beta \colon\mf B(\mf g)\to \mf B_n$ can be extended to an inclusion of $\mf g$ into $\mf X_{\mr{bir}}(\p^2)\cong\mf X_{\mr{bir}}(\mb F_n)$.

\medskip
{\bf 1.} Case $\mf g=A_1\times A_1$.
We set $A_1\times A_1=\langle h_1,v_1,v_1^-\rangle\oplus \langle h_2,v_2,v_2^-\rangle$ with $[h_1,h_2]=0$, $[h_i,v_j]=\delta_{ij}v_j$,  $[v_1,v_2]=0$,  $[h_i,v_j^-]=-\delta_{ij}v_j^-$ and $[v_i,v_j^-]=2\delta_{ij}h_i$.
We fix the Borel subalgebra  $ \langle h_1,v_1\rangle\oplus\langle h_2,v_2\rangle$ of~$\mf g$.
Since $v_i$ is an eigenvector of $\mf H$ we deduce that $\beta(v_i)$ is an eigenvector of $\beta(\mf H)$. Hence
$\beta(v_i)\in L_{j_i}$ for some $j_i\in\{-1,\ldots,n\}$, moreover $[L_{j_1},L_{j_2}]=0$ because $v_1$ and $v_2$ commute.  If $\beta$ can be extended to a map $\beta:\mf g\hookrightarrow\mf X_{\mr{bir}}(\mb F_n)$, then we can write $\beta(v_i^-)=P_i(x,y)\partial_x+Q_i(x,y)\partial_y$ for some $P_i,Q_i\in\C(x,y)$. There are two possibilities:
\begin{enumerate}[ (i)]
\item $j_1=-1$ and $j_2=0$, in which case $\beta(v_1)=\partial_x$, $\beta(h_1)=-x\partial_x$, $\beta(v_2)=\partial_y$ and $\beta(h_2)=-y\partial_y$. 	
We have
\[
0=\big[\partial_y,P_1(x,y)\partial_x+Q_1(x,y)\partial_y\big]=\partial_yP_1\partial_x+\partial_yQ_1\partial_y\Rightarrow\left\{\begin{array}{l}
		P_1(x,y)=p_1(x)\\
		Q_1(x,y)=q_1(x)
	\end{array}\right.\]
	\[-2x\partial_x=\big[\partial_x,p_1(x)\partial_x+q_1(x)\partial_y\big]=p_1'(x)\partial_x+q_1'(x)\partial_y\Rightarrow\left\{\begin{array}{l}
		p_1(x)=-x^2+c_1\\
		q_1(x)=c_2
	\end{array}\right.\]
	\[-\big((-x^2+c_1)\partial_x+c_2\partial_y\big)=\big[-x\partial_x,(-x^2+c_1)\partial_x+c_2\partial_y\big]=(-x^2+c_1+2x^2)\partial_x.
	\]
Therefore $c_1=c_2=0$ and $\beta(v_1^-)=-x^2\partial_x$. Switching the variables $x,y$ we obtain that $\beta(v_2^-)=-y^2\partial_y$. Thus, in this case $\beta(A_1\times A_1)=\mf{aut}(\mb F_0)$.
\item $j_1,j_2\ge 0$, in which case $\beta(v_1)=x^{j_1}\partial_y$, $\beta(h_1)=\frac{1}{j_1-j_2}(x\partial_x+j_2y\partial_y)$, $\beta(v_2)=x^{j_2}\partial_y$ and $\beta(h_2)=\frac{-1}{j_1-j_2}(x\partial_x+j_1y\partial_y)$.	Since $0=[x^{j_2}\partial_y,P_1\partial_x+Q_1\partial_y]$ we deduce that $P_1=P_1(x)$ and consequently 
\[
[\beta(v_1),\beta(v_1^-)]=[x^{j_1}\partial_y,P_1(x)\partial_x+Q_1(x,y)\partial_y]=x^{j_1-1}\big(x\partial_yQ_1(x,y)-j_1P_1(x)\big)\partial_y\neq 2\beta(h_1).
\]  
So this situation does not occur. 
\end{enumerate}

\medskip
{\bf 2.} $\mf g=A_2$. We set $A_2=\langle h_1,h_2,v_1^\pm,v_2^\pm,v_3^\pm\rangle$ with
$[h_j,v_i^\pm]=\pm\alpha_i(h_j)v_i^\pm$, $[v_1^\pm,v_2^\pm]=v_3^\pm$, $[v_i^+,v_i^-]\in\langle h_1,h_2\rangle$,
$[v_1^\pm,v_3^\mp]\in\langle v_2^\mp\rangle$, $[v_2^\pm,v_3^\mp]\in\langle v_1^\mp\rangle$,
and the remaining products being zero. The roots $\pm\alpha_1$, $\pm\alpha_2$, $\pm \alpha_3$ with $\alpha_3=\alpha_1+\alpha_2$ are depicted in Figure~\ref{A2}.
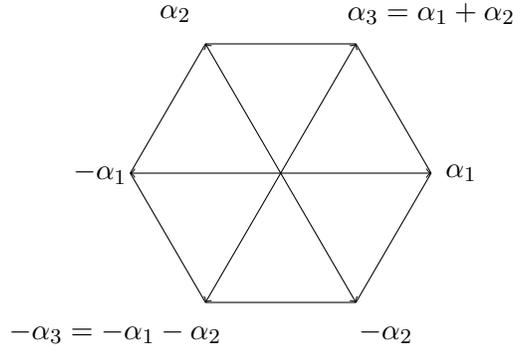
\begin{figure}[h]
	\begin{center}
		\begin{tikzpicture}[scale=2]
			\draw (1,0) to (.5,.86) to (-.5,.86) to (-1,0) to (-.5,-0.86) to (.5, -0.86) to (1,0);
			\node at (1.2,0) {$\alpha_1$};
			\node at (-1.2,0) {$-\alpha_1$};
			\node at (1,1.06) {$\alpha_3=\alpha_1+\alpha_2$};
			\node at (-.7,1.06) {$\alpha_2$};
			\node at (.7,-1.06) {$-\alpha_2$};
			\node at (-1.1,-1.06) {$-\alpha_3=-\alpha_1-\alpha_2$};
			\draw[->] (0,0) to (1,0);
			\draw[->] (0,0) to (.5,0.86);
			\draw[->] (0,0) to (-0.5,0.86);
			\draw[->] (0,0) to (-.5,-.86);
			\draw[->] (0,0) to (.5,-0.86);
			\draw[->] (0,0) to (-1,0);
		\end{tikzpicture}
	\end{center}
	\caption{Root system of the simple Lie algebra $A_2$.}\label{A2}
\end{figure}

If $\beta:A_2\hookrightarrow \mf X_{\mr{bir}}(\mb F_n)$ with $\beta(\mf B(A_2))\subset \mf B_n$ then $\beta(v_i^+)\in L_{j_i}$
for some $j_i\in\{-1,\ldots,n\}$. In fact, there is $i$ such that $j_i=-1$ because otherwise $\beta(\mf B(A_2))$ would be Abelian. Since $[v_1^+,v_3^+]=[v_2^+,v_3^+]=0$ we deduce that
$\beta(v_3^+)\in\langle\partial_y\rangle$, $\beta(v_i^+)\in\langle \partial_x\rangle$ and $\beta(v_j^+)\in\langle x\partial_y\rangle$ with $(i,j)\in\{(1,2),(2,1)\}$. Hence $\beta(\mf B(A_2))\subset\mf B_1$. Up to switching the generators $v_1^\pm$ and $v_2^\pm$ we can assume that
\[
\beta(v_1^+)=\partial_x,\ \beta(v_2^+)=x\partial_y,\ \beta(v_3^+)=\partial_y.\] 
If $\beta(v_1^-)=P_1\partial_x+Q_1\partial_y$, with $P_1,Q_1\in\C(x,y)$, then
\[\langle x\partial_x,y\partial_y\rangle\ni[\partial_x,P_1\partial_x+Q_1\partial_y]=\partial_xP_1\partial_x+\partial_xQ_1\partial_y,\]
\[\langle x\partial_y\rangle\ni[\partial_y,P_1\partial_x+Q_1\partial_y]=\partial_yP_1\partial_x+\partial_yQ_1\partial_y.\]
Then $P_1(x,y)=a_1x^2+a_0$ and $Q_1(x,y)=b_1xy+b_0$.
Since $[x\partial_x,\beta(v_1^+)]=[x\partial_x,\partial_x]=-\partial_x=-\beta(v_1^+)$ we deduce $[x\partial_x,\beta(v_1^-)]=\beta(v_1^-)$, that is 
\[
(a_1x^2+a_0)\partial_x+(b_1xy+b_0)\partial_y=\beta(v_1^-)=\big[x\partial_x,\beta(v_1^-)\big]=\big[x\partial_x,(a_1x^2+a_0)\partial_x+(b_1xy+b_0)\partial_y\big].
\]
Consequently $a_0=b_0=0$. Finally, since $0=[\beta(v_2^+),\beta(v_1^-)]=[x\partial_y,a_1x^2\partial_x+b_1xy\partial_y]$ we deduce that $a_1=b_1$, i.e. $\beta(v_1^-)\in\langle xR\rangle$,
where $R=x\partial_x+y\partial_y$. 

Next, if $\beta(v_2^-)=P_2\partial_x+Q_2\partial_y$ then 
$[\partial_x,P_2\partial_x+Q_2\partial_y]=0$ and consequently $P_2=P_2(y)$ and $Q_2=Q_2(y)$. Moreover,
\[
\langle x\partial_x,y\partial_y\rangle\ni\big[x\partial_y,P_2(y)\partial_x+Q_2(y)\partial_y\big]=xP_2'(y)\partial_x+\big(xQ_2'(y)-P_2(y)\big)\partial_y.
\]
Therefore $P_2(y)=a_2y$ and $Q_2(y)=b_2$.
Since $[x\partial_x,\beta(v_2^+)]=[x\partial_x,x\partial_y]=x\partial_y$ we deduce that $-\beta(v_2^-)=[x\partial_x,\beta(v_2^-)]=[x\partial_x,a_2y\partial_x+b_2\partial_y]=-a_2y\partial_x$ and consequently $\beta(v_2^-)\in\langle y\partial_x\rangle$. Finally, $\beta(v_3^-)=[\beta(v_1^-),\beta(v_2^-)]\in\langle yR\rangle$.
Hence $\beta(A_2)=\mf{aut}(\p^2)$ (cf. the description of $\mf{aut}(\p^2)$ given in \eqref{XP2}).

\medskip
{\bf 3.} Case $\mf g=B_2$. We set $B_2=\langle h_1,h_2, v_{10}^\pm,v_{01}^\pm,v_{11}^\pm, v_{21}^\pm\rangle$,
with
\[[h_k,v_{ij}^\pm]=\pm\alpha_{ij}(h_k)v_{ij}^\pm,\quad  [v_{10}^\pm,v_{01}^\pm]=v_{11}^\pm,\quad [v_{10}^\pm,v_{11}^\pm]=v_{21}^\pm,\quad [v_{ij}^+,v_{ij}^-]\in \langle h_1,h_2\rangle\]
and the remaining products being zero. The roots $\pm\alpha_{ij}$ of $B_2$, fulfilling
$\alpha_{ij}=i\alpha_{10}+j\alpha_{01}$ for $i=0,1,2$ and $j=0,1$, are depicted in Figure~\ref{B2}.
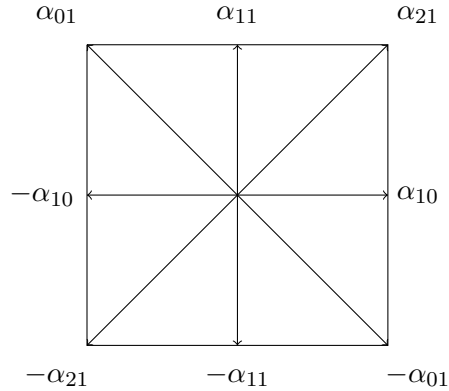
\begin{figure}[h]
	\begin{center}
		\begin{tikzpicture}[scale=2]
			\draw (-1,1) to (1,1) to (1,-1) to (-1,-1) to (-1,1);
			\draw[->] (0,0) to (-1,1);
			\draw[->] (0,0) to (-1,-1);
			\draw[->] (0,0) to (1,-1);
			\draw[->] (0,0) to (1,1);
			\draw[->] (0,0) to (1,0);
			\draw[->] (0,0) to (-1,0);
			\draw[->] (0,0) to (0,1);
			\draw[->] (0,0) to (0,-1);
			\node at (1.2,0) {$\alpha_{10}$};
			\node at (-1.3,0) {$-\alpha_{10}$};
			\node at (-1.2,1.2) {$\alpha_{01}$};
			\node at (0,1.2) {$\alpha_{11}$};
			\node at (1.2,1.2) {$\alpha_{21}$};
			\node at (-1.2,-1.2) {$-\alpha_{21}$};
			\node at (0,-1.2) {$-\alpha_{11}$};
			\node at (1.2,-1.2) {$-\alpha_{01}$};
		\end{tikzpicture}
	\end{center}
	\caption{Root system of the simple Lie algebra $B_2$.}\label{B2}
\end{figure}

We fix the Borel subalgebra $\mf B(B_2)=\langle h_1,h_2,v_{10}^+,v_{01}^+,v_{11}^+,v_{21}^+\rangle$.
The Lie algebra morphism $\beta:\mf B(B_2)\hookrightarrow\mf B_n$  sends the radical $\mf r$ of $\mf B(B_2)$ into the radical $L_0\oplus(L_{-1}\loplus\bigoplus_{k=1}^nL_k)$ of $\mf B_n$. Moreover, each eigenvector $v_{ij}^+$ is sent  into an eigenspace $L_{k}$ for some $k\in\{-1,\ldots,n\}$. 
As $\beta(\mf r)$ is not Abelian it must contain $L_{-1}$.
Since $v_{21}^+$ belongs to the center of $\mf r$,  $\beta(v_{21}^+)$ commutes with $\beta(\mf r)\supset L_{-1}$.
As $L_0$ is the only eigenspace in $\bigoplus_{k=1}^nL_k$ commuting with $L_{-1}$,
we deduce that $\beta(v_{21}^+)\in L_{0}$.
Hence
\[\beta(\mf r)=L_0\oplus (L_{-1}\loplus( L_{k_1}\oplus L_{k_2})), \quad 0<k_1<k_2.\]
If $[L_{-1},L_k]=0$ then $k\in\{-1,0\}$. Since the centralizer of $v_{10}^+$ in $\mf r$ is $\langle v_{10}^+,v_{21}^+\rangle$ we deduce  that
$\beta(v_{10}^+){\in}L_{-1}$.
Finally, $\mr{ad}_{v_{10}^+}^3=0$ implies that $\beta(\mf r)\subset\ker\mr{ad}_{\beta(v_{10}^+)}^3=L_0\oplus(L_{-1}\loplus (L_1\oplus L_2))$. 
Hence $\beta(\mf B(B_2))=\mf B_2$.
Up to rescaling we can assume that
\[\beta(v_{10}^+)=\partial_x,\quad
\beta(v_{21}^+)=\partial_y,\quad\beta(v_{11}^+)=x\partial_y,\quad\beta(v_{01}^+)=x^2\partial_y.\]
Write $\beta(v_{ij}^-)=P_{ij}(x,y)\partial_x+Q_{ij}(x,y)\partial_y$ with $P_{ij},Q_{ij}\in\C(x,y)$.
Since $[x\partial_x,\beta(v_{10}^+)]=-\beta(v_{10}^+)$ and $[y\partial_y,\beta(v_{10}^+)]=0$ we deduce that
\[
\big[x\partial_x,P_{10}\partial_x+Q_{10}\partial_y\big]=P_{10}\partial_x+Q_{10}\partial_y,\quad\big[y\partial_y,P_{10}\partial_x+Q_{10}\partial_y\big]=0.
\]
Therefore $\beta(v_{10}^-)=a_{10}x^2\partial_x+b_{10}xy\partial_y$. Since $[\beta(v_{01}^+),\beta(v_{10}^-)]=0$ we deduce that $b_{10}=2a_{10}$, i.e. $\beta(v_{10}^-)\in\langle x^2\partial_x+2xy\partial_y\rangle\in\mf{aut}(\mb F_2)\subset\beta(B_2)$.
Since $[\beta(v_{10}^+),\beta(v_{01}^-)]=0$ we deduce that $\beta(v_{01}^-)=P_{01}(y)\partial_x+Q_{01}(y)\partial_y$.
Imposing that $[\beta(v_{01}^+),\beta(v_{01}^-)]\in\langle x\partial_x,y\partial_y\rangle$ we deduce that $P_{01}'(y)=0$  and $xQ_{01}'(y)-2P_{01}(y)=0$, i.e.
$\beta(v_{01}^-)\in\langle\partial_y\rangle$, which is impossible.
Hence $\mf X_{\mr{bir}}(\p^2)\cong\mf X_{\mr{bir}}(\mb F_n)$ can not contain $B_2$.

\medskip
{\bf 4.} Case $\mf g=G_2$. We set $G_2=\langle h_1,h_2,v_{10}^\pm,v_{01}^\pm,v_{11}^\pm,v_{21}^\pm,v_{31}^\pm,v_{32}^\pm\rangle$ with 
$[h_k,v_{ij}^\pm]=\pm\alpha_{ij}(h_k)v_{ij}^\pm$,
$\langle [v_{ij}^\pm,v_{k\ell}^\pm]\rangle=\langle v_{i+k,j+\ell}^\pm\rangle$, where by convention $v_{ij}=0$ if $(i,j)\notin\{(1,0),(0,1),(1,1),(2,1),(3,1),(3,2)\}$, and $[v_{ij}^+,v_{ij}^-]\in\langle h_1,h_2\rangle$, and the remaining products being zero. The roots $\{\pm\alpha_{ij}\}$ of $G_2$ are depicted in Figure~\ref{G2}. They fulfill the relations
\[\alpha_{11}=\alpha_{10}+\alpha_{01},\quad \alpha_{21}=2\alpha_{10}+\alpha_{01},\quad\alpha_{31}=3\alpha_{10}+\alpha_{01},\quad \alpha_{32}=3\alpha_{10}+2\alpha_{01}.\]

\begin{figure}[h]
	\begin{center}
		\begin{tikzpicture}
			\draw (1,0) to (.5,.86) to (-.5,.86) to (-1,0) to (-.5,-0.86) to (.5, -0.86) to (1,0);
			\draw[->] (0,0) to (1,0);
			\node at (1.35,0) {$ \alpha_{10}$};
			\draw[->] (0,0) to (.5,0.86);
			\node at (.7,1.1) {$\alpha_{21}$};
			\draw[->] (0,0) to (-0.5,0.86);
			\node at (-.5,1.1) {$ \alpha_{11}$};
			\draw[->] (0,0) to (-.5,-.86);
			\draw[->] (0,0) to (.5,-0.86);
			\draw[->] (0,0) to (-1,0);
			\begin{scope}[rotate=30,scale=2]
				\draw (1,0) to (.5,.86) to (-.5,.86) to (-1,0) to (-.5,-0.86) to (.5, -0.86) to (1,0);
				\draw[->] (0,0) to (1,0);
				\draw[->] (0,0) to (.5,0.86);
				\node at (0.5,1) {$ \alpha_{32}$};
				\draw[->] (0,0) to (-0.5,0.86);
				\node at (-.6,1) {$\alpha_{01}$};
				\draw[->] (0,0) to (-.5,-.86);
				\draw[->] (0,0) to (.5,-0.86);
				\draw[->] (0,0) to (-1,0);
				\node at (1.2,-.05) {$\alpha_{31}$};
			\end{scope}
		\end{tikzpicture}
	\end{center}
	\caption{Root system of the simple Lie algebra $G_2$.}\label{G2}
\end{figure}
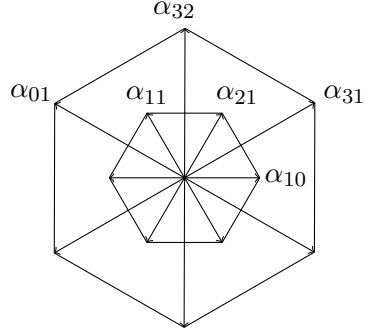
We observe that
\[	
\mf B^{(1)}(G_2)= \langle v_{10},v_{01},v_{11},v_{21},v_{31},v_{32}\rangle,\quad
	\mf B^{(2)}(G_2)= \langle v_{11},v_{21},v_{31},v_{32}\rangle, \quad
	\mf B^{(3)}(G_2)= \langle v_{32}\rangle\neq 0, 
\]
whereas $\mf B_n^{(3)}=0$ by formula  \eqref{B_der}.
Consequently $\mf B(G_2)$ can not be embedded into  $\mf B_n$ and $\mf X_{\mr{bir}}(\p^2)\cong\mf X_{\mr{bir}}(\mb F_n)$ can not contain~$G_2$.
\end{proof}

\section{End of the proof of Theorem~\ref{reg} and consequences}

In this section we complete the proof of Theorem~\ref{reg}, which shows that each finite-dimensional birationally integrable Lie algebra $\mf g$ of rational vector fields on a projective surface $S$ is regularizable. That is, $\mf g$ is birationally equivalent to the Lie algebra of a Lie subgroup of the automorphism group $\operatorname{Aut}(S')$ of a certain minimal surface $S'$. As an application of this theorem and other results of the article, we give a characterization of those finite-dimensional algebras of rational vector fields that are maximal among the birationally integrable ones.

\begin{proof}[End of the proof of Theorem~\ref{reg}.]
Let $\mf g$ be a finite-dimensional Lie subalgebra of $\mf X_{\rm{rat}}(S)$ generated by birationally integrable vector fields. Then $\mf g$ is contained in $\mf X_{\rm{bir}}(S)$ (cf. Theorem~\ref{XbirM}). According to Levi decomposition theorem, the exact sequence
\[
{0\to {\rm {rad}}({\mathfrak {g}})\to {\mathfrak{g}}\to {\mathfrak{g}}/{\rm {rad}}({\mathfrak{g}})\to 0},
\]
is split and the quotient $\mf s={{\mathfrak{g}}/{\rm {rad}}({\mathfrak {g}})}$ is semisimple. Since $ {\rm {rad}}({\mathfrak {g}})$ is solvable, it is regularizable (cf. Corollary~\ref{solv}), and $\mf s$, being semisimple, is regularizable as well (cf. Theorems~\ref{sl2} and~\ref{r>1}). We conclude that the semidirect sum $\mf g = {\rm {rad}}({\mathfrak {g}}) \roplus \mf s$ is also regularizable (cf. Theorem~\ref{semi}). 
\end{proof}

As a direct consequence of the previous theorem we obtain (cf. Remark~\ref{Lie_gr_alg}):
 
\begin{cor}\label{algebraic_subgroup}
Let $S$ be a projective surface and let $\mf g$ be a finite dimensional  birationally integrable Lie subalgebra of $\mf X_{\mr{rat}}(S)$. Then $\mf g$ is the Lie algebra of a subgroup of $\mr{Bir}(S)$ of Lie-type. That is, there is an algebraic subgroup $\bar G$ of $\mr{Bir}(S)$ and a Lie subgroup $G$ of $\bar G$ such that $\mf g= \operatorname{Lie}(G)$
\end{cor}

\begin{obs}
The Lie subgroup $G$ of $\operatorname{Bir}(\mb F_n)$ associated to the one-dimensional Lie algebra $\mf g=\langle\partial_x+y\partial_y\rangle\subset\mf{aut}(\mb F_n)$ is not Zariski closed. Thus, $\mf g$ provides an example for which $G\neq \bar G$ in the above theorem. In this case $\bar G$ is the two-dimensional algebraic group associated to $\langle \partial_x,y\partial_y\rangle$.
\end{obs}

If the Lie algebra $\mf g$ is semisimple, the situation described in the previous observation cannot occur. Furthermore, precise information can be given:

\begin{cor}\label{alg-group} Let $S$ be a projective surface and let $\mf g$ be a birationally integrable semisimple subalgebra of $\mf X_{\mr{rat}}(S)$. Then $\mf g$ is the Lie algebra of an algebraic subgroup $G$ of $\mr{Bir}(S)$. Moreover,  $G$ is birationally equivalent to one of the following groups.
\begin{enumerate}[ (1)]
\item $\mr{Aut}(\p^2)\cong\mr{PSL}_3(\C)$, 
\item $\mr{Aut}_0(\p^1\times\p^1)\cong\mr{PSL}_2(\C)\times\mr{PSL}_2(\C)$,
\item $\Big\{\big(\frac{ax+b}{cx+d},\frac{y}{(cx+d)^n}\big)\mid ad-bc=1\Big\}\subset\mr{Aut}(\mb F_n)$ for some $n\ge 0$,
\item $\Big\{\left(\frac{ax+b}{cx+d},\frac{ay+b}{cy+d}\right)\mid ad-bc=1\Big\}\subset\mr{Aut}_0(\p^1\times\p^1)$,
\item $\mr{SO}_3(\C)\subset 
 \mr{Aut}(\p^2)$,
\item $\Big\{\big(x,\frac{ay+b}{cy+d}\big)\mid ad-bc=1\Big\}\subset\mr{Aut}(C\times\p^1)$,	
\end{enumerate}
the four last groups being all isomorphic to $\mr{Aut}(\p^1)\cong\mr{PSL}_2(\C)$. 
\end{cor}

\begin{proof}[Proof of Corollary~\ref{alg-group}]
Theorems~\ref{sl2} and~\ref{r>1} provide holomorphic normal forms for semisimple birationally integrable Lie algebras of rational vector fields up to birational equivalence. In each case it is easy to identify the  algebraic group having the corresponding Lie algebra.
\end{proof}

\bigskip

\begin{defin}\label{def-max}
A finite-dimensional birationally integrable Lie subalgebra
$\mf g\subset \mf X_{\mr{rat}}(S)$ is said to be {\rm maximal} if there exists no finite-dimensional
birationally integrable Lie algebra $\mf g'\subset \mf X_{\mr{rat}}(S)$ that strictly contains $\mf g$.
\end{defin}

Due to Theorem~\ref{reg}, a Lie algebra $\mf g$ that is maximal in the sense of the above definition is necessarily birationally conjugated to $\mf{aut}(S)$, where $S$ is a minimal surface.  This shows in particular that $\mf g$ is the Lie algebra of an algebraic subgroup of $\operatorname{Bir}(S)$.

\begin{obs}	Let $C$ be an arbitrary curve. The infinite-dimensional Lie algebra	$\mf a_\infty(C)=\C y\partial_y\oplus\C(C)\partial_y$ is contained in $\mf X_{\mr{bir}}(C\times\p^1)$ but, being infinite-dimensional, it is not regularizable.	
Any finite dimensional Lie subalgebra $\mf g$ of $\mf a_{\infty}(C)$ is regularizable. In fact, there is an effective divisor $D$ of $C$ such that $\mf g\subset\C y\partial_y\oplus H^0(C,\mc O_C(D))\partial_y$ 
because every finite-dimensional vector subspace of $\C(C)$ is contained in some $H^0(C,\mc O_C(D))$.
If $C=\p^1$ then we can take $D=n(\infty)$ and $\C y\partial_y\oplus H^0(C,\mc O_C(D))\partial_y$ is contained in the maximal birationally integrable Lie algebra $\mf{aut}(\mb F_n)$.
If $C\neq\p^1$, according to Theorems~\ref{Xhol-ruled_1},~\ref{Xhol-ruled_2} and to Theorem~\ref{maximal} below, we have that
$\C y\partial_y\oplus H^0(C,\mc O_C(D))\partial_y=\mf{aut}(\p(\mc O_C(D)\oplus\mc O_C))$ is not contained in a maximal birationally integrable Lie algebra.
\end{obs}

\begin{teo}\label{maximal}
Let $S$ be a minimal projective surface. Then the following assertions are equivalent:
\begin{enumerate}[ (a)]
\item The Lie algebra $\mf{aut}(S)$ is not maximal.
\item $\mf{aut}(S)$ is tangent to a rational fibration and
$S$ is not biholomorphic to a product $C\times\p^1$.
\item $S$ is biholomorphic either to $\p(\mc O_C(np)\oplus\mc O_C)$, for some point $p$ belonging to an elliptic curve $C$ and some $n>0$, or to a non-trivial ruled surface over a curve of higher genus. 
\end{enumerate}
In that case $\operatorname{kod}(S)=-\infty$ and  $\mf{aut}(S)\subset\C y\partial_y\oplus\C(C)\partial_y\subset\mf X_{\mr{bir}}(S)$.
\end{teo}

\begin{proof}[Proof of Theorem~\ref{maximal}]
If has non-negative Kodaira dimension 
then  $\operatorname{Bir}(S)=\operatorname{Aut}(S)$ and $\mf{aut}(S)$ is maximal. Clearly, conditions (b) and (c) are not verified either. Therefore, in what follows we will assume that $S$ is a minimal surface with $\operatorname{kod}(S)=-\infty$.

The equivalence $(b)\Longleftrightarrow(c)$ is consequence of Theorems~\ref{Xhol-ruled_1} and \ref{Xhol-ruled_2}. The implication $(c)\Longrightarrow(a)$ is equivalent to say that, if $S$ is one of the ruled surfaces $\p E\to C$ considered in case (4) of Theorem~\ref{Xhol-ruled_1} or in cases (1), (2) and (3) of Theorem~\ref{Xhol-ruled_2}, then $\mf{aut}(\p E)$ is not maximal. 
This follows from the facts that $(\det E)^{-1}\otimes L^2\cong\mc O_C(D)$, for some divisor $D$ of $C$, and that $H^0(C,\mc O_C(D))\equiv\{f\in\C(C)\setminus\{0\}: (f)+D\ge 0\}\cup\{0\}$. Indeed, we have $H^0(C,\mc O_C(D))\hookrightarrow H^0(C,\mc O_C(D+np))$ for any $n\ge 1$ and, in each of the cases, we obtain a non-stationary sequence of birationally integrable Lie algebras
\[
\mf{aut}(\p E)\hookrightarrow\mf{aut}(\p(\mc O_C(D+p)\oplus \mc O_C))\hookrightarrow\mf{aut}(\p(\mc O_C(D+2p)\oplus \mc O_C))\hookrightarrow\cdots
\]
showing that $\mf{aut}(\p E)$ is not maximal.

In order to prove $(a)\Longrightarrow(c)$ we argue by contradiction supposing that $S$ satisfies~(a) but not~(c).
From (a) we deduce that there is a birationally integrable Lie
subalgebra $\mf g'\subset\mf X_{\rm rat}(S)$ such that $\mf{aut}(S) \subsetneq \mf g'$. Due to Theorem~\ref{reg}, there exists a minimal surface $S'$ and a birational map $f:S'\to S$ such that $f^*\mf{aut}(S)\subset f^*\mf g'\subset\mf{aut}(S')$. This implies that $\mf{aut}(S')$ strictly contains a subalgebra isomorphic to  $\mf{aut}(S)$ and that, in particular, $\dim\mf{aut}(S)<\dim\mf{aut}(S')$.
Since we are assuming that $S$ is a minimal surface with Kodaira dimension $-\infty$ and that does not satisfy (c), we deduce that $S$ is biholomorphic to either $\p^2$, $\mb F_n$ with $n\neq 1$, $\p A_1$, $\p A_0$, $\p E_\tau$ or $C\times\p^1$ with $C$ an arbitrary curve. 
We analyze each of these cases separately:
\begin{itemize}
\item Assume that $S=\p^2$ or $S=\mb F_0$. In these cases $\mf{aut}(S')$ would strictly contain a Lie subalgebra isomorphic to $\mf{sl}_3(\C)$ or $\mf{sl}_2(\C)\oplus\mf{sl}_2(\C)$, but such a surface $S'$ does not exist. 
\item Assume that  $S=\mb F_n$ with $n\geq 2$. Since $f^*$ is injective and maps the Borel subalgebra $\mf B_n$ of $\mb F_n$ into a Borel subalgebra of $S'$, we deduce that  $S'$ is different from $\p^2$ and $\mb F_0$ (because $\dim\mf B_n=n + 4\geq 6$ and the dimensions of the Borel subalgebras of $\mb F_0$ and $\p^2$ are $4$ and $5$ respectively), so $S'=\mb F_m$ with $m\geq 2$. Furthermore, as the dimension of $\mf{aut}(\mb F_m)$ increases with $m$ (in fact $\dim\mf{aut}(\mb F_m)=m+5$) we deduce that $m>n$. Composing $f$ with an automorphism of $\mb F_m$ we can assume that $f^\ast\mf B_n\subset\mf B_m$ (cf. Lemma~\ref{solv_S}). It follows that $f^\ast\mf B_n^{(2)}\subset\mf B_m^{(2)}=\C_m[x]\partial_y$, which implies that 
$f$ must preserve the corresponding fibrations.
From the description of $\mf{aut}(\mb F_n)$ given in (\ref{XFn})  we deduce that, up to composing $f$ with a suitable automorphism of $\mb F_n$, we can write $f(x,y)=\Big(x,\frac{a(x)y+b(x)}{c(x)y+d(x)}\Big)$, with $ad-bc=1$. Imposing that $f^*\big(x\partial_x+\frac{n}{2}y\partial_y\big)$ and $f^*\big(x\partial_x+nxy\partial_y\big)$ belong to $\mf{aut}(\mb F_m)$, and in particular the term in $y^2\partial_y$ does not appear (cf. Remark~\ref{y-quadrado}), we deduce that 
\[
-x\, c(x)\, a'(x) + a(x)\, \big(x\, c'(x) + n\, c(x)\big)=0
\]
and
\[
-2\,x\, c(x)\, a'(x) + a(x)\, \big(2\,x\, c'(x) + n\, c(x)\big)=0,
\]
which implies that $ac=0$. Imposing also that $f^*\partial_x\in\mf{aut}(\mb F_m)$ we obtain that 
if $a=0$ (resp. $c=0$) then $b'/b$ (resp. $a'/a$) is constant, deducing that $b$ (resp. $a$) is constant. Since $f^*(x\partial_x+\frac{n}{2}y\partial_y)\in\mf{aut}(\mb F_m)$ we conclude that  $m=-n$ (resp. $m=n$). As we are assuming that $m>n>1$ both cases give a contradiction.
\item Assume that $S=\p A_1$. In this case $S'$ is another ruled surface over an elliptic curve with $\dim\mf{aut}(S')>1$, and the generator $X$ of the $1$-dimensional Lie algebra $\mf{aut}(S)$ is birationally equivalent to a holomorphic vector field $X'$ on $S'$.
Since the birational map $f:S'\to S$ preserves the rational fibrations, $X'$ is transverse to the rational fibration and consequently it is a suspension. Then Theorem~\ref{pd} would imply
that $S$ and $S'$ are biholomorphic, contradiction.
\item Assume that $S=\p A_0$ or $S=\p E_\tau$. Then $\mf{aut}(S)$ is equal to $\mf b_{\rm o}$ or $\mf b_\tau$ with $\tau\in C\setminus\{\rm{o}\}$ respectively. In both cases $\mf{aut}(S)$ is a $2$-dimensional Abelian Lie algebra. We deduce that $S'$ is a ruled surface over $C$ with $\dim\mf{aut}(S')\geq 3$, which implies $S'=C\times\p^1$. But then $\mf b_\tau$ would be birationally conjugated to $\mf a_{0\nu}^1$ with $\nu\in\{0,1\}$, contradicting Proposition~\ref{Lab}.
\item If $S=C\times\p^1$ with $g(C)=1$ then, according to  Theorem~\ref{Xhol-ruled_1}, $\mf{aut}(S)=\C\partial_x\oplus\C_2[y]\partial_y\neq\mf{aut}^v(S)$ has maximal dimension between the ruled surfaces $S'$ over elliptic curves satisfying that $\mf{aut}(S')\neq\mf{aut}^v(S')$. Therefore, there is no surface $S'$ with the required properties.
\item If $S=C\times\p^1$ with $g(C)>1$ then, taking into account that $\dim\mf{aut}(S)<\dim\mf{aut}(S')$ and according to Theorem~\ref{Xhol-ruled_2}, we deduce that $\mf{aut}(S')\subset\C y\partial_y\oplus\C(C)\partial_y$.
If there were a  birational map $f:S'\to C\times\p^1$ such that $f^*\C_2[y]\partial_y=f^*\mf{aut}(C\times\p^1)\subset\mf{aut}(S')\subset\C y\partial_y\oplus\C(C)\partial_y$, its restriction $f_p$ to a generic fibre $x=p$ of the rational fibration would satisfy $f_p^*\C_2[y]\partial_y\subset\C_1[y]\partial_y$ which is not possible.
\end{itemize}
\end{proof}

\begin{obs}
A  maximal finite-dimensional birationally integrable subalgebra of $\mf{X}_{\rm{rat}}(S)$ is the Lie algebra of an algebraic group which is maximal among the algebraic subgroups of $\operatorname{Bir}(S)$. Using this observation and the above theorem, one recovers a result of P.~Fong, \cite[Theorem B]{Fong}, in the case where $S$ is an algebraic surface over the complex numbers. 
\end{obs}

\bibliographystyle{plain}

\end{document}